\documentclass[12pt]{amsart}

\usepackage{amsfonts,amsmath,amssymb,color,amscd,amsthm}
\usepackage[T1]{fontenc} 
\usepackage[francais,english]{babel} 
\usepackage{typearea}
\usepackage{graphics} 
\usepackage{caption}

\usepackage[all,arc,curve,color,frame,graph,matrix,cmtip]{xy}

\usepackage{tikz}
\usetikzlibrary{matrix}
\usepackage{tikz-cd}

\usepackage{scalerel}[2016/12/29]

\usepackage[shortlabels]{enumitem}
\setlist[enumerate]{label=\rm{(\arabic*)}}
\setlist[enumerate,2]{label=\rm({\it\roman*})}
\setlist[itemize]{label=\raisebox{0.25ex}{\tiny$\bullet$}}

\usepackage[backref, colorlinks, linktocpage, citecolor = blue, linkcolor = blue]{hyperref}

\newtheorem{theorem}{Theorem}
\newtheorem{lemma}{Lemma}[subsection]
\newtheorem{corollary}[lemma]{Corollary}
\newtheorem{proposition}[lemma]{Proposition}
\newtheorem{question}[lemma]{Question}
\theoremstyle{definition}
\newtheorem{definition}[lemma]{Definition}

\theoremstyle{remark}
\newtheorem{remark}[lemma]{Remark}
\newtheorem{algorithm}[lemma]{Algorithm}
\newtheorem{example}[lemma]{Example}

\newcommand\A{{\mathbb A}}
\newcommand\C{{\mathbb C}}

\newcommand\F{{\mathbb F}}
\newcommand\N{{\mathbb N}}

\newcommand\R{{\mathbb R}}
\newcommand\Z{{\mathbb Z}}

\newcommand\ZZ{{\mathcal{Z}}}

\newcommand\p{{\mathbb P}}
\renewcommand\k{\mathrm{k}}

\newcommand\Weyl{{\mathrm W}_{\infty}}
\newcommand\car{{\mathrm{char}}}

\newcommand{\trace}{\mathrm{trace}}
\newcommand{\pr}{\mathrm{pr}}

\newcommand\into{\hookrightarrow}
\newcommand{\predecessor}{\raisebox{0.7mm}{\makebox{\scaleobj{0.5}{\to}}}}
\newcommand{\predecessorLow}{\scaleobj{0.5}{\to}}

\renewcommand{\b}{{\rm Bir}({\mathbb P}^2)}
\newcommand{\bb}{{\mathfrak B}{\mathfrak i}{\mathfrak r}({\mathbb P}^2)}
\renewcommand{\r}{{\rm Rat}({\mathbb P}^2)}
\newcommand{\rr}{{\mathfrak R}{\mathfrak a}{\mathfrak t}({\mathbb P}^2)}
\newcommand{\Jonq}{{\rm Jonq}}
\newcommand{\JJonq}{{\mathfrak J}{\mathfrak o}{\mathfrak n}{\mathfrak q}}
\newcommand{\LL}{{\mathfrak L}}
\newcommand{\PP}{{\mathfrak P}}
\newcommand{\Comp}{{\mathfrak C}{\mathfrak o}{\mathfrak m}{\mathfrak p}}

\newcommand\lb[1]{\text{\it #1}}
\newcommand\lbd[2]{\text{\it #1\kern-1.1pt.\kern-1.1pt#2}}

\DeclareMathOperator{\Aut}{Aut}
\DeclareMathOperator{\GL}{GL}
\DeclareMathOperator{\PGL}{PGL}
\DeclareMathOperator{\SL}{SL}

\DeclareMathOperator{\Aff}{Aff}

\DeclareMathOperator{\M}{Mat}

\DeclareMathOperator{\B}{\mathcal{B}}
\DeclareMathOperator{\Base}{\mathrm{Base}}
\DeclareMathOperator{\Pic}{\mathrm{Pic}}
\DeclareMathOperator{\Sym}{Sym}
\DeclareMathOperator{\Bir}{Bir}

\DeclareMathOperator{\J}{J}

\DeclareMathOperator{\id}{id}

\DeclareMathOperator{\lgth}{lgth}
\newcommand{\dlgth}{\mathfrak{d}_{\lgth}}

\DeclareMathOperator{\comult}{comult}

\DeclareMathOperator{\triple}{{\mathcal T}}

\DeclareMathOperator{\Spec}{Spec}

\title[Length in the Cremona group]{Length in the Cremona group}

\author{J\'er\'emy Blanc}
\address{J\'er\'emy Blanc, Universit\"{a}t Basel, Spiegelgasse $1$, CH-$4051$ Basel, Switzerland.}
\email{jeremy.blanc@unibas.ch}

\author{Jean-Philippe Furter}
\address{Jean-Philippe Furter, Dpt. of Math., Univ. of La Rochelle, av. Cr\'epeau, 17000 La Rochelle, France}
\email{jpfurter@univ-lr.fr}

\date{\today}

\thanks{The first author acknowledges support by the Swiss National Science Foundation Grant  ``Birational Geometry'' PP00P2\_153026. The second author acknowledges support by the  R\'egion Poitou-Charentes \includegraphics[scale=0.05]{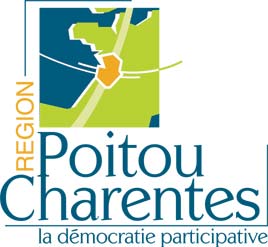} Both authors were also supported by the ANR Grant ``BirPol'' ANR-11-JS01-004-01.}

\begin{document}

\begin{abstract}
The Cremona group is the group of birational transformations of the plane.
A birational transformation for which there exists a pencil of lines which is sent onto another pencil of lines is called a Jonqui\`eres transformation. By the famous Noether-Castelnuovo theorem, every birational transformation $f$ is a product of Jonqui\`eres transformations.
The minimal number of factors in such a product
will be called the length, and written $\mathrm{lgth}(f)$. Even if this length is rather unpredictable, we provide an explicit algorithm to compute it, which only depends on the multiplicities of the linear system of $f$.

As an application of this computation, we give a few properties of the dynamical length of $f$ defined as the limit of the sequence $n \mapsto \mathrm{lgth} (f^n) / n$. It follows for example that an element of the Cremona group is distorted if and only if it is algebraic.
The computation of the length may also be applied to the so called Wright complex associated with the Cremona group: This has been done recently by Lonjou.
Moreover, we show that the restriction of the length to the automorphism group of the affine plane is the classical length of this latter group (the length coming from its amalgamated structure).
In another direction, we compute the lengths and dynamical lengths of all monomial transformations, and of some Halphen transformations. Finally, we show that the length is a lower semicontinuous map on the Cremona group endowed with its Zariski topology.
\end{abstract}  
\maketitle

\tableofcontents

\section{Introduction}
\subsection{The length of elements of the Cremona group}
Let us fix an algebraically closed field $\k$.
The Cremona group over $\k$, often written $\mathrm{Cr}_2(\k)$, is the group $\Bir(\p^2)$ of $\k$-birational transformations of the
projective plane $\p^2$.
Such transformations can be written in the form
\[[x:y:z]\dasharrow [u_0(x,y,z):u_1(x,y,z):u_2(x,y,z)]\]
where $u_0,u_1,u_2\in \k[x,y,z]$
are homogeneous polynomials of the same degree, and this degree is the degree of the map, if the polynomials have no common factor. The Cremona transformations of degree $1$ are the automorphisms of $\p^2$,
i.e.~the elements of the group $\Aut(\p^2)=\PGL_3(\k)$. The Cremona transformations of degree $2$ are called quadratic.

The group $\Bir(\p^2)$ is generated by the automorphism group  $\Aut (\p^2)$
and by the single involution $\sigma\colon [x:y:z]\dasharrow [yz:xz:xy]$, called the \emph{standard quadratic transformation}.
In Castelnovo's
proof of this result \cite{Cas}, an element of $\Bir(\p^2)$ is first decomposed into a product of \emph{Jonqui\`eres} elements (also called \emph{Jonqui\`eres transformations}).

These latter  maps are defined as the birational maps $f$ for which there exist points $p,q \in \p^2$ such that $f$ sends the pencil of lines passing through $p$ to the pencil of lines passing through $q$.
In this text, the group of Jonqui\`eres transformations preserving the pencil of lines passing through a given point $p \in \p^2$ is denoted by $\Jonq_p\subseteq \Bir(\p^2)$. The set of all Jonqui\`eres transformations is then equal to
 \[ \Jonq = \bigcup_{p \, \in \, \p^2} \Aut(\p^2)\Jonq_p \Aut(\p^2)= \bigcup_{p \, \in \, \p^2} \Aut(\p^2)\Jonq_p= \Aut(\p^2)\Jonq_{p_0}\Aut(\p^2), \]
 for any fixed point $p_0\in \p^2$. The above equalities follow from the equality $\alpha  \circ \Jonq_p \circ \,  \alpha^{-1} = \Jonq_{\alpha (p)}$, which holds for each $\alpha \in \Aut (\p^2)$. 
 
Nowadays, one can also see the proof of
Castelnuovo's
theorem by using the Sarkisov program (see \cite{Cor}), and the number of Jonqui\`eres transformations needed corresponds to the number of links involved, which do not preserve a fibration. 
The proof that every Jonqui\`eres transformation is a product of linear maps and $\sigma$ is then an easy exercise (see for example~\cite[$\S8.4$]{Alberich}).

In order to study the complexity of an element of $\Bir(\p^2)$, according to the above decomposition, the following definition seems natural:

\begin{definition}
For each Cremona transformation $f \in \Bir(\p^2)$ we define its \emph{length} $\lgth(f)$
as follows. If $f\in \Aut(\p^2)$, we set $\lgth(f)=0$.
Otherwise we set $\lgth (f) =n$, where $n$  is the least positive integer for which $f$ admits a decomposition
\[ f=\varphi_n \circ \cdots \circ \varphi_1, \quad \forall \, i, \, \varphi_i \in \Jonq. \] 
\end{definition}

With this definition, note that we have $\lgth (f^{-1} ) = \lgth (f)$.

For any fixed point $p_0 \in \p^2$ the group $\Bir(\p^2)$ is generated by its two subgroups $\Aut(\p^2)$ and $\Jonq _{p_0}$.
The length of $f$ might also be seen as the least non-negative
integer $n$ for which $f$ admits a decomposition
\[ f=\alpha_n \circ \varphi_n \circ  \cdots  \circ \alpha_1 \circ \varphi_1 \circ  \alpha_0,
\quad \forall \,i, \, \alpha_i \in \Aut ( \p^2) \text{ and } \forall \, j, \, \varphi_j  \in \Jonq_{p_0}. \]   
In particular, this least integer $n$
does not depend on
$p_0$. All this follows from the equality $\Jonq = \Aut(\p^2) \Jonq_{p_0} \Aut(\p^2)$.

This notion of length is similar to the case of automorphisms of the affine plane. Taking a linear embedding $\A^2 \hookrightarrow \p^2$, the classical Jung-Van der Kulk theorem says that $\Aut(\A^2)$ is generated by $\Aff_2 :=\Aut ( \p^2) \cap \Aut(\A^2)$ and $\Jonq_{p,\A^2}:=\Jonq_p \cap \Aut(\A^2)$ for any $p\in \p^2\setminus \A^2$ \cite{Jung,VdK53,LamJung}.
Furthermore, there
are no relations
except the trivial ones, i.e.~the group $\Aut(\A^2)$ is the amalgamated product of $\Aff_2$ and $\Jonq_{p,\A^2}$ over their intersection.

The length in $\Aut(\A^2)$ is then easy to compute, by writing an element in a reduced form
(i.e.~as a product of elements of $\Aff_2$ and $\Jonq_{p,\A^2}$ where two consecutive elements do not belong to the same group). 
It has moreover natural properties. For example, it
is lower semicontinuous for the Zariski topology on $\Aut(\A^2)$, as shown
in \cite{Furterlength}  when $\car(\k)=0$
(in fact this result also holds in positive characteristic by Theorem~\ref{Continuitylength} and Proposition~\ref{Prop:LengthThesameinA2BirP2} below).

The case of $\Bir(\p^2)$ is more complicated, as $\Bir(\p^2)$ is not
the amalgamated product of $\Aut (\p^2)$ and $\Jonq_{p_0}$. There is only one relation, of very small length \cite{BlaRel}, which makes the group $\Bir(\p^2)$ more complicated than the group $\Aut(\A^2)$ (see also \cite{GizRel,IskRel} for other presentations with generators and relations of $\Bir(\p^2)$).
In particular, there exist  elements of $\Bir(\p^2)$ of finite order (finitely many families up to conjugacy) which are neither conjugate to an element of $\Aut (\p^2)$ nor to an element of $\Jonq_p$ \cite{Blfiniteorder}, contrary to the case of amalgamated products.
Another way to see the difference is that $\Aut(\A^2)$ acts on a tree
thanks to its amalgamated structure
\cite{SerreTrees,LamyTits}, but $\Bir(\p^2)$ only acts on a simply connected simplicial complex
of dimension two \cite{Wright}.
The group $\Bir(\p^2)$ does not act (non-trivially) on a tree because it is not a non-trivial amalgamated product \cite{Cornulier}.

Computing the length of an element $f\in \Bir(\p^2)$ is then more tricky
 than the case of $\Aut(\A^2)$ and we cannot only take a reduced decomposition (i.e.~a product $f=\varphi_n \circ \cdots \circ \varphi_1$ of Jonqui\`eres elements such that $\varphi_{i+1}\circ \varphi_{i}$ is not Jonqui\`eres for $i=1,\dots,n-1$).
The length of such reduced decompositions is unbounded (Proposition~\ref{Prop:Unboudedlengthreduceddec}).
One way to give an upper bound for the length of an element is to follow the proof of Castelnuovo and to apply successive Jonqui\`eres elements to decrease the degree
(details are given in Algorithm~\ref{Algo:Castelnuovo}).  
There is no reason a priori to expect this upper bound to be equal to the length, but in one of our main results, Corollary~\ref{Cor:CastelnuovoOptimal}, we prove that this is actually the case.
We also prove that in the (possible) case where the  Algorithm of Castelnuovo does not provide the smallest possible degree after finitely many steps, then, this smallest possible degree can be obtained by
composing on the right with a single quadratic map (Corollary~\ref{Cor:ErrorQuadratic}).

Multiplying an element $f\in \Bir(\p^2)$ with a Jonqui\`eres element $\varphi$,
we have $\lgth (f\circ \varphi) = \lgth (f) + \varepsilon$ where $\varepsilon \in \{ -1, 0 , 1 \}$.
The possibilities occur in a rather chaotic way since
there are examples where $\deg(f\circ \varphi)>\deg(f)$ but $\lgth(f\circ \varphi)=\lgth(f)-1$
(Proposition~\ref{Proposition:Bertiniexample}\ref{Bertiniexample2}).
Moreover, the number of Jonqui\`eres elements $\varphi$ such that $\lgth(f\circ \varphi)=\lgth(f)-1$ can be infinite, even up to right multiplication with an element of $\Aut(\p^2)$ (Proposition~\ref{Proposition:Bertiniexample}\ref{Bertiniexample3}).

We will however show that there is a natural algorithm that yields the length. Moreover, we will show that the length only depends on combinatorial properties of the maps, namely the multiplicities of the base-points.
Let us briefly recall the definition of the base-points and their multiplicities.

Let $f$ be an element of $\Bir(\p^2)$. Write it in the form
\[ [x:y:z]\dasharrow [ u_0(x,y,z):u_1(x,y,z):u_2(x,y,z)]\]
where $u_0,u_1,u_2$ are homogeneous of the same degree $d:= \deg f$. Then, the linear system of $f$ is the net of curves
\[ \lambda_0 u_0 + \lambda_1 u_1 + \lambda_2 u_2 = 0, \; [\lambda_0 : \lambda_1 : \lambda_2 ] \in \p^2. \]
Equivalently, it is the inverse image by $f$ of the net of lines in $\p^2$. A linear system of this form is called 
\emph{homaloidal}.
It consists of curves of degree $d$ passing through finitely many points $p_1,\dots,p_r$ (lying on $\p^2$ or infinitely near) with some multiplicities $m_1,\dots,m_r$ such that $\sum m_i=3(d-1)$ and $\sum (m_i)^2=d^2-1$ (see below, in particular Remark~\ref{Rem:BasePtsClassical} and Lemmas~\ref{Lemm:BirP2ZP2} and \ref{Lemm:EasyWeyl}). The points $p_i$ are called the base-points of $f$ and the set $\{ p_1, \ldots,p_r  \}$ of base-points is denoted $\Base (f)$.

In particular, maps of degree $1$ have no base-points and maps of degree $2$ have three base-points of multiplicity~$1$. We say that $(m_1,\dots,m_r)$ is the \emph{homaloidal type} of $f$. It
is a finite sequence up to permutation, or equivalently a
multiset (a multiset, unlike a set, allows for multiple instances for each of its elements).
We often write $(d;m_1,\dots,m_r)$ to see the degree, even though the degree of course is uniquely determined by the multiplicities. 
In our text and by definition, each homaloidal type will be of this form, i.e.~will be the homaloidal type of at least one birational transformation of $\p^2$ (in \cite{Alberich,BlaCal} such homaloidal types are called \emph{proper homaloidal types}).
We also define the \emph{comultiplicity} of $f$ to be $\deg (f) - \max_i m_i$.
This notion is sometimes used in the literature, for instance in the proof of the Noether-Castelnuovo
theorem given by Alexander \cite{Alexander}. One can observe that $\comult(f)=1$ if and only if $f$ is a  Jonqui\`eres element (follows from Lemma~\ref{Lemm:JoCremWinfty} and Definition~\ref{definition: Jonquieres element}). To state our main result, we use the following notion:

\begin{definition} \label{DefiPredeBirP2}
Let $f\in \Bir(\p^2)$. A \emph{predecessor of $f$} is an element
of minimal degree among the elements of the form $f \circ \varphi$ where $\varphi$ is a Jonqui\`eres transformation.
\end{definition}

A precise description of the predecessors of an element of $\Bir(\p^2)$ is algorithmic and not very difficult to obtain. In particular, the following holds:

\begin{lemma}   \label{Lem:Pred}
Let $f\in \Bir(\p^2)$. 
\begin{enumerate}
\item\label{Lem:Pred1}
The homaloidal type of a predecessor of~$f$ is uniquely determined by 
the homaloidal type of~$f$.
\item\label{Lem:Pred2}
There are infinitely many predecessors of $f$, but only finitely many classes up to right composition with an element of $\Aut(\p^2)$. 
\item\label{Lem:Pred3}
If $\varphi$ is a Jonqui\`eres transformation such that $f \circ \varphi$ is a predecessor of $f$, then $\Base(\varphi^{-1})\subseteq \Base(f)$.
\end{enumerate}
\end{lemma}

\begin{remark}
Lemma~\ref{Lem:Pred}\ref{Lem:Pred2} asserts that any Cremona transformation $f \in \Bir (\p^2)$ admits finitely many predecessors up to right multiplication by an element of $\Aut(\p^2)$. However, we will see in Lemma~\ref{lemm:UnboundedNumberPred} that this number is not uniformly bounded on $\Bir(\p^2)$, even if it 
is bounded by an integer only depending on $\deg(f)$ (Remark~\ref{Rem:BoundPredDegree}).
\end{remark}

Computing a sequence of predecessors (which is algorithmic, as said before, and whose homaloidal types are uniquely determined by the one of the map we start with) yields then a finite algorithm to compute the length of any element of $\Bir(\p^2)$, as our main theorem states: 

\begin{theorem}\label{TheMainTheorem}
Let
$f_0$ be an element of $\Bir(\p^2)$,
let $n\ge 1$ be an integer, and let $(f_i)_{i\in \N}$ be a sequence of elements of $\Bir(\p^2)$ such that $f_i$ is a predecessor of $f_{i-1}$ for each $i\ge 1$. For all Jonqui\`eres elements
$\varphi_1, \ldots, \varphi_n$ of  $\Bir(\p^2)$,
the element $g_n=f \circ \varphi_1\circ \dots \circ \varphi_n$
satisfies
\begin{enumerate}
\item\label{Degfngn}
$\deg(f_n)\le \deg(g_n)$;
\item\label{Comultfngn}
$\comult(f_n)\le \comult(g_n)$;
\item\label{AnBnSymp2}
If $\deg(f_n)= \deg(g_n)$, then $f_n$ and $g_n$ have the same homaloidal type.
\end{enumerate}
In particular, $\lgth(f)=\min \{n\mid \deg(f_n)=1\}$.
\end{theorem}

Thus  the length of all maps of some given degree can be easily computed (see $\S\ref{SubSec:LengthSmallDegree}$ for
tables up to degree $12$).

Another consequence of Theorem~\ref{TheMainTheorem} is that the length of an element of $\Aut(\A^2)$, viewed as an element of $\Bir(\p^2)$, is the same as the classical length given by
the amalgamated product structure
(Proposition~\ref{Prop:LengthThesameinA2BirP2}).

In the general case, the length in $\Bir(\p^2)$ can be interpreted in terms of the natural distance defined on the already mentioned \emph{Wright complex} \cite{Wright} or simply on its associated graph.
We now
recall the construction of this graph (the Wright complex being then the two-dimensional simplicial complex obtained from this graph by adding a two-dimensional face to each triangle).
We fix two distinct points $p,q\in \p^2$ and look at the three subgroups
$G_0,G_1,G_2$ of $\Bir(\p^2)$
given by
\[G_0=\Aut(\p^2),\ G_1=\Jonq_p,\ G_2=\pi^{-1} \Aut(\p^1\times \p^1) \pi\]
where $\pi\colon \p^2\dasharrow \p^1 \times \p^1$ is the birational map induced by the projections away from $p$ and $q$ respectively. We then consider as vertices the set $\mathcal{A}_0\cup \mathcal{A}_1\cup \mathcal{A}_2$, where $\mathcal{A}_i= \{ G_i f \mid f\in \Bir(\p^2)\}$ is the set of right cosets modulo $G_i$, for $i=0,1,2$. There is a triangle between three elements of the three sets $\mathcal{A}_i$ if and only if these are of the form $G_0f$, $G_1f$ and $G_2f$, for some $f\in \Bir(\p^2)$. As proven in \cite{Wright}, the associated simplicial complex (i.e.~the Wright complex) is simply connected, which corresponds to
saying that $\Bir (\p^2)$ is the amalgamated product of the three groups $G_i$ along their pairwise intersections. The set $\mathcal{A}_0$ corresponds to the set of homaloidal linear systems and the distance between  $G_0 \varphi $ and $G_0= G_0 \id$
is given by $2\lgth(\varphi)$ for each $\varphi\in \Bir(\p^2)$
(see Lemma \ref{Lem:DistanceWright} below).

Hence, Theorem~\ref{TheMainTheorem} provides a way to compute the distance in this graph. In particular, the graph of Wright is of unbounded length (follows for instance from Lemma~\ref{Lemma:AutA2}), a fact that does not follow directly from its definition.
Another application of Theorem~\ref{TheMainTheorem},
done by Lonjou, is that the graph of Wright is not hyperbolic in the sense of Gromov \cite{Lonjou}.

In $\S\ref{SubSec:LengthMonomial}$, we compute the length and dynamical length of all monomial elements of $\Bir(\p^2)$. We relate these notions with some decompositions of elements in $\GL_2(\Z)$ and also with continued fractions.

\subsection{Dynamical length}
Since $\lgth (f \circ g) \leq  \lgth (f) + \lgth (g)$ for all $f,g \in \Bir (\p^2)$,
the sequence $n \mapsto \lgth (f^n)$ is subadditive so that the sequence $n \mapsto \frac{\lgth (f^n)}{n}$ admits a finite
limit when $n$ goes to infinity. This allows the following definition:

\begin{definition} \label{defintion: dynamical length}
For each $f\in \Bir(\p^2)$, the {\it dynamical length} is defined as
\[ \dlgth(f) := \lim_{n \to \infty} \frac{\lgth (f^n)}{n} \in \R_+.\] 
\end{definition}
Note that $\dlgth(f)$ is invariant under conjugation (contrary to the length), and satisfies  $0\le \dlgth(f)\le \lgth(f)$. It is not very easy to compute $\dlgth(f)$ in general, but we will do it precisely for all monomial elements of $\Bir(\p^2)$, and relate this to continued fractions and decompositions in
$\GL_2(\Z)$ (Section~\ref{SubSec:LengthMonomial}). We will also show that
\[\frac{1}{2}\Z_{ \geq 0} \cup \frac{1}{3}\Z_{ \geq 0}\subseteq \dlgth(\Bir(\p^2))=\{\dlgth(f)\mid f\in \Bir(\p^2)\}\] 
(Corollary~\ref{Cor:1213LengthSpectrum}),
but we do not
have any example of  a Cremona transformation $f\in \Bir(\p^2)$ such that
$\dlgth(f)\notin \frac{1}{2}\Z_{ \geq 0} \cup \frac{1}{3}\Z_{ \geq 0}$. In particular, every monomial map of $\Bir(\p^2)$ has a dynamical length in $\frac{1}{2}\Z_{ \geq 0}$
(follows from Proposition~\ref{proposition: Three cases for a unimodular matrix}).

\begin{question}
What does the set $\dlgth(\Bir(\p^2))=\{\dlgth(f)\mid f\in \Bir(\p^2)\}$ look like?
\end{question}

\subsection{Distorted elements}\label{Sec:Distorted}
We begin with the two following definitions.

\begin{definition}
If $G$ is a group generated by a finite subset $F\subseteq G$, the $F$-length $\lvert g\rvert_F$ of an element $g$ of $G$ is defined as the least non-negative integer $\ell$ such that $g$ admits an expression of the form $g= f_1 \ldots f_{\ell}$ where each $f_i$ belongs to $F \cup F^{-1}$. We then say that $g$ is \emph{distorted} if $\lim\limits_{n\to \infty} \frac{\lvert g^n\rvert_F}{n}=0$
(note that the limit $\lim\limits_{n\to \infty} \frac{\lvert g^n\rvert_F}{n}$ always exists and is a real number since the sequence $n \mapsto \lvert g^n\rvert_F$ is subadditive).
This notion actually does not depend on the chosen $F$, but only on the pair $(g,G)$.

If $G$ is any group, an element $g\in G$ is said to be \emph{distorted} if it is distorted in
some finitely generated subgroup of $G$.
\end{definition}

\begin{definition}
An element $f\in \Bir(\p^2)$ is said to be \emph{algebraic} (or elliptic) if it is contained in an algebraic subgroup of $\Bir(\p^2)$, or equivalently if the sequence $n \mapsto \deg(f^n) $ is bounded (\cite[\S 2.6]{BF13}). By \cite[Proposition~2.3]{BlaDes} (see also Proposition~\ref{Prop:BlaDes} which explains why the proof also works in positive characteristic), this is also equivalent to saying that $f$ is of finite order or conjugate to an element of $\Aut(\p^2)$.
\end{definition}

An easy computation shows that every element of $\Aut(\p^2)$ is distorted in $\Bir(\p^2)$ (Lemma~\ref{Lemm:AutP2distorted}).
Consequently, every algebraic element of $\b$ is distorted.
Using the dynamical length, we will prove the converse statement  (which has also be proven in the recent preprint \cite{CantatCornulier}, with another technique):

\begin{theorem}\label{Thm:DistAutP2}
Any distorted element of $\Bir(\p^2)$ is algebraic.
\end{theorem}

A function $S \colon \Bir(\p^2)\to \R_{\ge 0}$ is said to be \emph{subadditive} if it satisfies $S (f\circ g)\le S (f)+S (g)$ for all $f,g\in \Bir(\p^2)$. For such a function, the
sequence $n \mapsto  \frac{S (f^n)}{n}$ admits a finite limit
for any $f$, and
this limit
is moreover equal to zero when $f$ is distorted.

It turns out that the three following functions are subadditive on $\b$: the length,
the number of base-points, and the logarithm of the degree. For any such $S$, the corresponding limit
of the sequence $ n \mapsto \frac{S (f^n)}{n}$
is: the dynamical length $\dlgth(f)$, the dynamical number of base-points (written $\mu(f)$ in \cite{BlaDes}), and the logarithm of the dynamical degree $\log(\lambda(f))$, where the dynamical degree is $\lambda(f)=\lim\limits_{n\to \infty} (\deg(f^n))^{1/n}$.

We will show that $f$ is algebraic if and only if $\dlgth(f) = \mu (f) = \log(\lambda(f)) = 0$, thus proving Theorem~\ref{Thm:DistAutP2}. More precisely, we can decompose elements of $\Bir(\p^2)$ into five disjoint subsets of elements (see $\S\ref{Subsection:Dynamical-length-of-regularisable-elements}$), and the situation is as in
Figure~\ref{FigureSubadditive}. In particular,
\begin{figure}[h]
\[\begin{array}{|c|c|c|c|}
\hline
f & \dlgth(f) & \mu(f) & \log(\lambda(f))\\
\hline
\text{Algebraic elements}
& 0 & 0 & 0\\
\text{Jonqui\`eres twists} & 0 & >0 & 0\\
\text{Halphen twists} & >0\ (\text{Corollary~\ref{Cor:Halphentwist}}) & 0 & 0\\
\text{Regularisable loxodromic elements} & >0\  (\text{Proposition~\ref{Prop:RegLox}}) & 0 & >0\\
\text{Non-regularisable loxodromic elements} & \text{sometimes} >0\  (\text{Lemma~\ref{Lemma:AutA2}}) & >0 & >0\\
\hline\end{array}\] \caption{Positivity of $\dlgth(f)$, $\mu(f)$, $\log(\lambda(f))$ for elements $f\in\Bir(\p^2)$} \label{FigureSubadditive}\end{figure}
Corollary~\ref{Cor:Halphentwist} is sufficient for showing that an element $f$ of $\b$ is algebraic if and only if $\dlgth(f) = \mu (f) = \log(\lambda(f)) = 0$. However, Proposition~\ref{Prop:RegLox} shows us that this is also equivalent to $\dlgth(f) = \mu (f) = 0$.

\subsection{Lower semicontinuity of the length}

Even if $\Bir(\p^2)$ is not naturally an ind-group (\cite[Theorem 1]{BF13}), following \cite{De,Se}, it admits a natural Zariski topology (see Definition~\ref{defi: Zariski topology}). We prove that the length is compatible with this topology, and thus behaves well in families:

\begin{theorem}\label{Continuitylength}
The length map $ \lgth \colon \Bir (\p^2) \to \N$, $f \mapsto \lgth (f)$ is lower semiconti\-nuous. In other words, for each integer $\ell \ge 0$, the set $\{f\in \Bir(\p^2)\mid \lgth(f)\le \ell\}$ is closed.
\end{theorem}

As explained before, this implies the same result for automorphisms of $\A^2$, already proven in \cite{Furterlength} when $\car(\k)=0$. It also shows that some degenerations of birational maps are not possible. For instance, it is not possible to have a family of birational maps with homaloidal type $(8;4,3^5,1^2)$ which degenerates to a birational map of homaloidal type $(8;4^3,2^3,3^3)$ as the first homaloidal type has length $2$ and the second has length $3$ (see $\S\ref{SubSec:LengthSmallDegree}$). See \cite{BCM15,BlaCal} for more details on this question.

The authors warmly thank the referees for their detailed reading
and helpful suggestions.

\section{Reminders}
\label{Sec:Reminders}
When we want to decompose a birational
transformation of $\p^2$, we have to study the multiplicities of the linear system at points, and also the position of the points (if one is infinitely near to another, if
they
are on the same line, \ldots). A very fruitful approach consists of looking at the (linear and faithful) action of $\Bir (\p^2)$ on the so called Picard-Manin space. Forgetting the position of the points and studying only the arithmetic part can be done by studying an \emph{infinite Weyl group} $\Weyl$, as done in \cite{BlaCan}. This group still acts on the Picard-Manin space and contains $\Bir (\p^2)$. An analogous Weyl group is used in \cite{BlaCal}, but with a  slightly different definition.

\subsection{The bubble space and the Picard-Manin space}
Let us recall the following classical notions.

\begin{definition}Let $Y$ be a smooth projective rational surface. We denote by $\B(Y)$ the \emph{bubble space} of $Y$. It is the set of points that belong, as proper or infinitely near points to $Y$. More precisely, an element of $\B(Y)$ is the equivalence class of a triple $(p,X,\pi)$, where $X$ is a smooth projective surface, $\pi\colon X\to Y$ is a birational morphism (a sequence of blow-ups) and $p\in X$. Two triples $(p,X,\pi)$ and $(p',X',\pi')$ are equivalent if $(\pi')^{-1} \circ \pi \colon X \dasharrow X'$ restricts to an isomorphism $U\to U'$, where $U\subseteq X$, $U'\subseteq X'$ are two open neighbourhoods of $p$ and $p'$, and if $p$ is sent to $p'$ by this isomorphism.
\end{definition}

\begin{definition}\label{Defi:Order}
There  is a natural order on $\B(Y)$. We say that $(p,X,\pi) \ge (p',X',\pi')$ if $(\pi')^{-1} \circ \pi\colon X\dasharrow X'$ restricts to a morphism $U\to X'$, where $U\subseteq X$ is an open subset containing $p$ and if $p$ is sent on $p'$ by this morphism. 
\end{definition}

\begin{remark}
We have an inclusion $\p^2 \hookrightarrow \B(\p^2)$, that sends a point $p\in \p^2$ onto the equivalence class of $(p,\p^2,\mathrm{id})$. We will also see elements of $\B(\p^2)$ as points, the surfaces and the morphisms being then implicit.

Every birational map $\varphi\colon \p^2\dasharrow \p^2$ has a finite number of base-points. The set of all such points is denoted $\Base(\varphi)\subseteq \B(\p^2)$. Moreover, $\varphi$ induces a bijection $\B(\p^2)\setminus \Base(\varphi)\to \B(\p^2)\setminus \Base(\varphi^{-1})$.
\end{remark}
Let us recall the following classical notions. See for example ~\cite{Alberich} and references there.
\begin{definition}\label{Defi:InfNearProximate}
Let $p,q\in \B(\p^2)$. We say that $p$ is \emph{infinitely near }$q$ if
$p > q$ (for the order defined above). We say that $p$ is in the \emph{first neighbourhood of $q$} if $p>q$ and if there is no $r\in \B(\p^2)$ with $p>r>q$. We say that a point $p\in \B(\p^2)$ is a \emph{proper point} of $\p^2$ if $p$ is minimal. This corresponds to saying that $p\in \p^2\subseteq \B(\p^2)$.
\end{definition}

\begin{definition}
Let $Y$ be a smooth projective rational surface. Its Picard-Manin space $\ZZ_Y$ is defined as the inductive limit of all the Picard groups $\Pic(X)$, where $X$ is a smooth projective rational surface and $X\to Y$ is a birational morphism. 

More precisely, an element $c\in \ZZ_Y$ corresponds to an equivalence class of triples $(C,X,\pi)$, where $X$ is a smooth projective rational surface, $\pi\colon X\to Y$ is a birational morphism and $C\in \Pic(X)$. Two triples $(C_1,X_1,\pi_1)$ and $(C_2,X_2,\pi_2)$ are identified 
if one can find another smooth projective rational surface $X_3$ together with birational morphisms $\pi_1'\colon X_3\to X_1$, $\pi_2'\colon X_3\to X_2$ such that $\pi_1 \circ \pi_1'=\pi_2 \circ \pi_2'$, and such that $(\pi_1')^*(C_1)=(\pi_2')^*(C_2)$.

The $\Z$-module $\ZZ_Y$ is endowed with an intersection form and canonical form $\omega\colon \ZZ_Y\to \Z$. The canonical form sends a triple $(C,X,\pi)$ onto $K_X\cdot C\in \Z$, where $K_X$ is the canonical class of $X$. To intersect two classes, we take representants $(C_1,X,\pi)$ and $(C_2,X,\pi)$ on the same surface by taking a common resolution and compute $C_1\cdot C_2\in \Z$.
\end{definition}
\begin{remark}
If $(C,X,\pi)$ is a triple as in the above definition and $\epsilon\colon X'\to X$ is a birational morphism, then $(\epsilon^*(C),X',\pi\circ\epsilon)$ is equivalent to $(C,X,\pi)$. Moreover, $K_X\cdot C=\epsilon^*(K_X)\cdot \epsilon^*(C)= K_{X'}\cdot \epsilon^*(C)$. 

Using this remark we obtain that the canonical form $\omega\colon \ZZ_Y\to \Z$ defined above is independent of the choice of the triple in the equivalence class of an element of $\ZZ_Y$. The intersection form is also
well-defined since
$\epsilon^*(C)\cdot \epsilon^*(D)=C\cdot D$ for all $C,D\in \Pic(X)$.
\end{remark}
\begin{definition}
Let $Y$ be a smooth projective rational surface. For each point $q\in \B(Y)$, we define an element $e_q\in \ZZ_Y$ as follows: the point $q$ is the class of $(p,X,\pi)$, and $e_q\in \ZZ_Y$ is the class of $(E_p,\hat{X},\pi\circ\epsilon)$, where $\epsilon \colon \hat{X}\to X$ is the blow-up of $p\in X$, and $E_p=\epsilon^{-1}(p)\in \Pic( \hat{X} )$
is the exceptional divisor.
\end{definition}
\begin{lemma}\label{Lem:ZZYbigsum}
Let $Y$ be a smooth projective rational surface. The group $\ZZ_{Y}$ is naturally isomorphic to \[\ZZ_Y\simeq\Pic(Y)\oplus \bigoplus\limits_{p\in \B(Y)} \Z e_{p}.\] Moreover, the restriction of the intersection form of $\ZZ_Y$ on $\Pic(Y)$ is the classical one, and we have \begin{center}
 $C\cdot e_p=0$, $e_{p}^2=-1$, $e_{p}\cdot e_{q}=0$, $\omega(C)=C\cdot K_Y$,  $\omega( e_{p})=-1$.\end{center}
 for all $p,q\in \B(Y)$, $C\in \Pic(Y)$, $p\not=q$.
\end{lemma}
\begin{proof}The map sending $C\in \Pic(Y)$ onto the class of $(C,Y,\mathrm{id})$ yields  an inclusion $\Pic(Y)\hookrightarrow \ZZ_{Y}$. By definition, the restriction of the intersection form and the canonical form of $\ZZ_Y$ on $\Pic(Y)$ are the classical ones.

If $\epsilon\colon \hat{X}\to X$ is the blow-up of a point $p\in X$, then $\Pic(\hat{X})=\epsilon^{*}\Pic(X)\oplus \Z e_{p}$, where
the exceptional divisor $e_{p}\in \hat{X}$
satisfies $e_{p}^2=-1$, $e_{p}\cdot R=0$ for each $R\in \epsilon^{*}(\Pic(X))$. Moreover, $K_{\hat{X}}=\pi^{*}(K_X)+e_{p}$. This provides the result, as every birational morphism
$ X\to Y$, 
where $X$ is a smooth projective rational surface, is a sequence of blow-ups of finitely many points of $\B(Y)$.
\end{proof}
\begin{corollary}
The group $\ZZ_{\p^2}$ is naturally isomorphic to \[\Z e_0\oplus \bigoplus\limits_{p\in \B(\p^2)} \Z e_{p},\]
where $e_0\in \Pic(\p^2)$ is the class of a line and $e_{p}$ corresponds to the exceptional divisor of $p\in \B(\p^2)$. Moreover, we have \begin{center}
$(e_0)^2=1$, $e_{p}^2=-1$, $\omega( e_{p})=-1$, $\omega(e_0)=-3$ and $e_{p}\cdot e_{q}=0$\end{center}
 for all $p,q\in \B(\p^2)$, $p\not=q$.
\end{corollary}
\begin{proof}Follows from Lemma~\ref{Lem:ZZYbigsum} and the fact that $\Pic(\p^2)=\Z e_0$,  $(e_0)^2=1$, $K_{\p^2}=-3e_0$, so $\omega(e_0)=-3$.
\end{proof}
\begin{definition}\label{Defi:DegMulZZP2}
Let $a\in \ZZ_{\p^2}$ and $q\in \B(\p^2)$. We define  the \emph{degree} of $a$ to be $\deg(a)=e_0\cdot a\in \Z$ and the \emph{multiplicity of $a$ at $q$} to be $m_q(a)=e_{q}\cdot a\in \Z$.  We then define the set of \emph{base-points} $\Base(a)$ of $a$ to be $\{q\in \B(\p^2)\mid m_q(a)\not=0\}$.
\end{definition}
\begin{remark}\label{Rem:BasePtsClassical}
Let $\Lambda$ be a linear system on $\p^2$ which we assume of positive dimension and without fixed component. Denote by $p_1, \ldots,p_r$ its base-points  and by $\pi \colon X \to \p^2$ the blow-up of the points $p_i$ on $\p^2$. Then, the strict transform $\tilde{\Lambda}$  of $\Lambda$ on $X$ is a base-point free linear system and we have
\begin{equation} \tag{$\diamondsuit$}\label{tilde-Lambda}
\tilde{\Lambda}=d\pi^*(e_0)-\sum_{i=1}^r m_i e_{p_i} \in \Pic(X), \end{equation}
where $d$ is the degree of $\Lambda$, $m_1,\dots,m_r\ge 1$ are its multiplicities at the points $p_1,\dots,p_r$ and $e_{p_i}
\in 
\Pic(X)$ is the pull-back
(or total transform)
of the exceptional curve produced by blowing-up $p_i$.

Indeed, since we have $\Pic(X)=\pi^*(\Pic(\p^2))\oplus (\bigoplus_{i=1}^r \Z e_{p_i})$, there exist integers $d$, $m_1,\dots,m_r$ for which the equality \eqref{tilde-Lambda} holds. We then compute
\[ d=\pi^*(e_0)\cdot \tilde{\Lambda}=\pi^*(e_0)\cdot \pi^*(\Lambda)=e_0\cdot \Lambda=\deg(\Lambda)\]
and we see that the multiplicity of $\Lambda$ at $p_i$ is  $e_{p_i}\cdot \tilde{\Lambda}=m_i$. Hence, the definitions of base-points, degree and multiplicities coincide with the classical ones.
\end{remark}

\begin{definition}
Let $\varphi\colon Y_1\dasharrow Y_2$ be a birational map between two smooth projective rational surfaces. We define an isomorphism $\varphi_\bullet\colon \ZZ_{Y_1} \to \ZZ_{Y_2}$ in the following way:

An element $c\in \ZZ_{Y_1}$ corresponds to the class of a triple $(C,X,\pi_1)$.
By blowing-up more points if necessary, we may assume that $\pi_1$ is such that $\pi_2:= \varphi\circ \pi_1\colon X\to Y_2$ is a birational morphism.
We then define $\varphi_\bullet(c)\in \ZZ_{Y_2}$ to be the class of $(C,X,\pi_2)$.
\end{definition}
\begin{remark}
If $\varphi\colon Y_1\dasharrow Y_2$ and $\psi\colon Y_2\dasharrow Y_3$ are two birational maps between smooth projective rational surfaces, then $(\psi\circ \varphi)_\bullet=\psi_\bullet\circ \varphi_\bullet$. This implies that $\varphi$ and $\psi$ are isomorphisms of $\Z$-modules. They moreover preserve the intersection form and the canonical form (which can be checked on blowing-ups).
\end{remark}
We then obtain the following result:
\begin{lemma}\label{Lemm:BirP2ZP2}
The group $\Bir(\p^2)$ acts faithfully on $\ZZ_{\p^2}$ and preserves the intersection form and the canonical form. Moreover, if $f\in \Bir(\p^2)$, then $(f_\bullet)^{-1}(e_0)=de_0-\sum_{i=1}^r m_i e_{p_i}$, where $d=\deg f $, $p_1,\dots,p_r\in \B(\p^2)$ are the base-points of $f$ and $m_1,\dots,m_r\ge 1$ are their multiplicities.
\end{lemma}
\begin{proof}
We decompose every $f\in \Bir(\p^2)$ into $f=\eta \circ \pi^{-1}$, where $\eta \colon X \to \p^2$, $\pi \colon X\to \p^2$ are blow-ups of the base-points of $f^{-1}$ and $f$ respectively.
\[ \xymatrix{
& X \ar[dl]_{\pi} \ar[dr]^{\eta} \\ 
\p^2 \ar@{-->}[rr]^{f} & & \p^2} \]
We have $\Pic(X)=\pi^*(\Pic(\p^2))\oplus (\bigoplus_{i=1}^r \Z e_{p_i})$, where $e_{p_1},\dots,e_{p_r}$ are the pull-backs in $\Pic(X)$ of the exceptional divisors of the base-points $p_1,\dots,p_r$ of $f$ (or equivalently of~$\pi$) and can thus write $\eta^*(e_0)\in \Pic(X)$ as $\eta^*(e_0)=d \, \pi^*(e_0)-\sum m_i e_{p_i}$, where $d$ is the degree of the linear system and $m_1,\dots,m_r\ge 1$ are the multiplicities of the linear system at the points $p_1,\dots,p_r$ (see Remark~\ref{Rem:BasePtsClassical}). The fact that for each non-trivial $f\in \Bir(\p^2)$, we can choose a general point $p\in \p^2$, sent by $f$ onto another point $q$, yields $f(e_{p})=e_{q}\not=e_{p}$ and shows that the action is faithful.
\end{proof}

\begin{example}   \label{ExampleSigma}
Let $\sigma\colon [x:y:z]\dasharrow [yz:xz:xy]$ be the standard quadratic transformation of $\p^2$. Its base-points are $p_1=[1:0:0]$, $p_2=[0:1:0]$, $p_3=[0:0:1]$. We then write
\[X=\{([x_0:x_1:x_2],[y_0:y_1:y_2])\mid x_0y_0=x_1y_1=x_2y_2\}\]
and denote by $\pi\colon X\to \p^2$ and $\eta\colon X\to\p^2$ the first and second projections, which are blow-ups of $p_1,p_2,p_3$ and satisfy $\eta=\sigma\circ\pi$. There are six $(-1)$-curves $E_1,E_2,E_3,F_1,F_2,F_3$ on $X$, where $E_i=\pi^{-1}(p_i)$ and $F_i=\eta^{-1}(p_i)$ for $i=1,2,3$. 

The action of $\sigma$ on $\ZZ_{\p^2}$ is given as follows. Firstly we have 
$\sigma_\bullet(e_0)=2e_0-e_{p_1}-e_{p_2}-e_{p_3}$ (Lemma~\ref{Lemm:BirP2ZP2}).
Secondly we have $\sigma_\bullet (e_{p_1})=e_0-e_{p_2}-e_{p_3}$ thanks to the
corresponding equality $E_1 = \eta^* (e_0) - F_2-F_3$ which holds in $\Pic (X)$. The latter
equality holds because $E_1$ is the strict transform  of the line through $p_2,p_3$ by $\eta$. Similarly, we obtain $\sigma_\bullet(e_{p_2})=e_0-e_{p_1}-e_{p_3}$ and $\sigma_\bullet(e_{p_3})=e_0-e_{p_1}-e_{p_2}$.

For all other points $q\in \B(\p^2)\setminus \{p_1,p_2,p_3\}$, we have $\sigma_\bullet(e_{q})=e_{q'}$, for some $q'\in \B(\p^2)$.
\end{example}

In the sequel, the isomorphism $\varphi_\bullet\colon \ZZ_{Y_1} \to \ZZ_{Y_2}$ associated with $\varphi\colon Y_1\dasharrow Y_2$ will be denoted by $\varphi$.

\subsection{The infinite Weyl group}\label{SubSec:InfiniteWeylgroup}
\begin{definition}
Denote by $ \Aut(\ZZ_{\p^2})$ the group of linear automorphisms of the $\Z$-module $\ZZ_{\p^2}$ that preserve the intersection form and define $\Sym_{\p^2}\subseteq \Aut(\ZZ_{\p^2})$ to be the subgroup of elements that fix $e_0$ and permute the $e_{p}$, $p\in \B(\p^2)$.

We define $\Weyl  \subseteq   \Aut(\ZZ_{\p^2})$ to be the \emph{infinite Weyl group} generated by $\Bir(\p^2)$ and the group $\Sym_{\p^2}$.
\end{definition}
\begin{remark}
Note that $\Aut(\p^2)= \Sym_{\p^2}\cap \Bir(\p^2)$. Moreover, the Noether-Castelnuovo theorem yields $\Bir(\p^2)=\langle \Aut(\p^2),\sigma\rangle$, which implies that $\Weyl=\langle \Sym_{\p^2},\sigma\rangle$.
Later on (see Corollary~\ref{corollary:W=SymBirSym}),
we will prove that  $\Weyl = \Sym_{\p^2} \Bir(\p^2) \Sym_{\p^2}$.
\end{remark}

\begin{definition}\label{Defi:DegreeMult}
Let $f\in \Weyl$ and $q\in \B(\p^2)$. We define  the \emph{degree} of $f$ to be $\deg f = e_0\cdot f^{-1}(e_0)\in \Z$ and the \emph{multiplicity of $f$ at $q$} to be $m_q(f)=e_{q}\cdot f^{-1}(e_0)\in \Z$. We denote $\Base(f) \subseteq \B(\p^2)$ the set of points $q$ such that $m_q(f)\not=0$.
\end{definition}

\begin{remark} \label{Remark:relation-between-mutliplicities-in-Z-and-in-the-Weyl-group}
By construction, the degree, base-points and multiplicities of $f\in \Weyl$ are the same as for $f^{-1}(e_0)$ $\in \ZZ_{\p^2}$ (which were defined in Definition~\ref{Defi:DegMulZZP2}). By Lemma~\ref{Lemm:BirP2ZP2}, this definition coincides with the classical definition if $f\in \Bir(\p^2)$.
\end{remark}

\begin{lemma}   \label{Lemm:EasyWeyl}
\item
\begin{enumerate}
\item\label{L1}
Every element of $\Weyl$ preserves the intersection form and the canonical form.
\item\label{L2}
For each $f\in \Weyl$ we have 
\[f^{-1}(e_0) = (\deg f) \cdot e_0-\sum\limits_{q \, \in \,  \Base(f)} m_q(f) \cdot e_{q}\]
and the following equalities hold $($Noether equalities$)$, where $d = \deg f$:
\[\sum\limits_{q \, \in \,  \Base(f)}   m_q(f)=3(d-1),\ \sum\limits_{q \, \in \,  \Base(f)}  (m_q(f))^2=d^2-1.\]
\item\label{L3} For each $f\in \Weyl$, we have $\deg f^{-1} = \deg f$.

\item\label{L4}
$\Sym_{\p^2}=\{f\in \Weyl \mid f(e_0)=e_0\}=\{f\in \Weyl \mid \deg(f)=1\}=\{f\in \Weyl \mid \deg(f)=\pm 1\}.$

\item\label{L5}
For each $\alpha\in \Sym_{\p^2}$, we have $\sigma  \circ \alpha \circ \sigma \in \Sym_{\p^2}$ if and only if $\alpha$ preserves the set $\{e_{[1:0:0]},e_{[0:1:0]},e_{[0:0:1]}\}$. 

\end{enumerate}
\end{lemma}
\begin{proof}
\ref{L1}: Follows from the fact that $\Bir(\p^2)$ and $\Sym_{\p^2}$ preserve the intersection form and the canonical form.

\ref{L2}: The first equality follows from Definition~\ref{Defi:DegreeMult} and from the next identity:
\[ \forall \, a \in \ZZ_{\p^2}, \; a = (a \cdot e_0) \, e_0 - \sum_{q \, \in \,   \B(\p^2)} (a \cdot e_q) \, e_q. \] 
The Noether equalities follows from \ref{L1}, since $f^{-1}(e_0)^2=d^2-\sum (m_i)^2=(e_0)^2=1 $ and $\omega(f^{-1} (e_0)) = -3d+\sum m_i=\omega(e_0)=-3$.

\ref{L3}: We have $\deg f = e_0 \cdot f^{-1}(e_0) = f(e_0) \cdot e_0 = \deg f^{-1}$.

\ref{L4}: Let $f\in \Weyl$ be such that $\deg f=d=\pm 1$. It follows successively from the Noether equalities that all multiplicities $m_q(f)$ are zero, $d=1$ and $f(e_0)=e_0$. For each $p\in \B(\p^2)$ we have $f(e_{p})\cdot e_0=0$, so $f(e_{p})=\sum_{i=1}^r a_i e_{q_i}$, for some $q_1,\dots,q_r\in \B(\p^2)$. As $-1=(e_{p})^2=(f(e_{p}))^2=-\sum (a_i)^2$, we find that $f(e_{p})=\pm e_{q_i}$ for some $i$. Since $\omega(e_{p})=\omega(f(e_{p}) )$, we get $f(e_{p})= e_{q_i}$.  Hence, $f\in \Sym_{\p^2}$.

\ref{L5}: By \ref{L4}, we have $\sigma \circ \alpha \circ \sigma \in \Sym_{\p^2}$ if and only $\sigma \circ \alpha \circ \sigma(e_0)=e_0$. Since this corresponds to $\alpha \circ \sigma(e_0) = \sigma(e_0)$, the result follows from the equality $\sigma(e_0)=2e_0-e_{[1:0:0]}-e_{[0:1:0]}-e_{[0:0:1]}$ (Example~\ref{ExampleSigma}). 
\end{proof}

\begin{corollary}   \label{Coro:ExistenceAlpha}
Let $f,g\in \Weyl$. The following conditions are equivalent:
\begin{enumerate}
\item
 $f^{-1}(e_0)=g^{-1}(e_0).$
 \item
 There exists $\alpha\in \Sym_{\p^2}$ such that $g=\alpha \circ f$.
 \end{enumerate}
\end{corollary}

\begin{proof}
Let us write $\alpha=g \circ f^{-1}$. By Lemma~\ref{Lemm:EasyWeyl}\ref{L4}, $\alpha \in \Sym_{\p^2}$ if and only if $\alpha(e_0)=e_0$. Applying $g^{-1}$, this condition is equivalent to $f^{-1}(e_0)=g^{-1}(e_0)$.
\end{proof}

\begin{corollary} \label{corollary= degree of a composition}
If $f, g \in \Weyl$, we have
\[ \deg f \circ g^{-1} = (\deg f) (\deg g) - \sum_{q \in \B (\p^2) } m_q(f) m_q(g).\]
\end{corollary}

\begin{proof}
We have $\deg f \circ g^{-1} = e_0 \cdot  (f \circ g^{-1})^{-1}(e_0) = f^{-1}(e_0) \cdot g^{-1} (e_0)$, so that the result follows from Lemma~\ref{Lemm:EasyWeyl}\ref{L2}.
\end{proof}

\begin{lemma}  \label{Lemma:Imagepts}
Let $g\in \Weyl$ and let $q\in \B(\p^2)$. 
\begin{enumerate}
\item \label{ImageBsPt}
If $q\in \Base(g)$, then $g(e_q)=m_q(g) e_0-\sum_{p\in \Base(g^{-1})} a_p e_p$, $a_p\in \Z$.
\item  \label{ImageNotBsPt}
If $q\notin \Base(g)$, then $g(e_q)=e_{\tilde{q}}$ for some $\tilde{q}\in \B(\p^2) \setminus \Base(g^{-1})$.

\end{enumerate} 
In particular, $g$ induces a bijection $\B(\p^2)\setminus \Base(g) \to \B(\p^2)\setminus \Base(g^{-1})$.
\end{lemma}

\begin{proof}
We write $g(e_q)=de_0+\sum a_i e_{p_i}$, for some $\{p_1,\dots,p_n\}\subseteq \B(\p^2)$. We then observe that $d=e_0 \cdot g(e_q)=g^{-1}(e_0)\cdot e_q=m_q(g)$.

If $q\notin \Base(g)$, we then obtain $g(e_q)=\sum a_i e_{p_i}$. Since $1=-\omega(e_q)=-\omega(\sum a_i e_{p_i})=\sum a_i$ and $1=-(e_q)^2=\sum (a_i)^2$, we find that $g(e_q)$ is equal to $e_{\tilde{q}}$ for some $\tilde{q}\in \B(\p^2)$. Moreover, $\tilde{q}\not\in \Base(g^{-1})$, since $m_{\tilde{q}}(g^{-1})=e_{\tilde{q}}\cdot g(e_0)=e_q\cdot e_0=0$. This yields~\ref{ImageNotBsPt}.

To get \ref{ImageBsPt}, we consider the case $q\in \Base(g)$ and need to show that  if $p_i\notin \Base(g^{-1})$, then $a_i=0$. This is because $a_i=e_{p_i}\cdot g(e_q)=g^{-1}(e_{p_i})\cdot e_q$ and because $g^{-1}(e_{p_i})$ is equal to $e_{\tilde{p_i}}$ for some $\tilde{p_i}\in \B(\p^2)\setminus \Base(g)$ by~\ref{ImageNotBsPt}.
\end{proof}

\begin{corollary}
For each $g\in \Weyl$ and each $q\in \B(\p^2)$, we have 
\[q \notin \Base(g) \Leftrightarrow g(e_q)=e_{\tilde{q}}\text{ for some }\tilde{q}\in \B(\p^2).\]
\end{corollary}
\begin{proof}Follows from Lemma~\ref{Lemma:Imagepts}.\end{proof}

\begin{corollary} \label{corollary: inclusion of the base locus of a composition in a special case}
Let $f, g\in \Weyl$ be such that $ \Base (f) \subseteq  \Base (g^{-1} ) $, then we have
\[ \Base (f \circ g) \subseteq \Base (g) .\]
\end{corollary}

\begin{proof}
Take $q \in \B(\p^2) \setminus \Base (g)$. Then, by Lemma \ref{Lemma:Imagepts}, we have $g(e_q)=e_{\tilde{q}}$ for some $\tilde{q}\in \B(\p^2) \setminus \Base(g^{-1})$. It follows that $\tilde{q}\in \B(\p^2) \setminus \Base(f)$, so that $f (e_{\tilde{q}} ) = e_{ \tilde { \tilde{q} } }$ for some $\tilde { \tilde{q} } \in \B(\p^2)$, i.e.~$f \circ g (e_q) = e_{ \tilde { \tilde{q} } }$, proving that $q \notin \Base (f \circ g)$.
\end{proof}

As explained before, the infinite Weyl group $\Weyl$ contains $\Bir(\p^2)$. In some sense, this corresponds to forgetting the configuration of points. However, several properties of the action of $\Bir(\p^2)$ on $\ZZ_{\p^2}$ extend to $\Weyl$. For instance, the Noether equalities  (Lemma~\ref{Lemm:EasyWeyl}\ref{L2}) are fulfilled by any element of $\Weyl$. A priori, the degree and multiplicities could be negative, but we will show that it is not the case (Lemma~\ref{Lemm:SymBirSim}). Also, there are some elements of $\ZZ_{\p^2}$ which satisfy the Noether equalities but which are not in the orbit of $e_0$. However, there is an algorithm to decide if an element is in this orbit (Algorithm~\ref{Algo:Hudsonalgo} below, corresponding to the classical Hudson test).\\

\begin{lemma}  \label{Lemm:SymBirSim}
Let $f \in \Weyl \smallsetminus \Sym_{\p^2}$. 
\begin{enumerate}
\item\label{LU}
For each finite set $\Delta = \{q_1,\dots,q_s\} \subseteq  \B(\p^2)$ of $s \ge 1$ points there exists a dense open set $U\subseteq (\p^2)^s$ such that for each $(p_1,\dots,p_s) \in U$:
\begin{enumerate}
\item The points $p_1, \ldots,p_s$ are distinct;
\item There exists an element $g\in \Bir(\p^2)$ satisfying
\[ \deg g = \deg f \quad \text{and} \quad m_{p_i}(g)=m_{q_i}(f)  \text{ for } i=1,\dots,s.\]
 \end{enumerate}
\item\label{LPositv}
The degree and multiplicities of $f$ satisfy
 \[ \deg f \ge 2 \quad \text{and} \quad m_q(f) \geq 0  \text{ for each } q\in \B(\p^2).\]
\item\label{LProd}
There exist $\alpha,\beta\in \Sym_{\p^2}$ and $g\in \Bir(\p^2)$, such that $f=\alpha  \circ g  \circ \beta$.
\end{enumerate}
\end{lemma}

\begin{proof}Let us first observe that $\ref{LPositv}$ and $\ref{LProd}$ follow from $\ref{LU}$. Indeed, take for $\Delta$ the set  $\Base(f)= \{q_1,\dots,q_s\}$ and let $U$ be the corresponding open subset of $(\p^2)^s$ given by $\ref{LU}$. Choose $(p_1,\dots,p_s)\in U$ and choose $g\in \Bir(\p^2)$ such that $\deg g = \deg f $ and $m_{p_i}(g)=m_{q_i}(f)$ for $i=1,\dots,s$. Then choose $\beta   \in \Sym_{\p^2}$ that sends $e_{q_i}$ onto $e_{p_i}$ for each $i$, and note that $\beta \circ f^{-1}(e_0) = g^{-1}(e_0)$. Hence, the element $\alpha=f \circ \beta^{-1} \circ g^{-1}$ belongs to $\Sym_{\p^2}$ by Lemma~\ref{Lemm:EasyWeyl}\ref{L4}.

To prove $\ref{LU}$, we write $f=\alpha_{l} \circ \sigma \circ  \cdots \circ \alpha_1 \circ \sigma \circ \alpha_{0}$ where $l\ge 1$, $\alpha_0,\dots,\alpha_l\in\Sym_{\p^2}$ and prove the result by induction on $l$. As the result does not change under right or
left multiplication
by elements of $\Sym_{\p^2}$, we can moreover assume that $\alpha_0$ and $\alpha_{l}$ are equal to the identity. We can also always enlarge the set $\Delta$.

If $l=1$, then $f=\sigma$, so that $f$ has degree $2$ and three base-points of multiplicity~$1$ (see Example~\ref{ExampleSigma}). We may assume that $q_1,q_2,q_3$ are the base-points of $f$. Then, we can choose for $U$ the open subset of points $(p_1, \ldots,p_s)$ in $(\p^2)^s$ where $p_1,\ldots, p_s$ are distinct and where $p_1,p_2,p_3$ are not collinear. For each $(p_1,\dots,p_s)\in U$, we choose an element $\alpha\in \Aut(\p^2)$ sending $p_i$ onto $q_i$ for $i=1,2,3$ and choose $g=f \circ \alpha$.

For $l\ge 2$, we write $ f= f' \circ \sigma$ and apply the induction hypothesis to $f'$. Up to enlarging $\Delta = \{ q_1, \ldots, q_s \}$, we may assume that $q_1=[1:0:0],q_2=[0:1:0],q_3=[0:0:1]$. For each $i\ge 4$, define $q'_i$ as the unique point of $\B(\p^2)$ such that $\sigma ( e_{q_i})=e_{q_i'}$. For $i=1,2,3$, set $q'_i = q_i$.  One can assume that $\Base(f') \subseteq \{q_1',\dots,q_s'\}$. Let $U' \subseteq (\p^2)^{s}$ be an open subset associated to $f'$ and $\Delta' =\{q_1',\dots,q_s'\}$  via the induction hypothesis.

Let $T \subseteq  (\p^2)^3$ be the open subset 
of triplets $(p_1,p_2,p_3) \in (\p^2)^3$ such that the $p_i$ are in the affine chart $\{ [x:y:z] \in \p^2, \; x \neq 0 \}$ and are not collinear. We have
\[T =\left\{([1:a_1:a_2],[1:a_3:a_4],[1:a_5:a_6])  \mid a_1,\dots,a_6\in \k, \; \det  \left(\begin{array}{rcl}
1& 1& 1\\
a_1 & a_3& a_5\\
a_2 & a_4 & a_6\end{array}\right)  \not=0\right\}. \]
Let
$\rho\colon T \to \Aut(\p^2)= \PGL_3(\k)$ be the morphism defined by
\[ \rho ([1:a_1:a_2],[1:a_3:a_4],[1:a_5:a_6])  = \left(\begin{array}{rcl}
1& 1& 1\\
a_1 & a_3& a_5\\
a_2 & a_4 & a_6\end{array}\right).\]
If $\triple = (p_1,p_2,p_3)$ belongs to $T$, we set
\[ \sigma_{\triple} := \rho ( \triple )  \circ \sigma \circ \rho (\triple )^{-1} \in \Bir (\p^2).\]
Note that $\sigma_{\triple}$ is a quadratic involution having base-points at $p_1,p_2,p_3$.

We then denote by $U \subseteq (\p^2)^{s}$ the dense open subset of $s$-uples $p=(p_1,\dots,p_s)$ such that:
\begin{enumerate}
\item No three of the points $p_i$ are collinear (so that in particular the points $p_i$ are distinct) ;
\item The triple $\triple =(p_1,p_2,p_3)$ belongs to $T$;
\item The $s$-uple $p'=(p_1',\dots,p_s')$ belongs to $U'$, where the elements $p'_i$ are defined by
$p_i' :=p_i$ for $i \le 3$ and by $p'_i := \sigma_{\triple} (p_i) \in \p^2$ for $i \geq 4$.
\end{enumerate}

For each $p\in U$, the corresponding $p'\in U'$ yields an element $g'\in \Bir(\p^2)$ satisfying $\deg g' = \deg f' $ and $m_{q_i'}(f)=m_{p_i'}(f')$ for each $i$. Taking  $\beta'\in \Sym_{\p^2}$ that sends $e_{q_i'}$ onto $e_{p_i'}$ for each $i$, the fact that $\Base(f')\subseteq \{q_1',\dots,q_s \}$ implies as before that $\beta'(f'^{-1}(e_0))={g'}^{-1}(e_0)$, so $f' = \alpha \circ g' \circ \beta'$, for some $\alpha \in \Sym_{\p^2}$.

We write $\nu=\rho( \triple )\in \Aut(\p^2)$, $\sigma_{\triple} =\nu \circ \sigma \circ \nu^{-1}\in \Bir(\p^2)$ as before and obtain
\[f=f' \circ \sigma=\alpha \circ g' \circ \beta' \circ \sigma = \alpha \circ  g \circ \beta,\]
where $g= g' \circ \sigma_{\triple} \in \Bir(\p^2)$ and  $\beta=\sigma_{\triple} \circ \beta' \circ \sigma=\nu \circ \sigma \circ\nu^{-1} \circ \beta' \circ \sigma \in \Weyl$. 
For $i\in\{1,2,3\}$, both $\nu$ and $\beta'$ send  $e_{q_i}$ to $e_{p_i}$, hence $\nu^{-1} \circ \beta'$ fixes $e_{q_i}$. This shows that $\beta \in \Sym_{\p^2}$ (Lemma~\ref{Lemm:EasyWeyl}\ref{L5}), and thus that $\deg g = \deg f $.

It remains to observe that $\beta$ sends $e_{q_i}$ to $e_{p_i}$ for each $i$, to obtain $m_{p_i}(g)=m_{q_i}(f)$ for each $i$. The fact that $\nu^{-1} \circ \beta'$ fixes $e_0-e_{q_1}-e_{q_2}$ implies that $\sigma \circ \nu^{-1} \circ \beta' \circ \sigma$ fixes $\sigma(e_0-e_{q_1}-e_{q_2})=e_{q_3}$ (see Example~\ref{ExampleSigma}) and thus that $\beta$ sends $e_{q_3}$ to $e_{p_3}$. The same works for $e_{q_1},e_{q_2}$. For $i\ge 4$, we have 
\[\beta(e_{q_i}) = \sigma_{\triple} \circ \beta' \circ \sigma(e_{q_i})=\sigma_{\triple} \circ \beta'(e_{q_i'})=\sigma_{\triple} (e_{p_i'})=e_{p_i}. \qedhere \]
\end{proof}

The first two corollaries of Lemma~\ref{Lemm:SymBirSim} are stated for an easier reading. The first one is  \cite[Proposition 2.4]{BlaCal}:

\begin{corollary}  \label{corollary:existence-of-a-Cremona-transformation-with-assigned-homaloidal-type-on-a-dense-open-subset}
For each homaloidal type $(d ; m_1, \ldots,m_s)$, there exists a dense open subset $U \subseteq (\p^2)^s$ such that for each $(p_1, \ldots,p_s) \in U$ the following holds:
\begin{enumerate}
\item The points $p_1, \ldots,p_s$ are distinct;
\item There exists a Cremona transformation $f \in \Bir(\p^2)$ such that
\[ f^{-1} (e_0) = d e_0 - \sum_{i=1}^s m_i e_{p_i}.\]
\end{enumerate}
\end{corollary}

\begin{corollary}  \label{corollary:W=SymBirSym}
We have $\Weyl = \Sym_{\p^2} \Bir (\p^2)   \Sym_{\p^2}$.
\end{corollary}

In the next two corollaries, we give information about
the orbits of $e_0, e_q$ and $e_0-e_{q}$, where $q$ is a point of $\B(\p^2)$. The first one is the following positivity result on the degree and multiplicities of elements in the orbit of $e_0$, which also follows from \cite[Lemma 5.3]{BlaCan} (with another proof).

\begin{corollary}  \label{Coro:PositivityOrbit}
Each element $a\in \Weyl(e_0)$ can be written as
\[\begin{array}{l}a= (\deg a) \cdot e_0 - \sum_{q \in \Base(a)} m_q(a) \cdot e_q,\end{array}\]
where $\deg a \ge 1$, $m_q(a)\ge 1$ for each $q \in \Base(a)$ and
\[\begin{array}{l}\sum_{q\in \Base(a)} m_q(a)=3(\deg a - 1),\ \sum_{q\in \Base(a)} (m_q(a))^2=\deg(a)^2-1.\end{array}\] 
Moreover, for any two distinct $q,q'\in \B(\p^2)$ we have $m_q(a)+m_{q'}(a) \le \deg a $.
\end{corollary}

\begin{proof}
Write $a=f(e_0)$ for some $f\in \Weyl$, and decompose $f= \alpha \circ g \circ  \beta$ where $\alpha,\beta\in \Sym_{\p^2}$  and $g \in \Bir(\p^2)$, using Lemma~\ref{Lemm:SymBirSim}\ref{LProd}.
Hence, $a=\alpha \circ g(e_0)$. The description above follows then from Lemmas~\ref{Lemm:BirP2ZP2} and \ref{Lemm:EasyWeyl}. The inequality $m_q(a)+m_{q'}(a)\le \deg(a)$ can be checked for $g(e_0)$, since $\alpha\in \Sym_{\p^2}$. We can moreover assume that $q\in \p^2$ and $q'$ is either in $\p^2$ or in the first blow-up of $q$ (since $m_q(g(e_0)) \le m_{q'}(g(e_0))$ if $q\ge q'$). The result follows then from B\'ezout theorem, by intersecting the line through $q$ and $q'$ with the linear system corresponding to $g(e_0)$.
\end{proof}

The following result is obvious for orbits of $\Bir(\p^2)$ and
is here
generalised to orbits of $\Weyl$. This allows to say that elements of $\Weyl$ have a behaviour ``not too far'' from elements of $\Bir(\p^2)$. 
See also Lemma~\ref{Lemm:SymBirSim}\ref{LProd} for another result in this direction.

\begin{corollary}\label{Cor:IntersectionTwoOrbits}
Let $q\in \B(\p^2)$ and let $a\in \Weyl(e_0)$.
\begin{enumerate}
\item\label{Ge0e0}
For each $b\in \Weyl(e_0)$, we have $a\cdot b\ge 1$.
\item\label{Ge0e1}
For each $b\in \Weyl(e_0-e_q)$, we have $a\cdot b\ge 1$. 
\item\label{Ge0eq}
For each  $b\in \Weyl(e_q)$, we have $a\cdot b\ge 0$. 
\end{enumerate}
\end{corollary}
\begin{proof}
We apply an element of $\Weyl$ and assume that $b$ is equal to $e_0$, $e_0-e_q$, $e_q$ respectively. By Corollary~\ref{Coro:PositivityOrbit}, we have $a = (\deg a) \cdot e_0 - \sum_{p\in \Base(a)}m_p(a)\cdot e_p$, where $\deg a \ge 1$, $m_p(a)\ge 1$ for each $p\in \Base(a)$ and $\sum_{p\in \Base(a)} (m_p(a))^2=\deg(a)^2-1$.
We then find that $a \cdot b$ is equal to $\deg a$, $\deg a -m_q(a)$, $m_q(a)$ respectively. Assertions \ref{Ge0e0}, \ref{Ge0e1}, \ref{Ge0eq} are then given by $\deg a \ge 1$,  $(m_q(a))^2\le \deg(a)^2-1$ and $m_q(a)\ge 0$. 
\end{proof}

\subsection{Jonqui\`eres elements viewed in the Weyl group}
We now define the analogue of the groups $\Jonq_p\subseteq \Bir(\p^2)$ in the Weyl group:

\begin{definition}
For each $q\in \B(\p^2)$ we define $\J_q \subseteq \Weyl$ as the subgroup 
\[\J_q=\{\varphi\in \Weyl\mid \varphi(e_0-e_{q})=e_0-e_{q}\}.\]
\end{definition}

\begin{lemma}
For each $q\in \p^2$, we have $\Jonq_q = \J_q \cap \Bir (\p^2)$.
\end{lemma}
\begin{proof}
Let $\pi\colon \p^2\dasharrow \p^1$ be the projection from $q$, and let $\eta\colon X\to \p^2$ be the blow-up of $q$. The result follows from the fact that $e_0-e_q\in \ZZ_{\p^2}$ corresponds to the divisor of $\Pic(X)$ corresponding to the fibres of the morphism $\pi\circ \eta\colon X\to \p^1$.
\end{proof}

\begin{definition}  \label{Defi:Iota}
For each $q\in \B(\p^2)$ and for each finite set $\Delta \subseteq \B(\p^2) \smallsetminus \{q\}$ of even order $2n$, we define 
$\iota_{q,\Delta}\in \J_q$
to be the involution given by
\[\begin{array}{rclrcl}
\iota_{q,\Delta}(e_0)&=&(n+1)e_0-ne_q-\sum\limits_{r\in \Delta} e_r,&
\iota_{q,\Delta}(e_r)&=& e_0-e_q-e_r, r\in \Delta,\\
\iota_{q,\Delta}(e_q)&=& ne_0-(n-1)e_q-\sum\limits_{r\in \Delta} e_r,&
\iota_{q,\Delta}(e_r)&=&e_r,  r\in \B(\p^2)\setminus( \Delta\cup \{q\}).
\end{array}\]
\end{definition}

\begin{remark} \label{Remark:Iotaproduct}
In order to see that the elements $\iota_{q,\Delta}$ belong to $\J_q \subseteq \Weyl$, we can observe that $\iota_{q,\emptyset}$ is the identity, that $\iota_{q,\Delta}$ is equal to $\sigma$, up to left and right multiplication
by elements of $\Sym_{\p^2}$ when $\Delta$ contains $2$ elements, and that $\iota_{q,\Delta} \circ \iota_{q,\Delta'}=\iota_{q, (\Delta \cup \Delta') \smallsetminus (\Delta\cap \Delta')}$.
\end{remark}

\begin{definition} \label{Defi:sigmap1p2p3}
Let $p_1,p_2,p_3\in \B(\p^2)$ be $3$ distinct points. We define 
$\sigma_{p_1,p_2,p_3 }\in \Weyl$ as the involution given by
\[\begin{array}{lllllll}
\sigma_{p_1,p_2,p_3}(e_0)&=&2e_0-e_{p_1}-e_{p_2}-e_{p_3},&
\sigma_{p_1,p_2,p_3 }(e_{p_1})&=& e_0-e_{p_2}-e_{p_3}\\
\sigma_{p_1,p_2,p_3}(e_{p_2})&=& e_0-e_{p_1}-e_{p_3},&
\sigma_{p_1,p_2,p_3}(e_{p_3})&=& e_0-e_{p_1}-e_{p_2},\\
\sigma_{p_1,p_2,p_3}(e_r)&=&e_r, r\in \B(\p^2)\setminus\{p_1,p_2,p_3\},\
\end{array}\]
We observe that $\sigma_{p_1,p_2,p_3} \in \J_{p_i}$ for $i=1,2,3$, and that $\sigma_{p_1,p_2,p_3}=\tau_{p_2,p_3} \circ \iota_{p_1,\{p_2,p_3\}}$, where $\tau_{p_2,p_3} \in \Sym_{\p^2}$ is the transposition permuting $p_2$ and $p_3$. 
\end{definition}
\begin{remark}
When $p_1=[1:0:0]$, $p_2=[0:1:0]$, $p_3=[0:0:1]$, we observe that  $\sigma_{p_1,p_2,p_3}$ is similar to the standard quadratic involution $\sigma\colon [x:y:z]\dasharrow [yz:xz:xy]$. It is however not realised by an element of $\Bir(\p^2)$ as it fixes all points of $\p^2\setminus \{p_1,p_2,p_3\}$. Moreover, $\sigma_{p_1,p_2,p_3}$ and $\iota_{p_1,\{p_2,p_3\}}$ both belong to $\Sym_{\p^2}\sigma=\{\alpha \circ \sigma\mid \alpha\in \Sym_{\p^2}\}\subseteq \Weyl$.
\end{remark}

\begin{lemma}  \label{LemmJvarphiLin}
\item  
\begin{enumerate}
\item\label{LemmJvarphiLin1}
For each $q\in \B(\p^2)$ and each $\varphi\in \J_q$, we have $m_q(\varphi)=\deg(\varphi)-1$.
\item\label{LemmJvarphiLin2}
For each $q\in \B(\p^2)$ and each $\varphi\in \Weyl$ with $m_q(\varphi)=\deg(\varphi)-1$,
the set $\Delta=\Base(\varphi)\setminus\{q\}$ has even cardinality $2n\ge 0$ and
$$\varphi^{-1}(e_0)=(\iota_{q,\Delta})^{-1}(e_0)=(n+1)e_0-ne_q-\sum_{r\in \Delta}  e_{r}=e_0+\sum_{r\in \Delta} \left(  \frac{e_0-e_q}{2}-e_r \right) .$$
This yields the existence of $\alpha\in \Sym_{\p^2} $ such that $$\varphi=  \alpha \circ \iota_{q,\Delta}.$$
Moreover, $\alpha \in \J_q$ if and only if $\varphi \in \J_q$.
\end{enumerate}
\end{lemma}
\begin{proof}\ref{LemmJvarphiLin1}
The fact that $\varphi\in \J_q$ implies that
\[\deg(\varphi)-m_q(\varphi)=(e_0-e_q)\cdot \varphi^{-1}(e_0)=(e_0-e_q)\cdot e_0=1.\]
\ref{LemmJvarphiLin2}
Since $\Delta=\Base(\varphi) \smallsetminus \{q\}$ and $m_q(\varphi)=\deg(\varphi)-1$, we can write
\[\begin{array}{l}\varphi^{-1}(e_0)=(n+1)e_0-n e_q-\sum_{r\in \Delta} m_r e_r\end{array}\]
where $n\ge 0$ and $m_r\ge 1$ for each $r\in \Delta$ (Corollary~\ref{Coro:PositivityOrbit}). The Noether equalities (Lemma~\ref{Lemm:EasyWeyl}\ref{L2}) yield $\sum_{r\in \Delta} m_r=\sum_{r\in \Delta} (m_r)^2=2n$, so $m_r=1$ for each $r\in \Delta$, and thus $\Delta$ contains $2n$ elements.

Since $\varphi^{-1}(e_0) = (\iota_{q,\Delta}) ^{-1} (e_0)$, we have $\varphi=\alpha \circ \iota_{q,\Delta}$ for some $\alpha \in \Sym_{\p^2}$ (Corollary~\ref{Coro:ExistenceAlpha}). Since $\iota_{q,\Delta}\in \J_q$, we have $\alpha\in \J_q$  if and only $\varphi\in \J_q$.
\end{proof}

\begin{corollary} \label{corollary: normal form of the elements of Jq}
Any element $\varphi \in \J_q$ admits an expression
\[ \varphi = \alpha \circ \iota_{q, \Delta},\]
where  $\Delta := \Base (\varphi) \smallsetminus \{ q \}$ has even order and $\alpha \in \Sym_{\p^2} \cap \J_q = \{ \beta \in \Sym_{\p^2}, \; \beta (q) = q \}$.
\end{corollary}
\begin{proof}
Directly follows from Lemma~\ref{LemmJvarphiLin}.
\end{proof}

\begin{corollary}\label{Coro:JonqTwopts}
If $q,q'\in\B(\p^2)$ are two distinct points, then $\J_q\cap \J_{q'}$ consists of elements of degree $1$ or $2$.
\end{corollary}

\begin{proof}
It follows from Lemma~\ref{LemmJvarphiLin} that if $\varphi \in \J_q$ is a Jonqui\`eres element with $\deg \varphi \ge 3$, the multiplicity at $q$ is $\deg \varphi -1 \geq 2$ and $q$ is the unique point having this multiplicity.
\end{proof}

We now give the following definition, which generalise the one of Jonqui\`eres elements of $\Bir(\p^2)$, as Lemma~\ref{Lemm:JoCremWinfty} explains.

\begin{definition} \label{definition: Jonquieres element}
An element $\varphi\in \Weyl$ is said to be a \emph{Jonqui\`eres element} if there exists a point $q\in \B(\p^2)$ such that $m_q(\varphi)=\deg(\varphi)-1$.
\end{definition}
\begin{lemma}\label{Lemm:JoWinfty}
Let $\psi\in \Weyl$. The following conditions are equivalent:
\begin{enumerate}
\item \label{GenJ1}
$\psi$ is a  Jonqui\`eres element of $\Weyl$;
\item \label{GenJ2}
There exist $\alpha,\beta\in \Sym_{\p^2}$, $q\in \B(\p^2)$ and $\varphi\in \J_q$ such that $\psi=\alpha\circ \varphi\circ\beta$;
\item\label{GenJ3}
There exist $\alpha\in \Sym_{\p^2}$, $q\in \B(\p^2)$ and $\varphi\in \J_q$ such that $\psi=\alpha\circ \varphi$;
\item\label{GenJ4}
There exist $\alpha\in \Sym_{\p^2}$, $q\in \B(\p^2)$ and $\varphi\in \J_q$ such that $\psi=\varphi\circ\alpha$;
\item\label{GenJ5}
There exist $\alpha \in \Sym_{\p^2}$, $q\in \B(\p^2)$ and a finite set of even order $\Delta\subseteq \B(\p^2)\setminus \{q\}$  such that $\psi=\alpha\circ \iota_{q,\Delta}$.
\end{enumerate}
\end{lemma}
\begin{proof}
$\ref{GenJ1}\Rightarrow \ref{GenJ5}$ is given by Lemma~\ref{LemmJvarphiLin}\ref{LemmJvarphiLin2};  $\ref{GenJ5}\Rightarrow \ref{GenJ3}$ is given by the fact that $\iota_{q,\Delta}\in \J_q$ and $\ref{GenJ3}\Rightarrow \ref{GenJ2}$ follows by taking $\beta=\mathrm{id}$.

$\ref{GenJ2}\Rightarrow \ref{GenJ4}$:  Writing $\varphi'=\alpha\circ  \varphi\circ \alpha^{-1}$, we have $\varphi'(e_0-e_p)=e_0-e_p$ where $p\in \B(\p^2)$ is the element such that $\alpha(e_q)=e_p$. Hence $\psi=\varphi'\circ \alpha'$ where $\varphi'\in \J_p$ and $\alpha'=\alpha\circ\beta\in \Sym_{\p^2}$.

$\ref{GenJ4}\Rightarrow \ref{GenJ1}$: Taking $p\in \B(\p^2)$
such that
$\alpha(e_p)=e_q$ we get \[m_p(\psi)=e_p\cdot \psi^{-1}(e_0)=\alpha(e_p)\cdot \varphi^{-1}(e_0)=e_q\cdot \varphi^{-1}(e_0)=m_q(\varphi)\stackrel{\text{Lemma~\ref{LemmJvarphiLin}\ref{LemmJvarphiLin1}}}{=}\deg(\varphi)-1.\]
It remains to observe that 
\[\deg(\psi)= e_0\cdot \alpha^{-1}(\varphi^{-1}(e_0))=\alpha(e_0)\cdot \varphi^{-1}(e_0)=e_0\cdot \varphi^{-1}(e_0)=\deg(\varphi).\qedhere\]
\end{proof}

\begin{lemma}\label{Lemm:JoCremWinfty}
Let $f\in \Bir(\p^2)$. The following conditions are equivalent:\begin{enumerate}
\item\label{GenJBir1}
$f$ is a  Jonqui\`eres element of $\Bir(\p^2)$;
\item\label{GenJBir2}
$f$ is a Jonqui\`eres element of $\Weyl$;
\item\label{GenJBir3}
There exist $\alpha,\beta\in \Aut(\p^2)$, $q\in \p^2$, and $\varphi\in \Jonq_q\subseteq \Bir(\p^2)$ such that $f=\alpha\circ \varphi\circ\beta$;
\item\label{GenJBir4}
There exist $\alpha\in \Aut(\p^2)$, $q\in \p^2$, and $\varphi\in \Jonq_q\subseteq \Bir(\p^2)$ such that $f=\alpha\circ \varphi$;
\item\label{GenJBir5}
There exist $\alpha\in \Aut(\p^2)$, $q\in \p^2$, and $\varphi \in \Jonq_q \subseteq \Bir (\p^2)$ such that $f=\varphi \circ \alpha$.
\end{enumerate}
\end{lemma}

\begin{proof}
By definition, $f$ is a Jonqui\`eres element of $\Bir(\p^2)$ if and only if there exist two points $p,q \in \p^2$ such that the pencil of lines through $p$ is sent to the pencil of lines through $q$. Composing at the source or the target with a linear automorphism exchanging $p$ and $q$ yields then an element of $\Jonq_p$ or $\Jonq_q$. This yields the equivalence of  \ref{GenJBir1},\ref{GenJBir3},\ref{GenJBir4},\ref{GenJBir5}. As $\Aut(\p^2)\subseteq \Sym_{\p^2}$ and every Jonqui\`eres element of $\Bir(\p^2)$ is a Jonqui\`eres element of $\Weyl$, we have $\ref{GenJBir3}\Rightarrow \ref{GenJBir2}$ (Lemma~\ref{Lemm:JoWinfty}). It remains then to prove $\ref{GenJBir2}\Rightarrow \ref{GenJBir1}$. Assertion~\ref{GenJBir2} implies that $f$  has a base-point $p$ of multiplicity $\deg(f)-1$. We can moreover assume that $p$ is a proper point of $\p^2$
(replacing $p$ with the proper point $p' \in \p^2$ above which $p$ lies only increases the multiplicity).
The image of the pencil of lines through $p$ is then a pencil of lines, passing thus
through
a point $q\in \p^2$. This achieves the proof.
\end{proof}

\begin{definition} \label{definition: equality modulo the symmetric group}
Two elements $a,a'$ of $\ZZ_{\p^2}$ are said to be \emph{equal modulo $\Sym_{\p^2}$} if there exists an element $\alpha \in \Sym_{\p^2}$ such that $a' = \alpha (a)$. This is written $a\equiv_{\Sym_{\p^2}}\! a'$.
\end{definition}

\begin{remark}\label{Rem:EqualModSymBij}
Two elements $a,a'$ of $\ZZ_{\p^2}$ are equal modulo $\Sym_{\p^2}$ if and only if they have the same degree and if there exists a bijection $t \colon \Base (a) \to \Base (a')$ such that $m_{t(p)}(a') = m_p (a)$ for each $p \in \Base (a)$.

In particular, the set  $\Weyl(e_0) / \Sym_{\p^2} $ of equivalence classes modulo $\Sym_{\p^2}$ in $\Weyl(e_0)$ corresponds to the set
of homaloidal types.
\end{remark}

\begin{lemma} \label{lemma: computations using iota_{q,Delta} are sufficient}
For any element $a \in \ZZ_{\p^2}$ and any Jonqui\`eres element $\varphi \in \Weyl$, the element $\varphi (a)$ is equal to some element $\iota_{q, \Delta} (a)$ modulo $\Sym_{\p^2}$.
\end{lemma}

\begin{proof}
This is a direct consequence of Corollary~\ref{corollary: normal form of the elements of Jq} and Definition~\ref{definition: equality modulo the symmetric group}.
\end{proof}

We will use the following easy observation twice in the sequel.
\begin{lemma}\label{Lemm:JonqEven}
Let $\chi=(d;m_0,\dots,m_r)$ be the homaloidal type of a birational transformation of $\p^2$,  and let us assume that $d \ge 2$ and that $m_0\ge \dots \ge m_r\ge 1$. If $m_0+m_r=d$, then
$\chi = (d ; d-1, \underbrace{1, \ldots,1}_{2d-2})$
is the homaloidal type of a Jonqui\`eres element.
In particular, we remark for later use that $r=2d-2$ is even.
\end{lemma}
\begin{proof}
As $m_0+m_i\le d$ for each $i\ge 1$ (Corollary~\ref{Coro:PositivityOrbit}), we have $m_1=m_2=\dots=m_{r}=d-m_0$. The second Noether equality (Lemma~\ref{Lemm:EasyWeyl}\ref{L2}) then  gives  $d^2-1=(d-m_1)^2+rm_1^2$, whence $m_1(2d-m_1(r+1))=1$, so $m_1=1$ and $r=2d-2$.
\end{proof}

\subsection{Relation between the graph of Wright and $\ZZ_{\p^2}$}

As explained in the introduction, the graph of Wright is associated to the right cosets modulo the three groups \[G_0=\Aut(\p^2),\ G_1=\Jonq_p,\ G_2=\pi^{-1} \Aut(\p^1\times \p^1) \pi,\]
given by two points $p,q\in \p^2$. Looking at the action of $\Bir(\p^2)$ on $\ZZ_{\p^2}$, we can show that $G_0,G_1,G_2$ are the subgroups of $\Bir(\p^2)$ that fix the elements $e_0$, $e_0-e_p$, $2e_0-e_p-e_q$. One can thus see the graph of Wright as a subset of $\ZZ_{\p^2}$.
Here is the announced relation between the length in the Cremona group and the distance in the graph of Wright:

\begin{lemma}\label{Lem:DistanceWright}
Let $\varphi$ be an element of $\Bir(\p^2)$. Then, the distance  between $G_0 \varphi $ and $G_0 \id$ in the graph of Wright is equal to $2\lgth(\varphi)$.
\end{lemma}

\begin{proof}
Denote by $d(x,y)$ the distance between two vertices $x,y$ of the graph of Wright.
As $G_0=\Aut(\p^2)$, we have $\lgth(\varphi)=0\Leftrightarrow \varphi\in \Aut(\p^2)\Leftrightarrow G_0=G_0\varphi \Leftrightarrow d(G_0\varphi,G_0)=0$.

We can thus assume that $d(G_0\varphi,G_0)=n>0$. This distance is equal to
the length $n$ of the smallest path
\[v_0=G_0, v_1,\ldots,v_n=G_0\varphi\]
such that $v_0,\dots,v_n$ are vertices of the graph and such that there is an edge between $v_i$ and $v_{i+1}$ for $i=0,\dots,n-1$.

For $i=0,\ldots,n$, we write $s_i\in \{0,1,2\}$ the element such that $v_i\in \mathcal{A}_{s_i}$. We then associate to the vertices elements $\varphi_0,\dots,\varphi_{n-1}\in \Bir(\p^2)$, such that $v_i=G_{s_i}\varphi_i$ and $v_{i+1}=G_{s_{i+1}}\varphi_i$, for $i=0,\dots,n-1$. For $i=1,\ldots,n-1$, we have $v_{i}=G_{s_{i}}\varphi_{i}=G_{s_{i}}\varphi_{i-1}$, so there exists $a_{i}\in G_{s_i}$ such that $\varphi_i = a_i\varphi_{i-1}$. We moreover have
$ G_0\varphi=v_n= G_0 \varphi_{n-1}$,
so there is $a_n\in G_0$ such that $\varphi=a_n\varphi_{n-1}$. Writing $a_0=\varphi_0\in G_0$, we obtain 
\[\varphi=a_n a_{n-1}\cdots a_1a_0.\]
Conversely, every such decomposition provides a path, so $d(G_0\varphi,G_0)$ is the smallest integer $n$ such that $\varphi=a_n a_{n-1}\cdots a_1a_0$, with $a_0,a_n\in G_0$ and $a_1,\ldots,a_{n-1}\in G_0\cup G_1\cup G_2$. Every decomposition of smallest length is such that two consecutive $a_i$ do not lie in the same group (otherwise we replace them by their composition and reduce the length). 

Let us now show that there always exists a decomposition of smallest length involving only $G_0$ and $G_1$. Recall that $\Aut(\p^1\times \p^1)=\Aut^{\circ}(\p^1\times \p^1)\rtimes \langle \tau\rangle$, where $\tau$ is the exchange of the two factors and $\Aut^{\circ}(\p^1\times \p^1)=\PGL_2(\k)\times \PGL_2(\k)$. As $G_2=\pi^{-1}\Aut(\p^1\times \p^1)\pi$ where  $\pi\colon \p^2\dasharrow \p^1 \times \p^1$ is the birational map induced by the projections away from $p$ and $q$,  the relations $\pi^{-1}\Aut^{\circ}(\p^1\times \p^1)\pi\subset G_1$ and $\pi^{-1}\tau \pi \in G_0$ give us the inclusion $G_2 \subset (G_0 G_1) \cap (G_1 G_0)$. We can then replace in any decomposition of smallest length an element of $G_2$ by an element of $G_0G_1$ or $G_1G_0$, and simplify one of the elements with the next or the previous element.

We have then proven  that $d(G_0\varphi,G_0)$ is the smallest integer $n=2m$ such that $\varphi=a_n a_{n-1}\cdots a_1a_0$, with $a_i\in G_0$ if $i$ is even and $a_i\in G_1$ for $i$ odd. This yields $d(G_0\varphi,G_0)=2\lgth(\varphi)$.
\end{proof}

\section{The algorithm that computes the length and the proof of Theorem~\ref{TheMainTheorem}}

In this section, we give the proof of Theorem~\ref{TheMainTheorem}, by first working in the infinite Weyl group introduced in Section~\ref{Sec:Reminders} (in particular in $\S\ref{SubSec:InfiniteWeylgroup}$) and get the analogue of Theorem~\ref{TheMainTheorem} in $\Weyl$, namely Proposition~\ref{Prop:AlgoWeyl}. We then show (in Section~\ref{SubSec:FromWeyltoBir}) that the algorithm given in $\Weyl$ can actually be applied in $\Bir(\p^2)$.

\subsection{Degree, maximal multiplicity and comultiplicity}

\begin{definition}  \label{definition:maximal-multiplicity-and-comultiplicity}
Let $a\in \ZZ_{\p^2}$. We define the \emph{maximal multiplicity of $a$} to be \[m_{\max}(a)=\max\{m_q(a)\mid q\in \B(\p^2)\}\] and say that $a$ \emph{has maximal multiplicity at $q \in  \B(\p^2)$} if $m_q(a) = m_{\max}(a)$.

We define the \emph{comultiplicity} of $a$ to be $\comult a =\deg a -m_{\max}(a)$. 

Following the spirit of Definition~\ref{Defi:DegreeMult} (see also Remark~\ref{Remark:relation-between-mutliplicities-in-Z-and-in-the-Weyl-group})  the maximal multiplicity and comultiplicity of an element $f \in \Weyl$ are defined by
\[ m_{\max}(f)= m_{\max}(f^{-1}(e_0)) \text{ and } \comult f = \comult f^{-1}(e_0) .\]
\end{definition}

\begin{lemma}  \label{Lemm:EasyMultCoMult}
Let $a \in \Weyl (e_0) $. 
\begin{enumerate}
\item\label{EasyMultCoMult1}
If $\deg a =1$, then $a=e_0$, $m_{\max}(a)=0$ and $\comult a =1$.
\item\label{EasyMultCoMult2}
If $\deg a >1$, then $1\le m_{\max}(a)\le \deg a -1$ and $1\le \comult a \le \deg a-1$. 
\item\label{EasyMultCoMult3}
$\deg a =2\Leftrightarrow m_{\max}(a)=1$.
\item\label{EasyMultCoMult4}
$\comult a =1\Leftrightarrow a=\varphi(e_0)\mbox{ for some Jonqui\`eres }\varphi\in \Weyl$.
\end{enumerate}
\end{lemma}

\begin{proof}
If $\deg a =1$, then $a=e_0$ (Corollary~\ref{Coro:PositivityOrbit}), so $m_{\max}(a)=0$ and $\comult a =1$.

If $\deg a > 1$, then $1 \le m_{\max}(a)\le \deg a -1$ follows from Noether equalities and positivity of multiplicities, see Corollary~\ref{Coro:PositivityOrbit}. This yields $1\le \comult a \le \deg a -1$. Moreover, we have $m_{\max}(a)=1$ if and only if $\deg a =2$ (again by Corollary~\ref{Coro:PositivityOrbit}), and $\comult a =1$ if and only if $a=\varphi(e_0)$, where $\varphi$ is Jonqui\`eres (Lemma~\ref{LemmJvarphiLin}\ref{LemmJvarphiLin2}).
\end{proof}

We will often apply quadratic maps in the sequel, and need the following basic lemma.

\begin{lemma}\label{Lemm:sigmadeg}
Let $p_1,p_2,p_3\in \B(\p^2)$ be three distinct points, let $\sigma_{p_1,p_2,p_3}\in \Weyl$ be as in Definition~$\ref{Defi:sigmap1p2p3}$ and let $a\in \ZZ_{\p^2}$. 
Then the following hold:
\begin{enumerate}
\item\label{sigma123aegal}
$\sigma_{p_1,p_2,p_3}(a)=a \Leftrightarrow \deg \sigma_{p_1,p_2,p_3}(a) = \deg a \Leftrightarrow \deg(a)=m_{p_1}(a)+m_{p_2}(a)+m_{p_3}(a)$;
\item\label{sigma123degsmall}
$\deg(\sigma_{p_1,p_2,p_3}(a))<\deg(a)\Leftrightarrow \deg(a)<m_{p_1}(a)+m_{p_2}(a)+m_{p_3}(a)$;
\item\label{sigma123comult}
$\deg(\sigma_{p_1,p_2,p_3}(a))<\deg(a)\Rightarrow \comult(\sigma_{p_1,p_2,p_3}(a))\le \comult(a) .$
\end{enumerate}
\end{lemma}

\begin{proof}Writing $\xi=e_0-e_{p_1}-e_{p_2}-e_{p_3}$, we prove that $\sigma_{p_1,p_2,p_3}(v)=v+ (\xi\cdot v)\cdot \xi$ for each $v\in \ZZ_{\p^2}$. As $v\mapsto (\xi\cdot v)\cdot \xi$ is $\Z$-linear, it suffices to check this for $v=e_0$ and $v=e_q$, $q\in \B(\p^2)$, and this follows directly from the definition given in Definition~$\ref{Defi:sigmap1p2p3}$. We find
\[\deg \sigma_{p_1,p_2,p_3}(a) = e_0 \cdot (a + (\xi\cdot a)\cdot \xi)=\deg(a)+\xi\cdot a=2\deg(a)-m_{p_1}(a)-m_{p_2}(a)-m_{p_3}(a).\]
Hence, $\deg \sigma_{p_1,p_2,p_3}(a) = \deg a $ if and only if $a \cdot \xi=0$, which is equivalent to $\sigma_{p_1,p_2,p_3}(a)=a$. This yields \ref{sigma123aegal}. Assertion~\ref{sigma123degsmall} also follows from the above equalities. To prove \ref{sigma123comult}, we write $b=\sigma_{p_1,p_2,p_3}(a)=a-n  \xi$ where $n = -(\xi \cdot a) >0$ and choose a point $q\in \B(\p^2)$ where $a$ has maximal multiplicity.
We have
$\comult (a) = \deg (a) - m_q (a) = a \cdot (e_0 - e_q)  = b \cdot (e_0 -e_q)+n(\xi\cdot(e_0-e_q))$. As $ b \cdot (e_0 - e_q) = \deg (b) - m_q (b) \geq \comult (b)$, it suffices to observe that $\xi \cdot (e_0-e_q)\in \{0,1\}$.
\end{proof}

\begin{corollary}\label{Cor:sigmadegfix}
If $p_1,p_2,p_3\in \B(\p^2)$ are three distinct points and $a\in \ZZ_{\p^2}$ satisfies $ \deg(a)=m_{p_1}(a)+m_{p_2}(a)+m_{p_3}(a)$, then $\iota_{p_1,\{p_2,p_3\}}(a)=\tau(a)$, where $\tau\in \Sym_{\p^2}$ is the permutation of $p_2$ and $p_3$.
\end{corollary}
\begin{proof}
Follows from Lemma~\ref{Lemm:sigmadeg} and from the equality $\iota_{p_1,\{p_2,p_3\}}=\tau\circ \sigma_{p_1,p_2,p_3}$.
\end{proof}

The following result is quite old, and was first showed by Max Noether. We give here a proof inspired by \cite[Proposition 2.6.4]{Alberich}.

\begin{lemma}   \label{Lemma:Hudsonalgorithm}
Let $a\in \ZZ_{\p^2}$ be such that $a^2=1$, $\omega(a)=-3$, $\deg(a)\ge 2$ and $m_q(a)\ge 0$ for each $q\in \B(\p^2)$. Then, there exist three distinct points $p_1,p_2,p_3\in \Base(a)$ such that
\[\begin{array}{l}\sum_{i=1}^3 m_{p_i}(a)>\deg(a)\ \ (\text{\emph{Noether inequality}}).\end{array}\]
Moreover, for all $p_1,p_2,p_3$ as above, we have $\deg \sigma_{p_1,p_2,p_3}(a) < \deg a $.
\end{lemma}

\begin{proof}
We write $a=de_0-\sum_{i=1}^r m_i q_i$ where $d=\deg a \ge 2$, $p_1,\dots,p_r\in \B(\p^2)$ are distinct points and $m_i=m_{q_i}(a)$ for each $i$, and $m_1\ge m_2\ge \dots \ge m_r\ge 0$. The fact that $a^2=1$ and $\omega(a)=-3$ yield $\sum m_i^2=d^2-1$, $\sum m_i=3(d-1)$. This implies that $m_i< d$ for each $i$, and thus that $r\ge 3$ and $m_3>0$. We then compute
\[\begin{array}{l}(d-1)(3m_3-(d+1))=m_3(\sum m_i)-\sum m_i^2=  
\sum
m_i(m_3-m_i)\ge  \sum_{i=1}^2 m_i(m_3-m_i).\end{array}\]
Adding $(d-1)(m_1+m_2-2m_3)$ to both sides, we get
\[\begin{array}{l}(d-1)(m_1+m_2+m_3-(d+1))\ge (m_1-m_3)(d-1-m_1)+(m_2-m_3)(d-1-m_2).\end{array}\]
The right hand side being non-negative, we obtain $m_1+m_2+m_3>d$, as expected.
 The last part follows from Lemma~\ref{Lemm:sigmadeg}\ref{sigma123degsmall}.
\end{proof}

\begin{corollary}\label{Coro:DecrMaxMult}
Let $a\in \Weyl(e_0)$ be such that $\deg(a)\ge 2$ and let $p\in \B(\p^2)$ be a point of maximal multiplicity of $a$.
Then, there exists $\varphi\in \J_p$ such that $\deg(\varphi(a))<\deg(a)$.
\end{corollary}

\begin{proof}
By Lemma~\ref{Lemma:Hudsonalgorithm}, there exist three distinct points  $p_1,p_2,p_3\in \Base(a)$ such that $\sum_{i=1}^3 m_{p_i}(a)>\deg(a)$. As $p$ is a point of maximal multiplicity, we can assume $p=p_1$. We then choose $\varphi=\sigma_{p,p_2,p_3}\in \J_p$, which satisfies $\deg(\varphi(a))<\deg(a)$ (Lemma~\ref{Lemma:Hudsonalgorithm}).
\end{proof}

\begin{algorithm}[Hudson Test] \label{Algo:Hudsonalgo}
Lemma~\ref{Lemma:Hudsonalgorithm} yields the following algorithm, that decides whether an element of $\ZZ_{\p^2}$ belongs to $\Weyl(e_0)$ or not. If $a$ belongs to $\Weyl(e_0)$, one first needs to have $a^2=1$, $\omega(a)=-3$ (Noether equalities). If $\deg a =1$, then $a=e_0\in \Weyl(e_0)$. Otherwise, one needs to have $\deg a \ge 2$ and $m_q(a)\ge 0$ for each $q\in \B(\p^2)$
(Corollary~\ref{Coro:PositivityOrbit}).
Then one can apply Lemma~\ref{Lemma:Hudsonalgorithm} to obtain an element $a'\in \Weyl(a)$ of smaller degree, satisfying again the Noether equalities. If $\deg a' \ge 2$ and the multiplicities are again non-negative, one again applies the corollary and decreases the degree. At some moment, either we obtain $e_0$, and then $a\in \Weyl(e_0)$, or we get some negative degree or multiplicity, and then $a\notin \Weyl(e_0)$
(by Corollary~\ref{Coro:PositivityOrbit}).
\end{algorithm}

\begin{example}
Take
$12$ different points $q_1,\dots,q_{12}\in \B(\p^2)$.

The element $a=-7e_0+\sum_{i=1}^{12} 2e_{q_i}\in \ZZ_{\p^2}$ satisfies the Noether equalities, but has negative degree (and negative multiplicities), hence does not belong to $\Weyl(e_0)$.

The element
$a'=3e_0+e_{q_1}-\sum_{i=4}^{10} e_{q_i}$
satisfies the Noether equalities, has positive degree but has one negative multiplicity, hence does not belong to $\Weyl(e_0)$.

By Definition~\ref{Defi:sigmap1p2p3}, the element $\sigma_{q_1,q_2,q_3}(a')$ is equal to $a'' = 7e_0-3e_{q_1}-4e_{q_2}-4e_{q_3}-\sum_{i=4}^{10} e_{q_i}$. This element
satisfies the Noether equalities, has positive degree and non-negative multiplicities but does not belong to $\Weyl(e_0)$, as $a'$ does not.
\end{example}

\subsection{Predecessors}

\begin{lemma}\label{Lemm:Permutationdecreases}
Let $a\in \Weyl(e_0)$, $g\in \Weyl$, and $p_1,p_2\in \B(\p^2)$ be two distinct points. Denote by $\tau\in \Sym_{\p^2}$ the transposition that exchanges $p_1$ and $p_2$. The comparison of the two elements $b=g(a)$ and $c=\tau\circ g\circ \tau(a)$  of $\Weyl (e_0)$ is given as follows:
\begin{equation}\label{EquationProductNeg}\tag{$\clubsuit$}
\deg(c)-\deg(b) = (m_{p_1}(a)-m_{p_2}(a))(m_{p_1}(g)-m_{p_2}(g)). \end{equation} 
Moreover, the following hold:

\begin{enumerate}
\item\label{SameMultEqualSym}
$\deg(b)=\deg(c)\Leftrightarrow b\equiv_{\Sym_{\p^2}}\!c$;
\item\label{NotSameMultDecrease}
$\deg(b)>\deg(c)\Rightarrow \comult(b)\ge \comult(c)$;
\item\label{NotSameMultIncrease}
$\deg(b)<\deg(c)\Rightarrow \comult(b)\le \comult(c)$.
\end{enumerate}
\end{lemma}

\begin{proof}
For all $\alpha\in \ZZ_{\p^2}$ we have $ \tau(\alpha)-\alpha  = ( m_{p_1}(\alpha)  -m_{p_2} (\alpha) ) (e_{p_1} - e_{p_2} )$. This yields
\begin{equation}\label{EquationGeneral}\tag{\ensuremath{\spadesuit}}
( \tau(\alpha)-\alpha)\cdot \beta=( m_{p_1}(\alpha)  -m_{p_2} (\alpha) ) (m_{p_1}(\beta) - m_{p_2}(\beta) )\text{ for all } \alpha,\beta\in \ZZ_{\p^2}.
\end{equation}

We then write  $\Lambda=g^{-1}(e_0)$, and obtain $m_{p_i}(\Lambda)=m_{p_i}(g)$ for $i=1,2$. As $\deg(c)=\deg(g\circ \tau(a))$, we get
\[\begin{array}{rcl}
\deg(c)-\deg(b) &=& ( g\circ \tau(a)-g(a))\cdot e_0= (\tau(a)-a)\cdot g^{-1}(e_0)=(\tau(a)-a)\cdot \Lambda\\
&\stackrel{\ref{EquationGeneral}}{=}&( m_{p_1}(a)  -m_{p_2} (a) )(m_{p_1}(\Lambda)-m_{p_2}(\Lambda))\\
&=&( m_{p_1}(a)  -m_{p_2} (a) )(m_{p_1}(g)-m_{p_2}(g)),\end{array}\]
which achieves the proof of \eqref{EquationProductNeg}.

\ref{SameMultEqualSym}: If $b\equiv_{\Sym_{\p^2}}\!c$, then $c\cdot e_0=b\cdot e_0$, as $e_0$ is fixed by $\Sym_{\p^2}$, i.e.~$\deg(c)=\deg(b)$. Conversely, we suppose that $\deg(c)=\deg(b)$, which implies that $m_{p_1}(a)=m_{p_2}(a)$ or $m_{p_1}(g)=m_{p_2}(g)$ (by \ref{EquationProductNeg}), and 
we want to prove that $b\equiv_{\Sym_{\p^2}}\!c$. If $m_{p_1}(a)=m_{p_2}(a)$, then $\tau(a)=a$, which yields $c=\tau \circ g \circ \tau(a)=\tau (  g(a))=\tau(b)$. If $m_{p_1}(g)=m_{p_2}(g)$, then $\tau (\Lambda)= \Lambda$, i.e.~$(g\circ\tau)^{-1}(e_0)=g^{-1}(e_0)$. There exists thus $\beta\in \Sym_{\p^2}$ such that  $\beta \circ g=g \circ \tau$ (Corollary~\ref{Coro:ExistenceAlpha}). This yields $c=\tau \circ g  \circ \tau(a) = \tau \circ \beta \circ g(a)=(\tau\circ\beta)(b)\in \Weyl(c)$.

\ref{NotSameMultDecrease}: Assume that $\deg(b)>\deg(c)$. Up to exchanging $p_1$ and $p_2$, we can assume that $m_{p_1}(a)>m_{p_2}(a)$ and $m_{p_1}(g)<m_{p_2}(g)$ (by \ref{EquationProductNeg}).
To show that $\comult b=\comult g(a)\ge\comult \tau \circ g \circ \tau(a)=\comult c$, we denote by $q \in \B(\p^2)$ a point of maximal multiplicity of $b$ and write $\Gamma= g^{-1}(e_0-e_{q})$. This yields
\[\begin{array}{lllllll}
\comult b  &=& (e_0-e_{q})\cdot g(a) &=& g^{-1}(e_0-e_q)\cdot a  &= &\Gamma \cdot a, \\
\comult c   &\le&   \big(e_0-\tau(e_{q}) \big) \cdot c 
&=& (\tau \circ g \circ \tau)^{-1} \big( e_0 - \tau(e_{q}) \big) \cdot a &=&  \tau(\Gamma)\cdot a.
\end{array}\]
We then have
$\comult c-\comult b\le( \tau(\Gamma) -\Gamma) \cdot a  \stackrel{\ref{EquationGeneral}}{=}  ( m_{p_1}(\Gamma)-m_{p_2}(\Gamma) ) \cdot \big( m_{p_1}(a)-m_{p_2}(a) ) ,$
hence to prove the inequality $\comult b \ge \comult c$, it remains to see that $m_{p_1}(\Gamma)-m_{p_2}(\Gamma)>0$ is impossible (as $m_{p_1}(a)>m_{p_2}(a)$). Indeed, this would yield (as $m_{p_1}(\Lambda)<m_{p_2}(\Lambda)$)
\[ ( \tau(\Gamma)-\Gamma ) \cdot \Lambda \stackrel{\ref{EquationGeneral}}{=} ( m_{p_1}(\Gamma)-m_{p_2}(\Gamma) ) \cdot ( m_{p_1}(\Lambda)-m_{p_2}(\Lambda) ) < 0,\]
which implies that $\tau(\Gamma)\cdot \Lambda < \Gamma\cdot \Lambda=g^{-1}(e_0-e_{q})\cdot g^{-1}(e_0)=(e_0-e_{q})\cdot e_0=1$. This is impossible as $\tau(\Gamma)\in \Weyl(e_0-e_q)$ and $\Lambda\in \Weyl(e_0)$ (see Corollary~\ref{Cor:IntersectionTwoOrbits}\ref{Ge0e1}).

Assertion \ref{NotSameMultIncrease} follows from \ref{NotSameMultDecrease} by replacing $g$ with $\tau\circ g \circ\tau $, which exchanges $b$ and~$c$.
\end{proof}

\begin{definition} \label{DefiPredeWinfe0}
Let $a \in \Weyl(e_0)$. A \emph{predecessor of $a$} is an element of \[\{\varphi(a) \mid \varphi \in \Weyl \text{\emph{ is a Jonqui\`eres element}}\}\] of minimal degree.
\end{definition}

If $a \in \Weyl(e_0)$ has degree at least $2$ (i.e.~$a \neq e_0$, see Corollary~\ref{Coro:PositivityOrbit}), it follows from Corollary~\ref{Coro:DecrMaxMult}  that the predecessors of $a$ have degree smaller than $a$. The following fundamental lemma establishes (among
others) 
the uniqueness of a predecessor modulo $\Sym_{\p^2}$, and gives an explicit way to compute a predecessor.

\begin{lemma}   \label{Lemm:AlgoFirstStepS}
Let $a \in \Weyl(e_0)$ be an element of degree $d >1$.
Denote by $(d ; m_0, \ldots, m_r)$ its homaloidal type, where we may assume that $m_0 \geq \cdots \geq m_r \geq 1$. Setting $m_i=0$ for $i > r$, we obtain an infinite non-increasing sequence $(m_i)_{i \ge 0}$.
We will say that an ordering $p_0, \ldots,p_r$ of the base-points of $a$ is \emph{ non-increasing} if the $($finite$)$ sequence of multiplicities $i \mapsto m_{p_i} (a)$ is non-increasing. Equivalently, this means that $m_{p_i}(a) = m_i$ for $0 \leq i \leq r$. Then, the following assertions are satisfied:
\begin{enumerate}

\item \label{atleastonei}
The set $\mathcal{S}=\{s\ge 1\mid m_{0}+m_{2s-1}+m_{2s}\ge d\ge m_{0}+m_{2s+1}+m_{2s+2}\}$ is a non-empty subset of consecutive integers of the interval $[1;\frac{r}{2}]\subseteq \R$ $($whence $r\ge 2)$.

\item\label{ModuloSymP2Onepred}
All predecessors of $a$ are equal modulo $\Sym_{\p^2}$.
\item \label{eachsinS}
Choose any non-increasing ordering $p_0, \ldots,p_r$ of the base-points of $a$. Then, for each integer $s\in [1;\frac{r}{2}]$ and each $\alpha\in \Sym_{\p^2}$, the element $\alpha\circ \iota_{p_0,\{p_1,\dots,p_{2s}\}}(a)$ is a predecessor of $a$ if and only if $s\in \mathcal{S}$.
\item  \label{AllPredLikeThisBaseCont}
If $\varphi\in \Weyl$ is a Jonqui\`eres element
such that $\varphi(a)$ is a predecessor of $a$, then $\varphi$ is equal to $\alpha\circ \iota_{p_0,\{p_1,\dots,p_{2s}\}}$
for some choice of a non-increasing ordering $p_0, \ldots,p_r$ of the base-points of $a$, and for some $\alpha\in \Sym_{\p^2}$, $s\in \mathcal{S}$. In particular, we have $\Base(\varphi)\subseteq \Base(a)$.
\item\label{CoMult}
If $\varphi\in \Weyl$ is a Jonqui\`eres element, and $c$ is a predecessor of $a$, then $\comult(\varphi(a))\ge \comult(c)$. In particular, we have $\comult(a)\ge \comult(c)$.
\end{enumerate}
\end{lemma}

\begin{proof}
We prove three assertions:

{\bf (I): Proof of \ref{atleastonei}}. 
The inequality $r \geq 2$ follows from Lemma~$\ref{Lemma:Hudsonalgorithm}$.
We now show that the non-increasing sequence $ i\mapsto u_i$ defined by $u_i:= m_0+m_{2i-1}+m_{2i}$ for $i\ge 1$  satisfies
\[\mathrm{(Ia)}: u_1>d, \text{ and } \mathrm{(Ib)}: u_i<d\text{ for }i>r/2.\]

$\mathrm{(Ia)}:$ The inequality $u_1=m_0+m_1+m_2>d$ is Noether inequality (Lemma~\ref{Lemma:Hudsonalgorithm}).

$\mathrm{(Ib)}:$ The inequality $2i>r$ yields $m_{2i}=0$, which gives $u_i=m_0+m_{2i-1}\le m_0+m_{2i-2}\le \dots \le m_0+m_1\le d$ (Corollary~\ref{Coro:PositivityOrbit}). It remains to observe that $m_1=m_2=\dots=m_{2i-1}=d-m_0$ and $r=2i-1$ is impossible (Lemma~\ref{Lemm:JonqEven}).
Therefore, $\mathrm{(Ia)}$ and $\mathrm{(Ib)}$ are proven. These assertions imply \ref{atleastonei} because of the equality $\mathcal{S}=\{s \ge 1\mid u_s \geq d \geq u_{s+1}  \}.$

{\bf (II): For any non-increasing ordering $p_0, \ldots, p_r$ of the base-points of $a$, for any $\alpha\in \Sym_{\p^2}$, and for any $s \in {\mathcal S}$, the elements  $\alpha\circ \iota_{p_0,\{p_1,\dots,p_{2s}\}}(a)$ are all equal modulo $\Sym_{\p^2}$.}

Firstly, we fix a non-increasing ordering $p_0, \ldots, p_r$ of the base-points of $a$.
Define $\iota_s=\iota_{p_0,\{p_1,\dots,p_{2s}\}}$ and $c_s=\iota_s(a)$ for each integer $s\in [1;\frac{r}{2}]$, and show that $c_s\equiv_{\Sym_{\p^2}}\!c_{s'}$ for all $s,s'\in \mathcal{S}$. By \ref{atleastonei}, it suffices to prove this in the case where $s'=s+1$. The fact that $s,s+1\in \mathcal{S}$ implies that $d=u_{s+1}=m_0+m_{2s+1}+m_{2s+2}$, which means that $(e_0-e_{p_0}-e_{p_{2s+1}}-e_{p_{2s+2}})\cdot a=0$ and implies that $\iota_{p_0,\{p_{2s+1},p_{2s+2}\}}(a)=\tau(a)$ where $\tau\in \Sym_{\p^2}$ is the permutation of $p_{2s+1}$ and $p_{2s+2}$ (Corollary~\ref{Cor:sigmadegfix}). We moreover observe that $\iota_{s}=\iota_{p_0,\{p_1,\dots,p_{2s}\}}$ fixes $p_{2s+1}$ and $p_{2s+2}$, and thus commutes with $\tau$. This yields $c_{s'} = \iota_{s+1}(a) = \iota_s (\iota_{p_0, \{p_{2s+1}, p_{2s+2} \}} (a)) = \iota_{s} (\tau(a)) = \tau(\iota_{s}(a)) = \tau(c_s)$ as desired.

Secondly, we observe that the class of $\iota_s(a)$ does not depend on the
non-increasing ordering $p_0, \ldots, p_r$ of the base-points of $a$.
Indeed, two different orderings only differ by some product of
transpositions of two points having the same multiplicity.
The result then follows from Lemma~\ref{Lemm:Permutationdecreases}.

{\bf (III): For each Jonqui\`eres element $\varphi\in \Weyl$, one of the following holds:
\begin{enumerate}
\item[(A)]
$\varphi=\alpha\circ \iota_{p_0,\{p_1,\dots,p_{2s}\}}$ for some non-increasing ordering $p_0, \ldots,p_r$ of the base-points of $a$, for some $\alpha\in \Sym_{\p^2}$, and some $s \in{\mathcal S}$.
\item[(B)]
There exists a Jonqui\`eres element $\varphi' \in \Weyl$ such that $\deg(\varphi'(a))<\deg(\varphi(a))$ and $\comult(\varphi'(a))\le \comult(\varphi(a))$.
\end{enumerate}
Moreover, for any non-increasing ordering $p_0, \ldots,p_r$ of the base-points of $a$, for any $\alpha\in \Sym_{\p^2}$, and any integer $s \in [ 1 ; \frac{r}{2} ] \setminus {\mathcal S}$, 
if we set $\varphi=\alpha\circ \iota_{p_0,\{p_1,\dots,p_{2s}\}}$, then
the assertion (B) is satisfied.
}

We fix a Jonqui\`eres element $\varphi\in \Weyl$ and choose $l+1$ distinct points $p_0,\dots,p_l\in \B(\p^2)$ such that $\{p_0,\dots,p_l\}=\Base(\varphi)
\cup \Base(a)$ (whence $l\ge r$).

Suppose that there exist $i,j\in \{0,\dots,l\}$ such that $(m_{p_i}(a)-m_{p_j}(a))(m_{p_i}(\varphi)-m_{p_j}(\varphi))<0$. We denote by $\tau\in \Sym_{\p^2}$ the permutation of $p_i$ and $p_j$, write $\varphi'=\tau\circ\varphi\circ\tau$ (which is again a Jonqui\`eres element, by Lemma~\ref{Lemm:JoWinfty}), and get 
$\deg(\varphi'(a))-\deg(\varphi(a))\stackrel{\text{Lemma~\ref{Lemm:Permutationdecreases}.\ref{EquationProductNeg}}}{=}(m_{p_i}(a)-m_{p_j}(a))(m_{p_i}(\varphi)-m_{p_j}(\varphi))<0$. Moreover, Lemma~\ref{Lemm:Permutationdecreases}\ref{NotSameMultIncrease} yields $\comult(\varphi'(a))\le \comult(\varphi(a))$. We are then in case (B).

We can therefore assume, after reordering the points $p_i$, that
\[m_{p_0}(a)\ge  m_{p_1}(a)\ge \dots \ge m_{p_l}(a),\quad m_{p_0}(\varphi)\ge  m_{p_1}(\varphi)\ge \dots \ge m_{p_l}(\varphi).\]
In particular, $\Base(a)=\{p_k \mid k\in \{1,\dots,l\}\text{ and }m_{p_k}(a)>0\}=\{p_0,\dots,p_r\}$,
and the base-points $p_0, \ldots, p_r$ of $a$ are given in non-increasing order.
Moreover, $\varphi$ has maximal multiplicity at $p_0$. By Corollary~\ref{Lemm:EasyMultCoMult}\ref{EasyMultCoMult2}, this implies that $1\le m_{p_0} (\varphi) \le \deg(\varphi)-1$.
 Since $\varphi$ is a Jonqui\`eres element of $\Weyl$, it has a point of multiplicity $\deg(\varphi)-1$, so that we have $m_{p_0} (\varphi) =\deg(\varphi)-1$. 
There exists thus $\alpha\in \Sym_{\p^2}$ such that $\varphi=  \alpha \circ \iota_{p_0,\Delta}$, where $\Delta=\Base(\varphi)\setminus\{p_0\}=\{p_k \mid k\in \{1,\dots,l\}\text{ and }m_{p_k}(\varphi)>0\}$ has even cardinality (Lemma~\ref{LemmJvarphiLin}\ref{LemmJvarphiLin2}).
It follows that the set $\Delta$ is equal to $\{p_1,\dots,p_{2s}\}$ for some $s\ge 1$. If $s\in \mathcal{S}$, we are in case (A). 

We now assume that $s\not\in \mathcal{S}$, and show that we are in case (B), with $\varphi'=\varphi_{p_0,\{p_1,\dots,p_{2s'}\}}$, for some $s'\in \{s\pm 1\}$. We then only need to show that $\deg(c_{s'})<\deg(c_{s})$ and $\comult(c_{s'})\le \comult(c_s)$, where $c_s,c_{s'}$ are defined as in the proof of (II).
As $s\not\in \mathcal{S}$, we have either $u_s<d$ or $u_{s+1}>d$ (where $u$ is the sequence defined above, in the proof of
$(\mathrm{I})$).

If $u_s<d$, then $s>1$ by $(\mathrm{Ia})$, and we choose $s'=s-1\ge 1$. As $c_{s'}=\iota_{p_0,\{p_{2s-1},p_{2s}\}}(c_s)$ is equal to
$\sigma_{p_0,p_{2s-1},p_{2s}}(c_s)$ modulo $\Sym_{\p^2}$ (Definition~$\ref{Defi:sigmap1p2p3}$), it suffices to prove that $(e_0-e_{p_0}-e_{p_{2s-1}}-e_{p_{2s}})\cdot c_s<0$ (Lemma~\ref{Lemm:sigmadeg}), which follows from
\[\begin{array}{rcl}
(e_0-e_{p_0}-e_{p_{2s-1}}-e_{p_{2s}})\cdot c_s&=&\iota_{s}(e_0-e_{p_0}-e_{p_{2s-1}}-e_{p_{2s}})\cdot a\\
&=&-(e_0-e_{p_0}-e_{p_{2s-1}}-e_{p_{2s}})\cdot a=u_s-d.\end{array}\]

If $u_{s+1}>d$, we choose $s'=s+1$, which belongs to $[1;\frac{r}{2}]\subseteq[1;\frac{l}{2}]$ by $(\mathrm{Ib})$. As before, it suffices to check that $(e_0-e_{p_0}-e_{p_{2s+1}}-e_{p_{st+2}})\cdot c_s<0$, which follows from
\[\begin{array}{rcl}
(e_0-e_{p_0}-e_{p_{2s+1}}-e_{p_{2s+2}})\cdot c_s & = & \iota_{s}(e_0-e_{p_0}-e_{p_{2s+1}}-e_{p_{2s+2}})\cdot a \\
 & = &(e_0-e_{p_0}-e_{p_{2s-1}}-e_{p_{2s}})\cdot a=d-u_{s+1}.\end{array}\] 
 This achieves the proof of $\mathrm{(III)}$, which gives \ref{AllPredLikeThisBaseCont}. 
Together with  $\mathrm{(II)}$ and since $a$ admits at least one predecessor, this also gives \ref{ModuloSymP2Onepred} and \ref{eachsinS}.
It remains to prove \ref{CoMult}. We do it by induction on $\deg(\varphi(a))$.
The minimal case is when $\varphi(a)$ is a predecessor of $a$,  so
$\varphi(a)\equiv_{\Sym_{\p^2}}c$ by~\ref{ModuloSymP2Onepred}, whence $\comult(\varphi(a))=\comult(c)$. If $\deg(\varphi(a))>\deg(c)$, then,
by $\mathrm{(III)}$, there exists  a Jonqui\`eres element $\varphi'\in \Weyl$ such that $\deg(\varphi(a))> \deg(\varphi'(a))$ and $\comult(\varphi(a))\ge \comult(\varphi'(a))$. 
Since we have $\comult(\varphi'(a))\ge \comult(c)$ by the induction hypothesis, the result follows.
\end{proof}

Here is an example where the set $\mathcal{S}$ of Lemma~\ref{Lemm:AlgoFirstStepS} contains $2$ elements. However, we do not know  whether $\mathcal{S}$ can contain more than $2$ elements or not.

\begin{example}
Take $6$ different points $p_0,\dots, p_5 \in \B(\p^2)$. By Definition~\ref{Defi:sigmap1p2p3}, the element $\sigma_{p_0,p_1,p_2} \circ \sigma_{p_3,p_4,p_5} (e_0)  \in \Weyl(e_0)$ is equal to $a = 4 e_0 -2 e_{p_0} -2 e_{p_1} -2 e_{p_2}  -e_{p_3}-e_{p_4}-e_{p_5}$. Its homaloidal type is $(4 ; 2,2,2,1,1,1)$ and the set $\mathcal{S}$ of Lemma~\ref{Lemm:AlgoFirstStepS}\ref{atleastonei} is equal to $\mathcal{S} = \{ 1,2 \}$.
By Lemma~\ref{Lemm:AlgoFirstStepS}\ref{eachsinS}, the elements $ \iota_{p_0,\{p_1,\dots,p_{2s}\}}(a)$, $s \in \mathcal{S}$, are predecessors of $a$. By Definition~\ref{Defi:Iota}, we get
\[ \iota_{p_0,\{p_1,p_2 \}}(a)= 2 e_0-e_{p_3}-e_{p_4}-e_{p_5} = \iota_{p_0,\{p_1,\dots,p_{4}\}}(a). \]
As claimed by Lemma~\ref{Lemm:AlgoFirstStepS}\ref{ModuloSymP2Onepred}, these two predecessors are equal modulo $\Sym_{\p^2}$ (since they are even equal).
\end{example}

\begin{corollary}\label{Coro:UniquenessPredUpSym}
All predecessors of an element $a \in \Weyl(e_0)$ are equal modulo $\Sym_{\p^2}$ and only depend on the class of $a$ modulo $\Sym_{\p^2}$.
\end{corollary}

\begin{proof}
We first observe that the predecessors of $a \in \Weyl(e_0)$ are equal modulo $\Sym_{\p^2}$. If $a=e_0$, this is because $e_0$ is the only predecessor of $a$; otherwise, it follows from Lemma~\ref{Lemm:AlgoFirstStepS}\ref{ModuloSymP2Onepred}. We then observe that the sets of predecessors of $a$ and of $\alpha(a)$ are equal, for each $\alpha\in \Sym_{\p^2}$.
\end{proof}

\begin{remark}
By Remark~\ref{Rem:EqualModSymBij}, a homaloidal type may be identified with an element of $\Weyl(e_0) / \Sym_{\p^2}$.  It follows from Corollary~\ref{Coro:UniquenessPredUpSym} that for each homaloidal type $\chi  \in \Weyl(e_0) / \Sym_{\p^2}$, one can define its (unique) predecessor $\chi_1 \in \Weyl(e_0) / \Sym_{\p^2}$ in the following way: 
Choose any representative $a \in  \Weyl(e_0)$ of the coset $\chi$, then $\chi_1$ is defined as the equivalence class modulo $\Sym_{\p^2}$ in $\Weyl(e_0)$ of any predecessor
$a_1\in  \Weyl(e_0)$ of $a$.

Assume that $\chi = (d ; m_0, \ldots, m_r)$ with $d\ge 2$ and $m_0\ge \dots \ge m_r \geq 1$. Choose an integer $s\in [1,\frac{r}{2}]$  such that $m_{0}+m_{2s-1}+m_{2s}\ge d\ge m_{0}+m_{2s+1}+m_{2s+2}$ (this is doable thanks to Lemma~\ref{Lemm:AlgoFirstStepS}\ref{atleastonei}). Then 
\[\chi_1=(d-\varepsilon;m_0-\varepsilon, d-m_0-m_1, \dots, d-m_0-m_{2s},m_{2s+1}, \dots , m_r) \] where $\varepsilon=\sum_{i=1}^s (m_0+m_{2i-1}+m_{2i}-d)$ is positive (the form of $\chi_1$ follows from Lem\-ma~\ref{Lemm:AlgoFirstStepS}\ref{eachsinS}).  
Of course, as usual, we can remove the multiplicities which are zero
 and order the remaining ones in a non-decreasing manner.
\end{remark}

\begin{example}
To simplify the notation, a sequence of $s$ multiplicities $m$ is written $m^s$.
In the following list, the notation $\chi \predecessorLow \chi_1$ means that $\chi_1$ is the predecessor of the homaloidal type $\chi$.
\[\begin{array}{lllll}
(d;d-1,1^{2d-2}) \predecessor (1);& (4;2^3,1^3) \predecessor (2;1^3);& (38;18,13^3,12^4,6)\predecessor (23;12,8^3,7^3,6,3);\\
(16;6^5;5^3)\predecessor (12;5^3,4^4,2);  &(5;2^6)\predecessor (3;2,1^4); &(74;28,27^5,19^2,18)\predecessor (58;27,19^6,18,12).
\end{array}\]
\end{example}

\begin{lemma}  \label{Lemm:CompositionWinftyWithoutCommonBasepts}
Let $f,g\in \Weyl$ be elements such that $\Base(f)\cap \Base(g^{-1})=\emptyset$ and $\deg(g)=d>1$. Defining $h= f \circ g$, the following hold:
\begin{enumerate}
\item\label{CompWinftyWithoutCommonBasepts1}
$\deg(h)=\deg(f)\cdot \deg(g)$ and $\lvert \Base(h)\rvert=\lvert \Base(f)\rvert+\lvert \Base(g)\rvert$.
\item\label{CompWinftyWithoutCommonBasepts2}
If $\varphi\in \Weyl$ is a Jonqui\`eres element such that $\varphi(g^{-1}(e_0))$ is a predecessor of $g^{-1}(e_0)$, then $\varphi( h^{-1}(e_0))$ is a predecessor of $h^{-1}(e_0)$.
\end{enumerate}
\end{lemma}

\begin{proof}
If $\deg (f) =1$, we get $\Base(f)=\emptyset$,  $f^{-1}(e_0) = e_0$ and $h^{-1}(e_0)=g^{-1}(e_0)$, so that there is nothing to check. We can then assume that $\deg(f)=D>1$.
If $(d; m_0,\ldots,m_r)$ and $(D; \mu_{r+1}, \ldots, \mu_l)$ are the homaloidal types of $g$ and $f$ with $m_0 \ge \cdots \ge m_r \ge 1$ and $\mu_{r+1} \ge \cdots \ge \mu_l \ge 1$, choose orderings $p_0, \ldots,p_r$ and $q_{r+1}, \ldots,q_l$ of the base-points of $g$ and $f$ such that
\[\begin{array}{lll}
g^{-1}(e_0)=d e_0-\sum_{i=0}^r m_ie_{p_i}&\text{ and }&f^{-1}(e_0)=De_0-\sum_{i=r+1}^l \mu_i e_{q_i}. \end{array}\]
Since $\Base(f) \cap \Base(g^{-1}) =\emptyset $,
there exist points $p_{r+1},\dots,p_l\in \B(\p^2)\setminus \Base(g)=\B(\p^2)\setminus \{p_0,\dots,p_r\}$ such that $g^{-1}(e_{q_i})=e_{p_i}$ for $i=r+1,\dots,l$ (Lemma~\ref{Lemma:Imagepts}). We then obtain
\[\begin{array}{rcl}
h^{-1}(e_0)&=&g^{-1}(De_0-\sum_{i=r+1}^l \mu_i e_{q_i})=Dg^{-1}(e_0)-\sum_{i=r+1}^l \mu_i e_{p_i}\\
&=& Dd e_0-\sum_{i=0}^r Dm_i e_{p_i}-\sum_{i=r+1}^l \mu_i e_{p_i}.\end{array}\]
This already proves \ref{CompWinftyWithoutCommonBasepts1}. It then remains to show \ref{CompWinftyWithoutCommonBasepts2}. As $\varphi(g^{-1}(e_0))$ is a predecessor of $g^{-1}(e_0)$,
Lemma~\ref{Lemm:AlgoFirstStepS}\ref{AllPredLikeThisBaseCont} provides (up to a re-ordering of the points $p_0,\ldots,p_r$)  an element $\alpha\in \Sym_{\p^2}$ and an integer $s \in [ 1 ; \frac{r}{2} ]$ satisfying
\begin{equation} \label{Initial-s} \tag{\ensuremath{\heartsuit}} m_{0}+m_{2s-1}+m_{2s}\ge d\ge m_{0}+m_{2s+1}+m_{2s+2}
\end{equation}
(where $m_i=0$ if $i>r$) such that $\varphi=\alpha\circ \iota_{p_0,\{p_1,\dots,p_{2s}\}}$.
Since $\mu_{r+1} \le D-1$  (Corollary~\ref{Coro:PositivityOrbit}), we have
\[m_{p_0}(h)=m_0 D\ge \dots\ge m_{p_r}(h) =m_rD> m_{p_{r+1}}(h)=\mu_{r+1}\ge \dots \ge m_{p_{l}}(h)=\mu_l.\]
According to Lemma~\ref{Lemm:AlgoFirstStepS}\ref{eachsinS}
for showing that $\varphi(h^{-1}(e_0))$ is a predecessor of $h^{-1}(e_0)$ it is sufficient (and necessary) to prove that
\[ m_{p_0} (h) +m_{p_{2s-1}}(h)+m_{p_{2s}}(h) \ge Dd \ge m_{p_0}(h)+m_{p_{2s+1}}(h)+m_{p_{2s+2}}(h). \]
If $2s+2\le r$, this is just \eqref{Initial-s}  multiplied by $D$. Let us therefore now assume that $2s+2>r$, so that we have $2s+2\in \{r+1,r+2\}$.
If $2s+2=r+2$, it suffices to prove that $Dd\ge Dm_0+\mu_{r+1}+\mu_{r+2}$. This follows from the inequalities $\mu_{r+1}+\mu_{r+2}\le D$ and $m_0\le d-1$ (Corollary~\ref{Coro:PositivityOrbit}). If $2s+2=r+1$, it suffices to prove that $Dd\ge Dm_0+Dm_r+\mu_{r+1}$. If $m_0+m_r \le d-1$, this follows from $\mu_{r+1}\le D-1$.
It remains to show that the case $m_0+m_r \ge d$ can not occur.  Indeed, otherwise we would have $m_0+m_r=d$ (Corollary~\ref{Coro:PositivityOrbit}) and then $r$ should be even by Lemma~\ref{Lemm:JonqEven}. A contradiction.
\end{proof}

\begin{algorithm}[Computing the length in $\Weyl$]  \label{Algo:AlgorithInWeyl}
To each $a_0\in \Weyl(e_0)$ we can associate a sequence $a_1,a_2,\dots,$ of elements of $\Weyl(e_0)$ such that $a_i$ is a predecessor of $a_{i-1}$ for each $i\ge 1$. We then say that $a_i$ is a \emph{$i$-th predecessor} of $a_0$.

At some step $n\ge 0$, we have $a_n=e_0$ and then $a_j=e_0$ for all $j\ge n$. This provides then a finite sequence $(a_0,\dots,a_n)$ ending with $e_0$. Corollary~\ref{Coro:UniquenessPredUpSym} shows that
this
sequence is
unique modulo
$\Sym_{\p^2}$, and Lemma~\ref{Lemm:AlgoFirstStepS} provides 
an explicit way to compute it.

The fact that the number $n$ of steps is really the length in $\Weyl$ will be proven in Proposition~\ref{Prop:AlgoWeyl} below. 
\end{algorithm}

\begin{example}\label{Example:SomeSequences}
As before, a sequence of $s$ multiplicities $m$ is written $m^s$. We apply Algorithm~\ref{Algo:AlgorithInWeyl} to a list of homaloidal types:
\begin{itemize}
\item
$(11;6,5,4^2,3^2,2^2,1)\predecessor(5;3,2^3,1^3)\predecessor(2;1^3)\predecessor(1)$;
\item
$(19;7^7,4,1) \predecessor (13;5^6,4,1^2)\predecessor (8;4,3^5,1^2)\predecessor(4;3,1^6)\predecessor(1)$;
\item
$(40;18,17,14^2,12^4,3^2)\predecessor(25;12,10^3,8^2,5,3^3)\predecessor(13;8,5^2,3^6)\predecessor(5;2^6)\predecessor(3;2,1^4)\predecessor(1)$;
\item
$(38;14^6,11^2,5)\predecessor(29;13,11,10^5,5^2)\predecessor(16;6^5,5^3)\predecessor(12;5^3,4^4,2)\predecessor(7;3^4,2^3)$\\ $\predecessor(4;2^3,1^3)\predecessor(2;1^3)\predecessor(1)$;
\item
$(184;75,61^6,60,48)\predecessor(145;60,48^7,36)\predecessor(112;48,37^6,36,27)\predecessor(82;36,27^7,18)$ \\ $\predecessor(58;27,19^6,18,12)\predecessor(37;18,12^7,6)\predecessor(22;12,7^6,6,3)\predecessor(10;6,3^7)\predecessor(4;3,1^6)\predecessor(1)$.
\end{itemize}
\end{example}

\begin{example} \label{Example:SequenceBertini}
Applying Algorithm~\ref{Algo:AlgorithInWeyl} to $(17;6^8)$ yields a sequence of homaloidal types of particular interest:\[(17;6^8)\predecessor(14;6,5^6,3)\predecessor(8;3^7)\predecessor(5;2^6)\predecessor(3;2,1^4)\predecessor(1)\]
It provides all the symmetric homaloidal types (symmetric means here that all multiplicities are the same) except the simplest one, namely $(2;1^3)$ \cite[Lemma 2.5.5]{Alberich}. Note that $(17;6^8)$ and $(8;3^7)$ are the homaloidal types of classical Bertini and Geiser involutions, associated to the blow-ups of $8$, respectively $7$ general points of $\p^2$.
\end{example}

We can now give the following result, which is the analogue of Theorem~\ref{TheMainTheorem} for $\Weyl$.

\begin{proposition}\label{Prop:AlgoWeyl}
Let $n\ge 1$, let  $a_0\in \Weyl(e_0)$ and let $a_n\in \Weyl(e_0)$ be a $n$-th predecessor of $a_0$. For all Jonqui\`eres elements  $\psi_1, \ldots, \psi_n \in \Weyl$, the element $b_n=\psi_n\circ \dots \circ \psi_1(a_0)$ satisfies 
\begin{enumerate}
\item\label{DegAnBn}
$\deg(a_n)\le \deg(b_n)$;
\item\label{ComultAnBn}
$\comult(a_n)\le \comult(b_n)$;
\item\label{AnBnSymp2}
If $\deg(a_n)= \deg(b_n)$, then $a_n$ and $b_n$ are equal modulo $\Sym_{\p^2}$.
\end{enumerate}\end{proposition}

\begin{proof}
We prove the result by induction on the triples $(n,\lgth(a_0),\deg(b_1))$, ordered lexicographically, where $b_1=\psi_1(a_0)$.

If $n=1$, \ref{DegAnBn} is given by the definition of a predecessor, and \ref{ComultAnBn} and \ref{AnBnSymp2} follow respectively from Lemma~\ref{Lemm:AlgoFirstStepS}\ref{CoMult} and Lemma~\ref{Lemm:AlgoFirstStepS}\ref{ModuloSymP2Onepred}.
We will then assume that $n>1$ in the sequel.

We write $\psi=\psi_n\circ \cdots \circ\psi_1\in \Weyl$, and write $\Base(a_0)\cup \Base(\psi)=\{p_1,\dots,p_l\}$, for some distinct points $p_i$ that we can assume to be such that $m_{p_1}(a_0)\ge \dots \ge m_{p_l}(a_0)\ge 0$.
If there exist $i,j\in \{1,\dots,l\}$ such that $i<j$ and  $m_{p_i}(\psi)<m_{p_j}(\psi)$, we denote by $\tau\in \Sym_{\p^2}$ the permutation of $p_i$ and $p_j$, write $b_n'=\tau\circ \psi\circ \tau (a_0)$ and replace $b_n$ with $b_n'$ and $\psi_i$ with $\tau\circ \psi_i\circ\tau$ for $i=1,\dots,n$, which are Jonqui\`eres elements of $\Weyl$ (Lemma~\ref{Lemm:JoWinfty}). To see that this is possible, we use 
 \[
\deg(b_n')-\deg(b_n)\stackrel{\text{Lemma~\ref{Lemm:Permutationdecreases}.\ref{EquationProductNeg}}}{=}(m_{p_i}(a_0)-m_{p_j}(a_0))(m_{p_i}(\psi)-m_{p_j}(\psi))\le 0.\] 
If $\deg(b_n')=\deg(b_n)$ then $b_n$ and $b_n'$ are equal modulo $\Sym_{\p^2}$  (Lemma~\ref{Lemm:Permutationdecreases}\ref{SameMultEqualSym}). And if $\deg(b_n')<\deg(b)$, then $\comult(b_n')\le \comult(b_n)$ (Lemma~\ref{Lemm:Permutationdecreases}\ref{NotSameMultDecrease}). In both cases, proving the result for $b_n'$ gives the result for $b_n$. 
After finitely many steps, we then reduce to the case where 
 \[m_{p_1}(a_0)\ge \dots \ge m_{p_l}(a_0)\ge 0\text{ and }m_{p_1}(\psi)\ge \dots \ge m_{p_l}(\psi)\ge 0.\] 
In particular, both
$a_0$
and $\psi$ have maximal multiplicity at $p_1$.

For $i=1,\dots,n-1$, we define $a_i$ to be a $i$-th predecessor of $a_0$, and write $b_i=\psi_i\circ \dots \circ \psi_1(a_0)$.
If $a_n=e_0$, all assertions hold, so we can assume that $\deg(a_n)\ge 2$, which implies that $\comult(a_{n-1})\ge 2$
(by Lemma~\ref{Lemm:EasyMultCoMult}\ref{EasyMultCoMult4} the equality $\comult(a_{n-1})  = 1$ would give $a_n =e_0$).
By induction hypothesis, we have $\comult(b_{n-1})\ge \comult(a_{n-1})\ge 2$, so $b_n\not=e_0$. As this is true for any choice of Jonqui\`eres elements $\psi_1,\dots,\psi_n$, we find that \[\lgth(a_0)\ge n+1.\]
We now apply the algorithm to $\Lambda_0=\psi^{-1}(e_0)$. As $p_1$ is a point
of maximal multiplicity of $\Lambda_0$,
there exists $\theta_1\in \J_{p_1}$ such that $\Lambda_1=\theta_1(\Lambda_0)$ is a predecessor of $\Lambda_0$ (Lemma~\ref{Lemm:AlgoFirstStepS}\ref{eachsinS}). We then define a sequence $(\Lambda_i)_{i\ge 2}$, and Jonqui\`eres elements $(\theta_i)_{i\ge 2}$ such that $\Lambda_i=\theta_{i}(\Lambda_{i-1})$ is a predecessor of $\Lambda_{i-1}$ for each $i\ge 1$. Since $\psi$ is a product of $n$ Jonqui\`eres elements, we have $\lgth(\Lambda_0)\le n<\lgth(a_0)$. The pair $(n,\lgth(a_0))$ is thus bigger than $(n,\lgth(\Lambda_0))$, so we find that $\deg \Lambda_n \le \deg \psi(\Lambda_0)=\deg e_0 =1$ by induction hypothesis. This yields $\Lambda_n=e_0$, and thus $\theta=\theta_n\circ\cdots \circ \theta_1$ satisfies $e_0=\theta(\Lambda_0)=\theta(\psi^{-1}(e_0))$, which yields $\theta\circ\psi^{-1}\in \Sym_{\p^2}$ (Lemma~\ref{Lemm:EasyWeyl}\ref{L4}). Hence, we can replace $\psi_i$ with $\theta_i$ for $i=1,\dots,n$, since $b_n=\psi(a_0)$ is equal to $\theta(a_0)$ modulo $\Sym_{\p^2}$. This reduces to the case where $\psi_1\in \J_{p_1}$.

 Applying induction hypothesis to $b_1$ and the $n-1$ Jonqui\`eres $\psi_2,\dots,\psi_n$, we can reduce to the case where $b_i$ is a predecessor of $b_{i-1}$ for $i=2,\dots,n$. We can moreover assume that $\psi_2\in \J_q$ where $q\in \B(\p^2)$ is a base-point of $b_1$ of maximal multiplicity (Lemma~\ref{Lemm:AlgoFirstStepS}).
 
If $\deg(a_1)=\deg(b_1)$, then $a_1$ and $b_1$ are equal modulo $\Sym_{\p^2}$ (Lemma~\ref{Lemm:AlgoFirstStepS}), so the result follows, as in this case $b_i$ is a predecessor of $b_{i-1}$ for $i=1,\dots,n$, hence $b_i$ is equal to $a_i$ modulo $\Sym_{\p^2}$ for $i=1,\dots,n$. We can then assume that $\deg(a_1)<\deg(b_1)$. 

If $b_1$ has maximal multiplicity at $p_1$, then we can assume that $\psi_2\in \J_{p_1}$ (since $b_2=\psi_2(b_1)$ is a predecessor of $b_1$), and apply the induction hypothesis to $\psi_2\circ \psi_1,\psi_3,\dots,\psi_n$ to obtain that $ \deg(a_{n-1})\le \deg(b_n)$ and $ \comult(a_{n-1})\le\comult(b_n)$. The result follows in this case from $\deg(a_n)\le \deg(a_{n-1})$ and $\comult(a_n)\le \comult(a_{n-1})$ (Lemma~\ref{Lemm:AlgoFirstStepS}\ref{CoMult}).

Otherwise we can
find a point $q\in \B(\p^2)\setminus \{p_1\}$ such that $b_1$ has maximal multiplicity at $q$ and $\psi_2\in \J_q$. We claim that there exists $r \in \B(\p^2)\setminus \{p_1,q\}$ such that $\deg(b_1)<m_{p_1}(b_1)+m_{q}(b_1)+m_r(b_1)$.
Let us first show why this claim achieves the proof, before proving the claim.
We choose the involution $\sigma_{p_1,q,r}\in \J_{p_1}\cap \J_q$ as in Definition~$\ref{Defi:sigmap1p2p3}$, write $\psi_1'=\sigma_{p_1,q,r}\circ\psi_1\in \J_{p_1}$ and $\psi_2'=\psi_2\circ \sigma_{p_1,q,r}\in \J_q$. We can thus replace $\psi_1$ and $\psi_2$ with the Jonqui\`eres elements $\psi_1'$ and $\psi_2'$ respectively, without changing $b_i$ for $i=2,\dots,n$. This replaces $b_1$ with $b_1'=\sigma_{p_1,q,r}(b_1)$, which satisfies $\deg(b_1')<\deg(b_1)$ (Lemma~\ref{Lemm:sigmadeg}\ref{sigma123degsmall}). The result then follows by induction hypothesis. 

It remains to prove the claim. As $p_1$ is a point of maximal multiplicity of $a_0$, there exists $\varphi_1\in \J_{p_1}$ such that $\varphi_1(a_0)$ is a predecessor of $a_0$,
so that $\varphi_1(a_0)$ is
equal to $a_1$ mod $\Sym_{\p^2}$ (Lemma~\ref{Lemm:AlgoFirstStepS}). We set
$\nu=\varphi_1 \circ (\psi_1)^{-1}\in \J_{p_1}$ and find a set $\Delta \subseteq \B(\p^2) \smallsetminus \{p_1\}$ of even order, such that 
$\nu^{-1}(e_0)=e_0+\sum_{r\in \Delta} \frac{e_0-e_{p_1}}{2}-e_r$
(Lemma~\ref{LemmJvarphiLin}). In particular,
\begin{center}$\deg \varphi_1(a_0) = \deg \nu \circ \psi_1(a) = \nu^{-1}(e_0) \cdot b_1=\deg b_1 +\sum_{r\in \Delta} \left(\frac{\deg b_1-m_{p_1}(b_1)}{2}-m_r(b_1)\right).$\end{center}
Since $\deg \varphi_1(a_0) < \deg b_1$, there exist two distinct points $r_1,r_2\in \Delta$ such that
\[\begin{array}{rcl}0 &>& \left(\frac{\deg b_1-m_{p_1}(b_1)}{2}-m_{r_1}(b_1)\right)+\left(\frac{\deg b_1-m_{p_1}(b_1)}{2}-m_{r_2}(b_1)\right)\\
&=&\deg b_1-m_{p_1}(b_1)-m_{r_1}(b_1)-m_{r_2}(b_2).\end{array}\]
Since  $m_q(b_1) \ge \max\{m_{r_1}(b_1),m_{r_2}(b_1)\}$ and $q\not=p_1$,
we can replace one of the two points $r_1,r_2$ with $q$ and denote by $r$ the other one. This achieves the proof of the claim.
\end{proof}

\subsection{From the Weyl group to the Cremona group}\label{SubSec:FromWeyltoBir}
Starting from an element of $\Bir(\p^2)$, Algorithm~\ref{Algo:AlgorithInWeyl} yields a way to decompose it into a product of Jonqui\`eres elements of $\Weyl$, and the optimality of the algorithm in $\Weyl$ is given by Proposition~\ref{Prop:AlgoWeyl}. We now show that the algorithm also works in $\Bir(\p^2)$. To do this, we first make the following easy observation, which relates the two notions of predecessors
already defined
for elements of $\Bir(\p^2)$ and for elements of $\Weyl(e_0)$ (Definitions~\ref{DefiPredeBirP2} and \ref{DefiPredeWinfe0} respectively). We will then prove that the hypothesis of Lemma~\ref{Lemma:TwoTypesPred} is in fact always satisfied (Proposition~$\ref{Prop:PredecessorsBirP2FromCastelnuovo-predecessor})$, and this will allow us to give a stronger version of Lemma~\ref{Lemma:TwoTypesPred} (Corollary~\ref{Cor:EquiTwoPred} below).

\begin{lemma}\label{Lemma:TwoTypesPred}
Let $f \in \Bir(\p^2)$.
If there exists a Jonqui\`eres element $\varphi\in \Bir(\p^2)$ such that $\varphi(f^{-1}(e_0))\in \ZZ_{\p^2}$ is a predecessor of $f^{-1}(e_0)$ $($in the sense of Definition~$\ref{DefiPredeWinfe0})$, then:
\begin{enumerate}
\item\label{fvarphimPred}
$f\circ \varphi^{-1}\in \Bir(\p^2)$ is a predecessor of $f$ $($in the sense of Definition~$\ref{DefiPredeBirP2})$.
\item\label{gpredgmpred}
$g^{-1}(e_0)$ is a predecessor of $f^{-1}(e_0)$, for each predecessor $g\in \Bir(\p^2)$ of $f$.
\end{enumerate}
\end{lemma}

\begin{proof}
\ref{fvarphimPred}:
To prove that $f\circ \varphi^{-1}$ is a predecessor of $f$, we only need to show that
$\deg(f\circ \psi^{-1})\ge \deg( f\circ \varphi^{-1})$
(which is equivalent to $\deg ( \varphi(f^{-1}(e_0))) \le \deg (\psi(f^{-1}(e_0)))$)
for each Jonqui\`eres element $\psi\in \Bir(\p^2)$. As $\varphi(f^{-1}(e_0))$ is a predecessor of $f^{-1}(e_0)$ it satisfies $\deg( \varphi(f^{-1}(e_0))) \le \deg (\psi(f^{-1} (e_0)))$ for each Jonqui\`eres element $\psi \in \Weyl$, and thus in particular for each Jonqui\`eres element $\psi \in \Bir(\p^2)$ (Lemma~\ref{Lemm:JoCremWinfty}).

\ref{gpredgmpred}: If $g\in \Bir(\p^2)$ is a predecessor of $f$, then $g=f\circ \kappa$ for some Jonqui\`eres transformation $\kappa\in \Bir(\p^2)$, and $\deg(g)=\deg(f\circ \varphi^{-1})$ by \ref{fvarphimPred}. The element $g^{-1}(e_0)=\kappa^{-1}(f^{-1}(e_0))$ has then the same degree as the predecessor $\varphi(f^{-1}(e_0))$ of $f^{-1}(e_0)$, and is thus also
a predecessor of $f^{-1}(e_0)$ (by Definition~\ref{DefiPredeWinfe0}).
\end{proof}

We first recall the following famous result, corresponding to the algorithm defined in \cite[Chapter 8]{Alberich}, adapting the proof of Castelnuovo \cite{Cas}.

\begin{proposition}[Castelnuovo reduction]\label{Prop:AlgoCas}
Let $f\in \Bir(\p^2)$ be of degree $d>1$, let $p\in \p^2$ be a base-point of
$f$ of maximal multiplicity. We define $M=\{q\in\Base(f)\setminus \{p\} \mid m_{p}+2m_{q}>d\}$. Then, $M$ contains at least two elements, and the following hold:
\begin{enumerate}
\item\label{Meven}
If $\lvert M\rvert$ is even, there is an element $\varphi\in \Jonq_{p}$ such that $\Base(\varphi)=\{p\}\cup M$.
\item\label{Modd}
If $\lvert M\rvert$ is odd, there is $q\in M$ of minimal multiplicity and an element $\varphi\in \Jonq_{p}$ such that $\Base(\varphi)=\{p\}\cup (M\setminus \{q\})$.
\end{enumerate}
For each $\varphi$ as above, we have $\deg (f \circ \varphi^{-1})<\deg (f)$.
\end{proposition}

\begin{remark}
In \cite{Alberich}, the elements of $M$ are called the \emph{major base-points} of $f$ and their number $\lvert M \rvert$ is written $h$ (see \cite[Definition 8.2.1]{Alberich}).
\end{remark}

\begin{proof}
The proof of the proposition lies in Chapter 8 and especially in \S8.3 of \cite{Alberich}. The fact that $h=\lvert M\rvert\ge 2$ is \cite[Lemma 8.2.6]{Alberich}. The existence of $\varphi$ and the fact that $\deg (f \circ \varphi^{-1})<\deg (f)$ is given at page 242, in the proof of \cite[Theorem 8.3.4]{Alberich}.
\end{proof}

\begin{definition} \label{Def:Castelnuovo-predecessor}
Let $f\in \Bir(\p^2)$. If $\deg(f)>1$, we define a \emph{Castelnuovo-predecessor of $f$} to be an element
of the form $f \circ \varphi^{-1}$ where the Jonqui\`eres transformation $\varphi \in \Bir(\p^2)$ has been described in Proposition~\ref{Prop:AlgoCas}. If $\deg(f)=1$, we define $f$ to be its own Castelnuovo-predecessor.
\end{definition}

\begin{lemma}\label{Lemm:CastelnuovopredJonq}
Let $f\in \Bir(\p^2)$ be a Jonqui\`eres  element of degree $d>1$. Then every Castelnuovo-predecessor of $f$ has degree $1$.
\end{lemma}

\begin{proof}
The homaloidal type of $f$ is $(d ; d-1, 1^{2d -2} )$. If $p \in \p^2$ is a base-point of $f$ of maximal multiplicity, then the set $M$ of Proposition~\ref{Prop:AlgoCas} has even cardinality and satisfies $\{ p \} \cup M = \Base (f)$. If $\varphi\in \Jonq_{p}$ is such that $\Base(\varphi)=\{p\}\cup M$, then we have $\deg ( f \circ \varphi^{-1} ) = d^2 - (d-1)^2 - (2d-2) = 1$.
\end{proof}

\begin{algorithm}[Algorithm of Castelnuovo]\label{Algo:Castelnuovo}
Taking $f_0\in \Bir(\p^2) \setminus \Aut(\p^2)$, Proposition~\ref{Prop:AlgoCas} yields a Jonqui\`eres element
$\varphi_1 \in \Bir(\p^2)$
such that the Castelnuovo-predecessor
$f_1=f_0 \circ \varphi_1^{-1}$
satisfies $\deg (f_1)<\deg (f_0)$. Applying again the result finitely many times, we find a sequence of Castelnuovo-predecessors $f_0,f_1,\dots,f_n$ and
a sequence of Jonqui\`eres elements
$\varphi_1,\varphi_2,\dots,\varphi_{n}$,
which lead to a decomposition of $f_0$ into
$f_n\circ \varphi_{n}\circ \dots \circ \varphi_1$,
where $f_n \in \Aut(\p^2)$.
Since $\varphi_{n}':=f_n\circ \varphi_{n}$ is Jonqui\`eres, this algorithm actually provides a decomposition of $f_0$ into the product of $n$ Jonqui\`eres elements
\[ f_0 = \varphi'_n \circ \varphi_{n-1} \circ \cdots \circ  \varphi_1.\]
As we will show in Corollary~\ref{Cor:CastelnuovoOptimal},
this integer $n$ (which is the integer $n$ for which the algorithm stops, i.e.~for which $\deg f_n = 1$) is the length of $f_0$.
\end{algorithm}

Recall that a predecessor of an element $a \in \Weyl (e_0)$ is an element of minimal degree among all the elements of the form $\varphi (a)$ where $\varphi$ is a Jonqui\`eres element of $\Weyl$. The next fundamental result shows that we can choose $\varphi$ to be in $\Bir(\p^2)$ if
$a=f^{-1}(e_0)$ for some $f\in \Bir(\p^2)$.

\begin{proposition}  \label{Prop:PredecessorsBirP2FromCastelnuovo-predecessor}
Let $f\in \Bir(\p^2)\setminus \Aut(\p^2)$ and let $p_0\in \p^2$ be a base-point of $f$ of maximal multiplicity. 
\begin{enumerate}
\item  \label{PredecessorsBirP21} There exists a Jonqui\`eres element
$\psi \in \Jonq_{p_0}$ such that $\psi(f^{-1}(e_0))\in \ZZ_{\p^2}$ is a predecessor of $f^{-1}(e_0)$.
For each such $\psi$, it follows from Lemma~$\ref{Lemma:TwoTypesPred}\ref{fvarphimPred}$ that the element $f\circ \psi^{-1}\in \Bir(\p^2)$ is a predecessor of~$f$.
\item\label{PredecessorsBirP22}
Let $\varphi\in \Jonq_{p_0}$ be a Jonqui\`eres element such that $f \circ \varphi^{-1}$ is a Castelnuovo-predecessor of $f$ $($see Proposition~$\ref{Prop:AlgoCas}$ and Definition~$\ref{Def:Castelnuovo-predecessor}$$)$. Then, we
can choose $\psi$ as above such that one of the following assertions is satified:
\begin{enumerate}
\item
$\psi=\varphi$;
\item
$\rho = \psi \circ\varphi^{-1}$ is a quadratic map. Moreover there is a unique $($proper$)$
base-point
$p'$ of maximal multiplicity  of $f\circ \varphi^{-1}$. This point is also a
base-point
of  maximal multiplicity of $f\circ \psi^{-1}$ $($but not necessarily unique$)$.  We have $p'\not=p_0$ and $\rho\in \Jonq_{p'}\cap \Jonq_{p_0}$.
\end{enumerate}
\end{enumerate}
\end{proposition}

\begin{proof}
Let $\varphi\in \Jonq_{p_0}$ be a Jonqui\`eres element provided by the Castelnuovo reduction. 
Write $\Base(f)=\{p_0,\dots,p_r\}$ where the points $p_i$ are distinct and set $m_i=m_{p_i}(f)$ for each $i$. Choose
the order such that $m_0\ge m_1\ge \dots \ge m_r$, and such that for any $i\ge 1$,  either $p_i$ is  a proper point of $\p^2$ or $p_i$ is in the first neighbourhood of some $p_j$ with $j <i$.
We then define $m_i=0$ for each integer $i>r$, and write $M=\{p_i\mid i\ge 1, m_0+2m_i>d\}$ as in Proposition~\ref{Prop:AlgoCas}. 

Suppose first that $\lvert M\rvert$ is even. In this case, $\varphi\in \Jonq_{p_0}$ satisfies $\Base(\varphi)=\{p_0\}\cup M$ (Proposition~\ref{Prop:AlgoCas}\ref{Meven}) and is then equal to $\varphi=\alpha \circ \iota_{p_0,M}$ for some $\alpha\in \Sym_{\p^2}$ (Lemma~\ref{LemmJvarphiLin}\ref{LemmJvarphiLin2}). Writing $2s=\lvert M\rvert$, we find
$m_0+m_{2s-1}+m_{2s}=((m_0+2m_{2s-1})+(m_0+2m_{2s}))/2 > d\ge ((m_0+2m_{2s+1})+(m_0+ 2m_{2s+2}))/2= m_0+m_{2s+1}+m_{2s+2}$,
which implies that $\iota_{p_0,\{p_1,\dots,p_{2s}\}}(f^{-1}(e_0))$ is a predecessor of $f^{-1}(e_0)$
(Lemma~\ref{Lemm:AlgoFirstStepS}\ref{eachsinS}), so the same holds for $\varphi(f^{-1}(e_0))$.  This achieves the proof, by choosing $\psi=\varphi$, whence $\rho=\mathrm{id}$.

Suppose now that $\lvert M\rvert$ is odd. In this case, $\varphi\in \Jonq_{p_0}$ satisfies $\Base(\varphi)=\{p_0\}\cup (M\setminus \{q\})$ for some $q\in M$ of minimal multiplicity (Proposition~\ref{Prop:AlgoCas}\ref{Modd}) and is equal to $\varphi=\alpha \circ \iota_{p_0,M\setminus \{q\}}$ for some $\alpha\in \Sym_{\p^2}$ (Lemma~\ref{LemmJvarphiLin}\ref{LemmJvarphiLin2}). Writing $2s+1=\lvert M\rvert$, we can assume that $p_{2s+1}=q$, which yields $\iota_{p_0,M\setminus \{q\}}=\iota_{p_0,\{p_1,\dots,p_{2s}\}}$,  and find
$m_0+m_{2s-1}+m_{2s}=((m_0+2m_{2s-1})+(m_0+ 2 m_{2s}))/2>d$.
We then obtain two cases:

If $d\ge  m_0+m_{2s+1}+m_{2s+2}$, then $\iota_{p_0,\{p_1,\dots,p_{2s}\}}(f^{-1}(e_0))$ is a predecessor of $f^{-1}(e_0)$ (Lemma~\ref{Lemm:AlgoFirstStepS}\ref{eachsinS}), so the same holds for $\varphi(f^{-1}(e_0))$. We then choose $\psi=\varphi$ as before.

The last case is when $d<m_0+m_{2s+1}+m_{2s+2}$, which implies that $m_{2s+2}>0$ (Corollary~\ref{Coro:PositivityOrbit}) and thus that $2s+2\le r$. We can assume that $p_{2s+2}$ is not infinitely near to a point $p_i$ with $i>2s+2$ (otherwise  we have $m_i=m_{2s+2}$, so we exchange $p_{2s+2}$ with $p_i$).
Since $d\ge m_0+m_{2s+3}+m_{2s+4}$, the element 
$a_{s+1}=\iota_{p_0,\{p_1,\dots,p_{2s+2}\}}(f^{-1}(e_0))$ is a predecessor of $f^{-1}(e_0)$ (Lemma~\ref{Lemm:AlgoFirstStepS}\ref{eachsinS}). Moreover, $a_{s}=\iota_{p_0,\{p_1,\dots,p_{2s}\}}(f^{-1}(e_0))$ is not  a predecessor of $f^{-1}(e_0)$ (Lemma~\ref{Lemm:AlgoFirstStepS}\ref{eachsinS}), which implies that $\deg(a_s)>\deg(a_{s+1})$.  We will prove the following numerical assertions:
\begin{center}\begin{tabular}{ll}
$\mathrm{(I)}$& The point $p_{2s+2}$ is 
a base-point of $a_{s+1}$ of maximal multiplicity; \\
$\mathrm{(II)}$& $m_{p_{2s+2}}(a_s)> m_{p_i}(a_s)$ for $i=1,\dots,2s$;\\
$\mathrm{(III)}$& $(e_0-e_{p_0}-e_{p_{2s+1}}-e_{p_{2s+2}})\cdot a_s=\deg(a_s)-m_{p_{0}}(a_s)-m_{p_{2s+1}}(a_s)-m_{p_{2s+2}}(a_s)
<0$;\\
$\mathrm{(IV)}$& The point $p_{2s+1}$ is the unique base-point of $a_{s}$ of maximal multiplicity. \\
\end{tabular}\end{center}
Before proving these assertions, let us show how they imply the result.

We write $\Lambda=\varphi(f^{-1}(e_0))$, which corresponds to the linear system of $f\circ \varphi^{-1}$.
For $i=1,\dots,r$, we then denote by $q_i\in \B(\p^2)$ the point such that $\alpha(e_{p_i})=e_{q_i}$. As $\varphi=\alpha\circ \iota_{p_0,\{p_1,\dots,p_{2s}\}}$, we find  $\Base(\Lambda)\subseteq\{q_0,\dots,q_r\}$.
Moreover,  $\varphi$ and $\iota_{p_0,\{p_1,\dots,p_{2s}\}}$ belong to $\J_{p_0}$, so that $\alpha$ also belongs to $\J_{p_0}$. This gives us $\alpha ( e_0 - e_{p_0}) = e_0 - e_{p_0}$ and finally $q_0=p_0$ is a proper point of $\p^2$.
Assertion $(\mathrm{IV})$ implies that $p':=q_{2s+1}$ is the unique base-point of $f\circ \varphi^{-1}$ of maximal multiplicity, in particular  $p'$ is a proper point of $\p^2$.
We then observe that $q_{2s+2}$ is either a proper point of $\p^2$ or a point infinitely near $p_0$ or $p'$. Indeed, it cannot be infinitely near $q_i$ if $i\in \{1,\dots,2s\}$ by $(\mathrm{II})$ and if $i>2s+1$, because $p_{2s+1}$ is not infinitely near $p_{i}$. Moreover, $p_0,p',q_{2s+2}$ are not collinear because of $(\mathrm{III})$ and B\'ezout Theorem. Up to change of coordinates, we can thus assume that
$\{p_0,p'\}=\{[1:0:0],[0:1:0]\}$
and that $p_0,p',q_{2s+2}$ are the three base-points of a quadratic involution $\rho\in \Jonq_{p_0}\cap \Jonq_{p'}\subseteq\Bir(\p^2)$ which is one of the two following
\[ [x:y:z]\dasharrow [yz:xz:xy] \text{ or } [x:y:z]\dasharrow [z^2: xy: xz],\]
and satisfies then $\rho=\beta\circ \sigma_{p_0,p',q_{2s+2}}$ for some $\beta\in \Sym_{\p^2}$
(see Definition~\ref{Defi:sigmap1p2p3}).
The result then follows by setting $\psi:=\rho\circ \varphi\in \Jonq_{p_0}$.
Indeed, as $\Base(\beta\circ \sigma_{p_0,p',q_{2s+2}} \circ \alpha)=\{p_0,p_{2s+1},p_{2s+2}\}$, we have $\beta\circ \sigma_{p_0,p',q_{2s+2}} \circ \alpha=\gamma\circ \iota_{p_0,\{p_{2s+1},p_{2s+2}\}}$ for some $\gamma\in \Sym_{\p^2}$, which yields $\psi=\beta\circ \sigma_{p_0,p',q_{2s+2}} \circ \alpha\circ \iota_{p_0,\{p_1,\dots,p_{2s}\}}=\gamma \circ \iota_{p_0,\{p_1,\dots,p_{2s+2}\}}$, and implies that $\psi(f^{-1}(e_0))=\gamma(a_{s+1})$ is a predecessor of $f^{-1}(e_0)$. Moreover, $p'=q_{2s+1}\in \p^2$ is such that $\rho\in \Jonq_{p'}$ and is a point
of maximal multiplicity of $\psi(f^{-1}(e_0))$; this follows from $(\mathrm{I})$ and from $\gamma(e_{p_{2s+2}})=e_{p'}$, which is given by 
$e_0-e_{p'}=\rho (e_0-e_{p'})=\rho \circ \alpha(e_0-e_{p_{2s+1}})=\gamma\circ \iota_{p_0,\{p_{2s+1},p_{2s+2}\}}(e_0-e_{p_{2s+1}})=\gamma(e_0-e_{p_{2s+2}})=e_0-\gamma(e_{p_{2s+2}})$.\\

It remains to prove the assertions $(\mathrm{I})$--$(\mathrm{IV})$. 

$(\mathrm{I})$:  Writing $\mu=\frac{d-m_0}{2}$, we have $m_0\ge \dots \ge m_{2s}\ge m_{2s+1}>\mu\ge m_{2s+2}\ge \dots \ge m_{r}$ and $2\mu=d-m_0<m_{2s+1}+m_{2s+2}$. This yields
$m_{p_{2s+2}}(a_{s+1})=d-m_0-m_{2s+2}=2\mu -m_{2s+2}
\geq  
\mu>2\mu-m_{i}=
 d-m_0 - m_i = m_{p_i}(a_{s+1})$ for each $i\in \{1,\dots,2s+1\}$.
We moreover have $m_{p_{2s+2}}(a_{s+1})
\geq\mu\geq m_i= m_{p_i}(a_{s+1})$ for each $i\in \{2s+3,\dots,r\}$. It then remains to show that $m_{p_{2s+2}}(a_{s+1})>m_{p_0}(a_{s+1})$. This holds, because otherwise $p_0$ would be a point of maximal multiplicity of $a_{s+1}$, which would yield the existence of $\kappa\in \J_{p_0}$ such that $\deg(\kappa(a))<\deg(a)$ (Corollary~\ref{Coro:DecrMaxMult}), contradicting the fact that $a_{s+1}$ is a predecessor of $f^{-1}(e_0)$.

$(\mathrm{II})$: Follows from $m_{p_{2s+2}}(a_s)=m_{2s+2}>d-m_0-m_{2s+1}\ge d-m_0-m_{i}=m_{p_i}(a_s)$, for $i\in \{1,\dots,2s\}$.

$(\mathrm{III})$--$(\mathrm{IV})$: 
Set $\nu=m_{p_{0}}(a_s)+m_{p_{2s+1}}(a_s)+ m_{p_{ 2s+2 }}(a_s)
-\deg(a_s)$. The equality $a_{s+1}=\iota_{p_0,\{p_{2s+1},p_{2s+2}\}}(a_s)$ gives $\deg(a_{s+1})=\deg(a_s)-\nu$, whence $\nu>0$, i.e.~$(\mathrm{III})$. It also provides
\[m_{p_0}(a_s)= m_{p_0}(a_{s+1})+ \nu,
m_{p_{2s+1}}(a_s) = m_{p_{2s+2}}(a_{s+1}) +\nu , m_{p_{2s+2}}(a_{s}) = m_{p_{2s+1}}(a_{s+1} ) + \nu,  \]
and $m_{p_i} (a_s) = m_{p_i} (a_{s+1})$ for $i \neq 0, 2s+1,2s+2$.
Since $p_{2s + 2}$ was a base-point of maximal multiplicity of $a_s$ (by $(\mathrm{I})$), it follows that $p_{2s + 1}$ is a base-point of maximal multiplicity of $a_{s+1}$ and that such base-points of maximal multiplicity of $a_{s+1}$
belong to the set $\{ p_0, p_{2s+1}, p_{2s+2} \}$. However, we have already seen in the proof of $(\mathrm{I})$ that  $m_{p_{2s+2}}(a_{s+1})>m_{p_i}(a_{s+1})$ for $i=0$ or $i=2s+1$. This proves $(\mathrm{IV})$.
\end{proof}

\begin{corollary}\label{Cor:EquiTwoPred}
Let $f,g \in \Bir(\p^2)$. Then, the two following assertions are equivalent:
\begin{enumerate}
\item \label{predecessor-in-Bir(p2)}
$g$ is a predecessor of $f$ $($Definition~$\ref{DefiPredeBirP2})$;
\item \label{predecessor-in-W(infinity)}
$g^{-1}(e_0)$ is a predecessor of $f^{-1}(e_0)$ $($Definition~$\ref{DefiPredeWinfe0})$ and $f^{-1}\circ g\in \Jonq\subseteq\Bir (\p^2)$.
\end{enumerate}
\end{corollary}

\begin{proof}
By Proposition~\ref{Prop:PredecessorsBirP2FromCastelnuovo-predecessor}\ref{PredecessorsBirP21}, the assumptions of Lemma~\ref{Lemma:TwoTypesPred} are always fulfilled. Therefore, the implication $\ref{predecessor-in-Bir(p2)}\Rightarrow\ref{predecessor-in-W(infinity)}$ follows from Lemma~\ref{Lemma:TwoTypesPred}\ref{gpredgmpred} and the implication $\ref{predecessor-in-W(infinity)}\Rightarrow\ref{predecessor-in-Bir(p2)}$ follows from Lemma~\ref{Lemma:TwoTypesPred}\ref{fvarphimPred} applied with $\varphi =g^{-1}\circ f\in \Jonq$.
\end{proof}

We are now ready to give the proof of Lemma~\ref{Lem:Pred} and Theorem~\ref{TheMainTheorem}.

\begin{proof}[Proof of Lemma~$\ref{Lem:Pred}$]
Let $f\in \Bir(\p^2)$, and let $g\in \Bir(\p^2)$ be a predecessor of $f$, which is then equal to $g=f \circ \varphi$ for some Jonqui\`eres element $\varphi\in \Bir(\p^2)$. 

By Corollary~\ref{Cor:EquiTwoPred}, $g^{-1}(e_0)=\varphi^{-1}(f^{-1}(e_0))$ is a predecessor of $f^{-1}(e_0)$. This implies that $\Base(\varphi^{-1}) \subseteq \Base(f^{-1}(e_0))=\Base(f)$ (Lemma~\ref{Lemm:AlgoFirstStepS}\ref{AllPredLikeThisBaseCont}), and gives \ref{Lem:Pred3}. Moreover, the homaloidal type of $g$, which is the class of $g^{-1}(e_0)$ modulo $\Sym_{\p^2}$, is uniquely determined by the homaloidal type of $f$ (Corollary~\ref{Coro:UniquenessPredUpSym}). This proves \ref{Lem:Pred1}. 

It remains to prove \ref{Lem:Pred2}. The set of predecessors being invariant under right multiplication by elements of $\Aut(\p^2)$, it is infinite. It remains to see that the number of classes modulo $\Aut(\p^2)$ is finite. This corresponds to saying that the number of possibilities for $\Base(\varphi^{-1})$ is finite, and is thus given by \ref{Lem:Pred3}.
\end{proof}

\begin{proof}[Proof of Theorem~$\ref{TheMainTheorem}$]
For each $i\ge 0$, we set $a_i:=(f_i)^{-1}(e_0)\in \Weyl(e_0)$. By Corollary~\ref{Cor:EquiTwoPred}, $a_i$ is a predecessor of $a_{i-1}$ for each $i\ge 1$, so $a_i$ is a $i$-th predecessor of $a_0$ for each $i\ge 1$.

We then write $b_n=g_n^{-1}(e_0)=\varphi_n^{-1}\circ \dots \circ \varphi_1^{-1}(a_0)$. As $\varphi_1,\dots,\varphi_n$ are Jonqui\`eres elements of $\Bir(\p^2)$, they are all the more Jonqui\`eres elements of $\Weyl$ (Lemma~\ref{Lemm:JoCremWinfty}). 
Then, the three assertions~\ref{Degfngn}-\ref{Comultfngn}-\ref{AnBnSymp2} of Theorem~\ref{TheMainTheorem} directly follow from the three corresponding assertions~\ref{DegAnBn}-\ref{ComultAnBn}-\ref{AnBnSymp2} of Proposition~\ref{Prop:AlgoWeyl} that we now recall: \ref{DegAnBn} $\deg(a_n)\le \deg(b_n)$; \ref{ComultAnBn} $\comult(a_n)\le \comult(b_n)$; \ref{AnBnSymp2} if $\deg(a_n)= \deg(b_n)$, then $a_n$ and $b_n$ are equal modulo $\Sym_{\p^2}$.

It remains to observe that $\lgth(f_0)=\min \{n\mid \deg(f_n)=1\}$. To do this, we write 
$\ell=\lgth(f_0)$,
$m=\min \{n\mid \deg(f_n)=1\}$, and prove $\ell\ge m$ and $m\ge \ell$.

The fact that $\ell=\lgth(f)$ yields the existence of Jonqui\`eres elements  $\varphi_1, \ldots, \varphi_\ell \in \Bir(\p^2)$ such that $\deg(f \circ \varphi_1\circ \dots \circ \varphi_\ell)=1.$ By~\ref{Degfngn}, we find $\deg(f_\ell)\le 1$, which yields  $\ell\ge m$.

Writing for each $i$ a Jonqui\`eres element $\psi_i\in \Jonq$ such that $f_{i-1}=f_{i}\circ \psi_i$, we obtain $f_0=f_m\circ \psi_m\circ \dots \circ\psi_1$. As $\deg(f_m)=1$, this implies that $\ell=\lgth(f_0)\le m$.
\end{proof}

\begin{corollary}[Length of predecessors, associated to any point of maximal multiplicity]  \label{Cor:decrdegrlgth}
Let $f\in \Bir(\p^2)\setminus \Aut(\p^2)$. For each point $q \in \p^2$ of maximal multiplicity of $f$, there exists $\varphi \in \Jonq_q\subseteq\Bir(\p^2)$ such that $f\circ \varphi$ is a predecessor of $f$. Moreover, every predecessor $g$ of $f$ satisfies $\lgth(g)=\lgth(f)-1$ and $\deg(g)<\deg(f)$.
\end{corollary}

\begin{proof}
The existence of $\varphi$ is given by Proposition~\ref{Prop:PredecessorsBirP2FromCastelnuovo-predecessor}.
Any predecessor $g$ of $f$ satisfies $\deg(g)<\deg(f)$ because such an inequality already holds for a Castelnuovo-predecessor according to (the well-known) Proposition~\ref{Prop:AlgoCas}.
Finally, we have $\lgth(g)=\lgth(f)-1$ by Theorem~\ref{TheMainTheorem}.
\end{proof}

\begin{corollary}\label{Coro:PredCasPredisPredPred}
Let $f$ be an element of $\Bir(\p^2)\setminus \Aut(\p^2)$.
Take $p \in \p^2$ a proper base-point of maximal multiplicity of $f$, and $\varphi \in \Jonq_{p}$ a Jonqui\`eres element such that $h:=f \circ \varphi^{-1}$ is a Castelnuovo-predecessor of $f$.
Then, the following hold:
\begin{enumerate}
\item\label{CorpsiphiEx}
There exist  a Jonqui\`eres element $\psi \in \Jonq_{p}$ and a point $q \in \p^2$ such that:
\begin{enumerate}
\item $g:=f\circ \psi^{-1}$ is a predecessor of $f$;
\item $q$ is a point of maximal multiplicity of $g$ and of $h$;
\item The element $\rho: = \psi \circ \varphi^{-1}$ belongs to $\Jonq_q$ and has degree $\le  2$.
\end{enumerate}
\item\label{CorpsiphiDiff}
For all $\psi,q,g,\rho$ as in \ref{CorpsiphiEx}, and for each $\kappa\in \Jonq_q$, we have
the following equivalence:
\begin{center}
$h\circ \kappa$ is a predecessor of $h$ $\Leftrightarrow$ $h\circ \kappa$ is a predecessor of $g$.
\end{center}
Furthermore, there always exists an element $\kappa$ satisfying these two equivalent conditions.
\end{enumerate}
\end{corollary}

\begin{proof}
Assertion~\ref{CorpsiphiEx} follows from Proposition~\ref{Prop:PredecessorsBirP2FromCastelnuovo-predecessor} (if $\psi=\varphi$ we choose any point $q\in \p^2$ of maximal multiplicity of $g=h$). 

To prove \ref{CorpsiphiDiff}, we first observe that $h=g\circ \rho$, with $\rho\in\Jonq_q$, which implies that the two sets
\[ \mathcal{A}_h=\{ h\circ \kappa \mid \kappa \in \Jonq_q\}\text{ and }\mathcal{A}_g=\{ g\circ \kappa' \mid \kappa' \in \Jonq_q\} \]
are equal. Secondly, since $q$ has maximal multiplicity for $g$ and $h$,
the set $\mathcal{A}_g$ contains a predecessor of $g$ and the set $\mathcal{A}_h$ contains a predecessor of $h$ (Proposition~\ref{Prop:PredecessorsBirP2FromCastelnuovo-predecessor}). Hence, an element of $\mathcal{A}_g=\mathcal{A}_h$ is a predecessor of $g$ (respectively of $h$) if and only if it has minimal degree in $\mathcal{A}_g=\mathcal{A}_h$. This yields \ref{CorpsiphiDiff}. The situation is as follows:
\[\begin{tikzpicture}[scale=0.7]
\node (f) at (0, 3) {$f$};
\node (h) at (-9, 2) {$f \circ \varphi^{-1}=h$};
\node (g) at (9, 1) {$g=f \circ \psi^{-1}=h\circ \rho^{-1}$};
\node (s) at  (0, 0) {$s=h \circ \kappa  = g \circ \kappa'$}; 
\node[draw, rectangle,rounded corners=1.5pt,inner sep=0pt,minimum size=0.5cm,font=\tiny] (f-g) at (3, 2.3) {\ Predecessor\ \ };
\draw (f) -- (f-g); \draw[->,>=latex] (f-g) -- (g);
\node[draw, rectangle,rounded corners=1.5pt,inner sep=0pt,minimum size=0.5cm,font=\tiny] (f-h) at (-4, 2.5) {\ Castelnuovo-predecessor\ \ };
\draw (f) -- (f-h); \draw[->,>=latex] (f-h) -- (h);
\node[draw, rectangle,rounded corners=1.5pt,inner sep=0pt,minimum size=0.5cm,font=\tiny] (h-s) at (-4, 1) {\ Predecessor\ \ };
\draw (h) -- (h-s); \draw[->,>=latex] (h-s) -- (s);
\node[draw, rectangle,rounded corners=1.5pt,inner sep=0pt,minimum size=0.5cm,font=\tiny] (g-s) at (4.5, 0.5) {\ Predecessor\ \ };
\draw (g) -- (g-s); \draw[->,>=latex] (g-s) -- (s);
\node[draw, rectangle,rounded corners=1.5pt,inner sep=0pt,minimum size=0.5cm,font=\tiny] (h-g) at (0,1.6) {$\rho$};
\draw (h) -- (h-g); \draw[->,>=latex] (h-g) -- (g);
\end{tikzpicture}\]
\end{proof}

\begin{corollary}[The Algorithm of Castelnuovo also provides the length]\label{Cor:CastelnuovoOptimal}
For each $f\in \Bir(\p^2)\setminus \Aut(\p^2)$ and each Castelnuovo predecessor $h\in \Bir(\p^2)$  of $f$, we have $\lgth(h)=\lgth(f)-1$.

Hence, writing $f=f_0$ and denoting by $f_1,f_2,\dots$ elements of $\Bir(\p^2)$ such that $f_i$ is a Castelnuovo-predecessor of $f_{i-1}$ for $i\ge 1$, we find $\lgth(f)=\min \{n\mid \deg(f_n)=1\}$.
\end{corollary}

\begin{proof}
By definition of a Castelnuovo-predecessor there exists a proper base-point $p$ of maximal multiplicity of $f$ and a Jonqui\`eres tranformation $\varphi \in \Jonq_{p}$ such that $h=f \circ \varphi^{-1}$. 

By Corollary~\ref{Coro:PredCasPredisPredPred}, there exists $\psi \in \Jonq_{p}$ such that $g=f \circ \psi^{-1}$ is a predecessor of $f$ and an element $s\in \Bir(\p^2)$ which is a predecessor of both $g$ and $h$.

As $f\notin\Aut(\p^2)$, we have $\deg(f)>1$. Corollary~\ref{Cor:decrdegrlgth} yields $\lgth(g)=\lgth(f)-1\ge 0$, because $g$
is a predecessor
of $f$.  It remains to prove that $\lgth(h)=\lgth(g)$.

If $f$ is a Jonqui\`eres element, then $\lgth(f)=1$, whence $\lgth(g)=0$.
By Lemma~\ref{Lemm:CastelnuovopredJonq}, we also have $\lgth(h)=0$.

If $f$ is not a Jonqui\`eres element, then $\deg(h)>1$ and $\deg(g)>1$. We  find  $\lgth(s)=\lgth(h)-1=\lgth(g)-1$, as $s$ is a predecessor of $g$ and $h$ (Corollary~\ref{Cor:decrdegrlgth}).\end{proof}

\begin{corollary}\label{Cor:ErrorQuadratic}
Let $f = f_0$ be an element of $\Bir(\p^2)\setminus \Aut(\p^2)$.
Let $f_1,\ldots,f_n$ be elements of  $\Bir(\p^2)$
such that $f_i$ is a Castelnuovo-predecessor of $f_{i-1}$ for $i=1,\ldots,n$.
Then, setting $g_0 = f$, there exist $g_1,g_2,\ldots,g_n\in \Bir(\p^2)$
such that $g_i$ is a predecessor of $g_{i-1}$ for $i=0,\ldots,n$, and such that $\deg((g_i)^{-1}\circ f_i)\le 2$ for $i=0,\dots,n$.
\end{corollary}

\begin{proof}
Note that if $\deg(f_i)=1$ for some $i<n$, then for each $m\in \{i+1,\ldots,n\}$ we have $f_{m}=f_{i}$ (Definition~\ref{Def:Castelnuovo-predecessor}), so we can simply choose $g_m=f_m$. It suffices then to do the case where $\deg(f_i)>1$ for each $i\in \{0 ,\ldots,n-1\}$.

For each $i \in \{0,\ldots,n-1\}$, let $p_i \in \p^2$ be a proper  base-point of maximal multiplicity of $f_i$, such that $f_{i+1}=f_{i}\circ (\varphi_{i})^{-1}$ for some Jonqui\`eres element $\varphi_{i}\in \Jonq_{p_{i}}\subseteq\Bir(\p^2)$ as in Proposition~\ref{Prop:AlgoCas}. 

For $i=0, \ldots,n-1$, we define inductively $g_{i+1}\in \Bir(\p^2)$.
To do this, we apply Proposition~\ref{Prop:PredecessorsBirP2FromCastelnuovo-predecessor}, and find $\psi_i\in \Jonq_{p_i}$ such that $g_{i+1}=f_{i}\circ (\psi_i)^{-1}$ is a predecessor of $f_{i}$, and that either $\psi_i=\varphi_i$, or $\psi_i \circ(\varphi_i)^{-1}$ is a quadratic map that belongs to $\Jonq_{q}$, where $q\in \p^2$ is the unique base-point of maximal multiplicity of $f_{i+1}$
(so that $q = p_{i+1}$ if $i < n$).
We then find that $(g_{i+1})^{-1}\circ f_{i+1}=\psi_{i}\circ (\varphi_{i})^{-1}$ has degree at most $2$. 

It remains to show that $g_{i+1}$ is a predecessor of $g_{i}$ for $i=0,\ldots,n-1$. If $f_i=g_i$, this is true, since $g_{i+1}$ is a predecessor of $f_i$. Otherwise, we have $i \geq 1$, $f_{i}=f_{i-1}\circ (\varphi_{i-1})^{-1}$,  $g_{i}=f_{i-1}\circ (\psi_{i-1})^{-1}$, and  $\psi_{i-1} \circ(\varphi_{i-1})^{-1}\in \Jonq_{p_i}$. As $g_{i+1}=f_i\circ (\psi_{i})^{-1}$ is a predecessor of $f_i$ and $\psi_i\in \Jonq_{p_i}$, it is also a predecessor of $g_i$ (Corollary~\ref{Coro:PredCasPredisPredPred}\ref{CorpsiphiDiff}).
\end{proof}

\section{Examples and applications}
\subsection{Length of birational maps of small degree}\label{SubSec:LengthSmallDegree}
The next table gives all homaloidal types of degree $\le 6$. 
The homaloidal types are given in the second column (as already said before, a sequence of $s$ multiplicities $m$ is written $m^s$). In the first column ($\#$) a label is associated to each homaloidal type: This is the degree ``{\it $d$ }''  if the type is Jonqui\`eres, or ``{\it $d$.$i$}'' for the others (the order, for each degree, being the anti-lexicographic order according to the multiplicities). Then, the third column ($\ell$) gives the length and the fourth (pr.) gives the predecessor (designated by its label). If the Castelnuovo-predecessor is different from the predecessor, it is also given, but in parenthesis. We see that the lengths are not directly related to the ordering of homaloidal types that we use.
\[\begin{array}{llll}
\begin{array}[t]{|l|l|l|l|}
\hline \#\! & \text{h.~type}& \ell& \text{pr.}\!\\\hline
\lb{1} &1&{\bf 0}&\\\hline
\lb{2} &1^3&{\bf 1}&\lb{1}\\\hline
\lb{3}& 2,1^4&{\bf 1}&\lb{1}\\ \hline 
\lb{4}& 3,1^6&{\bf 1}&\lb{1}\\ \hline\end{array}&
\begin{array}[t]{|l|l|l|l|}
\hline \# & \text{h.~type}& \ell& \text{pr.}\!\\\hline
\lbd{4}{1}& 2^3,1^3&{\bf 2}&\lb{2}\\\hline
\lb{5}&4,1^8&{\bf 1}&\lb{1}\\
\lbd{5}{1}&3,2^3,1^3&{\bf 2}& \lb{2} (\!\lb{3})\\
\lbd{5}{2}&2^6&{\bf 2}&\lb{3}\\\hline
 \end{array} &
\begin{array}[t]{|l|l|l|l|l|}
\hline \# & \text{h.~type}& \ell& \text{pr.}\!\\\hline
\lb{6}&5,1^{10}&{\bf 1}&\lb{1}\\
\lbd{6}{1}&4,2^4,1^3&{\bf 2}&\lb{2}\\
\lbd{6}{2}&3^3,2,1^4&{\bf 2}&\lb{3}\\
\lbd{6}{3}&3^2,2^4,1&{\bf 2}&\lb{3}\\\hline
\end{array}\end{array}\]
The types of degree $7$, $8$, $9$ are given below, and provide the first types
of length $3$:
\[\begin{array}{llll}
\begin{array}[t]{|l|l|l|l|l|}
\hline \# & \text{hom.~type}& \ell& \text{pr.}\!\\\hline
\lb{7}&6,1^{12}&{\bf 1}&\lb{1}\\
\lbd{7}{1}&5,2^5,1^3&{\bf 2}&\lb{2} (\!\lb{3})\\
\lbd{7}{2}&4,3^3,1^5&{\bf 2}&\lb{3} (\!\lb{4})\\
\lbd{7}{3}&4,3^2,2^3,1^2&{\bf 2}&\lb{3}\\
\lbd{7}{4}&3^4,2^3&{\bf 3}&\lbd{4}{1} (\!\lbd{5}{1})\\\hline
\lb{8}&7,1^{14}&{\bf 1}&\lb{1}\\
\lbd{8}{1}&6,2^6,1^3&{\bf 2}&\lb{2}\\
\lbd{8}{2}&5,3^3,2^2,1^3&{\bf 2}&\lb{3}\\\hline
\end{array}&
\begin{array}[t]{|l|l|l|l|l|}
\hline \# & \text{hom.~type}& \ell& \text{pr.}\!\\\hline
\lbd{8}{3}&5,3^2,2^5&{\bf 2}&\lb{3}\\
\lbd{8}{4}&4^3,3,1^6&{\bf 2}&\lb{4}\\
\lbd{8}{5}&4^3,2^3,1^3&{\bf 3}&\lbd{4}{1}\\
\lbd{8}{6}&4^2,3^2,2^3,1&{\bf 3}&\lbd{4}{1} (\!\lbd{5}{1})\!\\
\lbd{8}{7}&4,3^5,1^2&{\bf 2}&\lb{4}\\
\lbd{8}{8}&3^7&{\bf 3}&\lbd{5}{2}\\\hline
\lb{9}&8,1^{16}&{\bf 1}&\lb{1}\\
\lbd{9}{1}&7,2^7,1^3&{\bf 2}&\lb{2} (\!\lb{3})\\\hline
\end{array} &
\begin{array}[t]{|l|l|l|l|l|}
\hline \# & \text{hom.~type}& \ell& \text{pr.}\!\\\hline
\lbd{9}{2}&6,3^4,2,1^4&{\bf 2}&\lb{3}\\
\lbd{9}{3}&6,3^3,2^4,1&{\bf 2}&\lb{3}\\
\lbd{9}{4}&5,4^3,1^7&{\bf 2}&\lb{4} (\!\lb{5})\\
\lbd{9}{5}&5,4^2,3,2^3,1^2&{\bf 3}&\lbd{4}{1} (\!\lbd{5}{1})\!\\
\lbd{9}{6}&5,4,3^4,1^3&{\bf 2}&\lb{4}\\
\lbd{9}{7}&5,4,3^3,2^3&{\bf 3}&\lbd{4}{1}\\
\lbd{9}{8}&4^4,2^4&{\bf 3}&\lbd{5}{1} (\!\lbd{6}{1})\!\\
\lbd{9}{9}&4^3,3^3,2,1&{\bf 3}&\lbd{5}{1}\\ \hline
\end{array}\end{array}\]
There are then $17$ types of degree $10$ and $19$ types of degree $11$, each of length $\le 3$.
\[\begin{array}{lll}
\begin{array}[t]{|l|l|l|l|l|}
\hline \# & \text{hom.~type}& \ell& \text{pr.}\!\\\hline
\lb{1\!0}&9,1^{18}&{\bf 1}&\lb{1}\\
\lbd{1\!0}{1}&8,2^8,1^3&{\bf 2}&\lb{2}\\
\lbd{1\!0}{2}&7,3^5,1^5&{\bf 2}&\lb{3}(\!\lb{4})\\
\lbd{1\!0}{3}&7,3^4,2^3,1^2&{\bf 2}&\lb{3}\\
\lbd{1\!0}{4}&6,4^3,2^3,1^3&{\bf 3}&\lbd{4}{1} (\!\lbd{6}{1})\!\\
\lbd{1\!0}{5}&6,4^2,3^3,1^4&{\bf 2}&\lb{4}\\ \hline \end{array} & 
\begin{array}[t]{|l|l|l|l|l|}
\hline \# & \text{hom.~type}& \ell& \text{pr.}\!\\\hline
\lbd{1\!0}{6}&6,\!4^2\!,\!3^2\!,\!2^3\!,\!1&{\bf 3}&\lbd{4}{1}\\
\lbd{1\!0}{7}&6,3^7&{\bf 2}&\lb{4}\\
\lbd{1\!0}{8}&5^3,4,1^8&{\bf 2}&\lb{5}\\
\lbd{1\!0}{9}&5^3\!,\!3,\!2^3\!,\!1^3&{\bf 3}&\lbd{5}{1}\\
\lbd{1\!0}{1\!0}&5^3,2^6&{\bf 3}&\lbd{5}{2}\\
\lbd{1\!0}{1\!1}&5^2\!,\!4^2\!,\!2^4\!,\!1&{\bf 3}&\lbd{5}{1} (\!\lbd{6}{1})\!\\ \hline
\end{array} & 
\begin{array}[t]{|l|l|l|l|l|}
\hline \# & \text{hom.~type}& \ell& \text{pr.}\!\\\hline
\lbd{1\!0}{1\!2}&5^2\!,\!4,\!3^3\!,\!2,\!1^2&{\bf 3}&\lbd{5}{1}\\
\lbd{1\!0}{1\!3}&5^2,3^5,2&{\bf 3}&\lbd{5}{2}\\
\lbd{1\!0}{1\!4}&5,4^3,3^2,2^2&{\bf 3}&\lbd{5}{1}\\
\lbd{1\!0}{1\!5}&4^6,1^3&{\bf 3}&\lbd{6}{1}\\
\lbd{1\!0}{1\!6}&4^5,3^2,1&{\bf 3}&\lbd{6}{3}\\\hline\end{array}\end{array}\]
\[\begin{array}{lll}\begin{array}[t]{|l|l|l|l|l|}
\hline \# & \text{hom.~type}& \ell& \text{pr.}\!\\\hline
\lb{1\!1}& 10,1^{20}&{\bf 1}&\lb{1}\\
\lbd{1\!1}{1}&9,2^9,1^3&{\bf 2}&\lb{2}(\!\lb{3})\\
\lbd{1\!1}{2}&8,3^5,2^2,1^3&{\bf 2}&\lb{3}\\
\lbd{1\!1}{3}&8,3^4,2^5&{\bf 2}&\lb{3}\\
\lbd{1\!1}{4}&7,4^3,3^2,1^5&{\bf 2}&\lb{4}\\
\lbd{1\!1}{5}&7,\!4^3\!,3,\!2^3,\!1^2&{\bf 3}&\lbd{4}{1}\\
\lbd{1\!1}{6}&7,4^2,3^3,2^3&{\bf 3}&\lbd{4}{1} (\!\lbd{5}{1})\!\\ \hline
\end{array} & 
\begin{array}[t]{|l|l|l|l|l|}
\hline \# & \text{hom.~type}& \ell& \text{pr.}\!\\\hline
\lbd{1\!1}{7}&7,4,3^6,1&{\bf 2}&\lb{4}\\
\lbd{1\!1}{8}&6,5^3,1^9&{\bf 2}&\lb{5}(\!\lb{6})\\
\lbd{1\!1}{9}&6,5^2\!,4,2^4\!,\!1^2&{\bf 3}&\text{\it 5\!.\!1}(\!\text{\it 6\!.\!1})\!\\
\lbd{1\!1}{1\!0}&6,5^2\!,3^3\!,2,\!1^3&{\bf 3}&\lbd{5}{1}\\
\lbd{1\!1}{1\!1}&6,5^2\!,3^2\!,2^4&{\bf 3}&\lbd{5}{2}\!\\
\lbd{1\!1}{1\!2}&6,\!5,\!4^2\!,\!3^2\!,\!2^2\!,\!1&{\bf 3}&\lbd{5}{1}\\
\lbd{1\!1}{1\!3}&6,4^5,1^4&{\bf 2}&\lb{5}\\ \hline
\end{array} & 
\begin{array}[t]{|l|l|l|l|l|}
\hline \# & \text{hom.~type}& \ell& \text{pr.}\!\\\hline
\lbd{1\!1}{1\!4}&6,4^3,3^4&{\bf 3}&\lbd{5}{1}\\
\lbd{1\!1}{1\!5}& 5^4,2^5&{\bf 3}&\text{\it 6\!.\!1}(\!\text{\it7\!.\!1})\!\\
\lbd{1\!1}{1\!6}&5^3,4,3^3,1^2&{\bf 3}&\text{\it 6\!.\!2}(\!\text{\it7\!.\!2})\!\\
\lbd{1\!1}{1\!7}&5^2,4^4,2,1^2&{\bf 3}&\lbd{6}{1}\\
\lbd{1\!1}{1\!8}&5^2,4^3,3^2,2&{\bf 3}&\lbd{6}{3}\\ \hline\end{array}\end{array}\]

There are $29$ types of degree $12$, each of length $\le 4$.

\[\begin{array}{lll}\begin{array}[t]{|l|l|l|l|l|}
\hline \# & \text{hom.~type}& \ell& \text{pr.}\!\\\hline
\lb{1\!2}&11,1^{22}&{\bf 1}&\lb{1}\\
\lbd{1\!2}{1\!}&10,2^{10},1^3&{\bf 2}&\lb{2}\\
\lbd{1\!2}{2\!}&9,3^6,2,1^4&{\bf 2}&\lb{3}\\
\lbd{1\!2}{3}&9,3^5,2^4,1&{\bf 2}&\lb{3}\\
\lbd{1\!2}{4}&8,4^4,3,1^6&{\bf 2}&\lb{4}\\
\lbd{1\!2}{5}&8,4^4,2^3,1^3&{\bf 3}&\lbd{4}{1}\\
\lbd{1\!2}{6}&8,\!4^3,3^2\!,2^3\!,1&{\bf 3}&\text{\it 4\!.\!1}(\!\text{\it 5\!.\!1})\!\\
\lbd{1\!2}{7}&8,4^2,3^5,1^2&{\bf 2}&\lb{4}\\
\lbd{1\!2}{8}&7,5^3,2^4,1^3&{\bf 3}&\text{\it 5\!.\!1}(\!\text{\it 7\!.\!1})\!\\
\lbd{1\!2}{9}&7,\!5^2\!,\!4,\!3^2\!,\!2^2\!,\!1^{\!2}\!&{\bf 3}&\lbd{5}{1}\\ \hline
\end{array} & 
\begin{array}[t]{|l|l|l|l|l|}
\hline \# & \text{hom.~type}& \ell& \text{pr.}\!\\\hline
\lbd{1\!2}{1\!0}&7,5^2,3^4,2^2&{\bf 3}&\lbd{5}{2}\\
\lbd{1\!2}{1\!1}&7,5,4^4,1^5&{\bf 2}&\lb{5}\\
\lbd{1\!2}{1\!2}&7,5,4^3\!,\!3,\!2^3&{\bf 3}&\lbd{5}{1}\\
\lbd{1\!2}{1\!3}&7,5,4^2\!,\!3^4\!,\!1&{\bf 3}&\lbd{5}{1}\\
\lbd{1\!2}{1\!4}&6^3,5,1^{10}&{\bf 2}&\lb{6}\\
\lbd{1\!2}{1\!5}& 6^3,4,2^4,1^3&{\bf 3}&\lbd{6}{\!1}\\
\lbd{1\!2}{1\!6}& 6^3,3^3,2,1^4&{\bf 3}&\lbd{6}{\!2}\\
\lbd{1\!2}{1\!7}&6^3,3^2,2^4,1&{\bf 3}&\lbd{6}{\!3}\\
\lbd{1\!2}{1\!8}&6^2\!,5^2\!,2^5\!,1&{\bf 3}&\text{\it 6\!.\!1}(\!\text{\it 7\!.\!1})\!\\
\lbd{1\!2}{1\!9}&6^2\!,5,4,3^3\!,\!1^3&{\bf 3}&\text{\it 6\!.\!2}(\!\text{\it 7\!.\!2})\!\\ \hline
\end{array} & 
\begin{array}[t]{|l|l|l|l|l|}
\hline \# & \text{hom.~type}& \ell& \text{pr.}\!\\\hline
\lbd{1\!2}{2\!0}&6^2\!,5,4,3^2\!,2^3&{\bf 3}&\text{\it 6\!.\!3}(\!\text{\it 7\!.\!3})\!\\
\lbd{1\!2}{2\!1}&6^2\!,4^4\!,2,1^3&{\bf 3}&\lbd{6}{1}\\
\lbd{1\!2}{2\!2}&6^2\!,4^3\!,3^2,2,1&{\bf 3}&\lbd{6}{3}\\
\lbd{1\!2}{2\!3}&6,5^3,3^3,2,1&{\bf 3}&\text{\it 6\!.\!2}(\!\text{\it 8\!.\!2})\\\lbd{1\!2}{2\!4}&6,5^2\!,4^3\!,2^2\!,1&{\bf 3}&\lbd{6}{1}\\
\lbd{1\!2}{2\!5}&6,5,4^4,3^2&{\bf 3}&\text{\it 6\!.\!3}(\!\text{\it 7\!.\!3})\\
\lbd{1\!2}{2\!6}&5^4,4^2,3,1^2&{\bf 3}&\lbd{7}{3}\\
\lbd{1\!2}{2\!7}&5^4,4,3^3&{\bf 4}&\lbd{7}{4}\\
\lbd{1\!2}{2\!8}&5^3,4^4,2&{\bf 4}&\lbd{7}{4}\\\hline\end{array}\end{array}\]
We could of course continue like this but the number of homaloidal types grows very quickly. We give below the homaloidal types of length $\ell\in \{2,\dots, 7\}$ of smallest degree:
\[\begin{array}{llll}
\begin{array}[t]{|l|l|l|}
\hline
\ell & d& \text{mult.} \\
\hline
2&4&2^3,1^3\\
3&7&3^4,2^3\\
\hline
\end{array} & \begin{array}[t]{|l|l|l|}
\hline
\ell & d& \text{mult.} \\
\hline
4&12&5^4,4,3^3\\
4&12&5^3,4^4,2\\
\hline
\end{array}& \begin{array}[t]{|l|l|l|}
\hline
\ell & d& \text{mult.} \\
\hline
5&16&6^5,5^3\\
6&27&11^4,10,6^4\\
\hline
\end{array}& \begin{array}[t]{|l|l|l|}
\hline
\ell & d& \text{mult.} \\
\hline
6&27&10^4,9^4,2\\
7&38&14^6,11^2,5\\
\hline
\end{array}\end{array}\]

\subsection{Automorphisms of the affine plane}

As explained before, there is a natural length in the group $\Aut(\A^2)$, since
this group is an amalgamated product of $\Aff_2 = \Aut (\p^2) \cap \Aut(\A^2)$ and $\Jonq_{p, \A^2}=\Jonq_p \cap \Aut(\A^2)$, where we fix a linear embedding $\A^2 \hookrightarrow \p^2$, and a point $p\in \p^2$ outside the image. By construction, the length of an element of $\Aut(\A^2)$, viewed in the amalgamated product, is at least equal to its length in $\Bir(\p^2)$. We show in Proposition~\ref{Prop:LengthThesameinA2BirP2} that the two lengths are in fact equal, using the following result.

\begin{lemma}  \label{Lemm:CompositionWithoutCommonBasepts}
Let $f,g\in \Bir(\p^2)$ be elements such that $\Base(f)\cap \Base(g^{-1})=\emptyset$. Then,
\[\begin{array}{l}
\deg(f \circ g)=\deg(f) \cdot \deg(g), \quad \lvert \Base(f \circ g)\rvert=\lvert \Base(f)\rvert+\lvert \Base(g)\rvert  \quad \text{and}\\
\lgth(f \circ g) = \lgth(f)+\lgth(g).\end{array}\]
\end{lemma}

\begin{proof}
The two first equalities follow from Lemma~\ref{Lemm:CompositionWinftyWithoutCommonBasepts}\ref{CompWinftyWithoutCommonBasepts1}. We
prove the third one by induction on $\lgth(g)$, the case where $\lgth(g)=0$ being obvious, since $g\in \Aut(\p^2)$ in that case.

We then consider the case where $d=\deg(g)=\deg(g^{-1})>1$, and take $\varphi\in \Jonq\subseteq \Bir(\p^2)$ such that $g_1=g\circ \varphi^{-1}$ is a predecessor of $g$. Then, $g_1^{-1}(e_0)=\varphi(g^{-1}(e_0))$ is a predecessor of $g^{-1}(e_0)$ (Corollary~\ref{Cor:EquiTwoPred}). This implies that 
$\varphi((f\circ g)^{-1}(e_0))$ is a predecessor of $(f\circ g)^{-1}(e_0)$ (Lemma~\ref{Lemm:CompositionWinftyWithoutCommonBasepts}) and thus that  $f\circ g_1=f\circ g\circ \varphi^{-1}$ is a predecessor of $f\circ g$ (Corollary~\ref{Cor:EquiTwoPred}). Hence, we obtain  $\lgth(g_1)=\lgth(g)-1$ and $\lgth(f\circ g_1)=\lgth(f\circ g)-1$ (Theorem~\ref{TheMainTheorem}).

Since $\Base(\varphi)\subseteq \Base(g)$ (Lemma~\ref{Lem:Pred}),
we have 
$\Base(g_1^{-1})=\Base ( \varphi \circ g^{-1} )  \subseteq   \Base(g^{-1} )$  (Corollary~\ref{corollary: inclusion of the base locus of a composition in a special case}),
and thus $\Base(f)\cap \Base(g_1^{-1})=\emptyset$.
We can thus apply the induction hypothesis to get $\lgth(f\circ g_1)=\lgth(f)+\lgth(g_1)$. This achieves the proof.
\end{proof}

\begin{proposition}\label{Prop:LengthThesameinA2BirP2}
Let $f\in \Aut(\A^2)$. Taking an inclusion $\Aut(\A^2) \hookrightarrow \Bir(\p^2)$ given by a linear embedding $\A^2 \hookrightarrow \p^2$, the length of $f$ in $\Bir(\p^2)$ is equal to its length in the amalgamated product $\Aut(\A^2)$.
\end{proposition}

\begin{proof}
We write $f=a_n \circ \varphi_n \circ  \cdots \circ a_1 \circ \varphi_1 \circ a_0$, where each $a_i$ is an element of $\Aff_2$ and each $\varphi_i$ is an element of $\Jonq_{p, \A^2}$. If $f$ belongs to $A$, then its length is $0$ in $\Bir(\p^2)$ and in the amalgamated product, so we can assume that $n\ge 1$, that $\varphi_i\in \Jonq_{p, \A^2}\setminus \Aff_2$ for $i=1,\dots,n$ and that $a_i\in \Aff_2\setminus \Jonq_{p, \A^2}$ for $i=1,\dots,n-1$. We then need to prove that $\lgth(f)=n$. To do this, we first observe that $\varphi_i$ and $(\varphi_i)^{-1}$ contract exactly one curve of $\p^2$, namely the line $L_\infty=\p^2\setminus \A^2$. It implies that $(\varphi_i)^{-1}$ and $\varphi_i$ have only one proper base-point, and since both preserve lines through $p$, then $p$ is the unique proper base-point of $\varphi_i$ and $(\varphi_i)^{-1}$. Moreover, each $a_i$ is an automorphism of $\p^2$ that does not fix $p$, for $i=1,\dots,n-1$. Hence by induction the element $a_i \circ  \varphi_i  \circ \cdots \circ a_1 \circ \varphi_1 \circ a_0$ contracts the line $L_\infty$ onto $a_i(p)\not=p$, which is the unique proper base-point of $(a_i \circ  \varphi_i  \circ \cdots \circ a_1 \circ \varphi_1 \circ a_0)^{-1}$. In particular, $(a_i \circ  \varphi_i  \circ \cdots \circ a_1 \circ \varphi_1 \circ a_0)^{-1}$ and $\varphi_{i+1}$ do not have any common base-point, so $\lgth( a_{i+1} \circ \varphi_{i+1} \circ \cdots \circ a_1 \circ \varphi_1 \circ a_0) = \lgth(a_i \circ  \varphi_i \circ \cdots \circ a_1 \circ \varphi_1 \circ a_0 )+1$ for each $i$. This provides the result.
\end{proof}

\subsection{Decreasing the length and increasing the degree via a single Jonqui\`eres element}
In this section, we mainly provide an example of a Cremona transformation $f$ and a Jonqui\`eres element $\varphi$ such that $\lgth (f \circ \varphi^{-1} ) = \lgth (f) -1$ and $\deg (f \circ \varphi^{-1} ) >  \deg (f)$:

\begin{proposition}  \label{Proposition:Bertiniexample}
Fixing $8$ general points $p_0,\dots,p_7\in \p^2$, the following hold:
\begin{enumerate}
\item \label{Existence-of-the-Bertini-involution}
There exists a birational involution $f\in \Bir(\p^2)$
of homaloidal type $(17;6^8)$
such that 
\[\begin{array}{c}f(e_0)=17e_0-\sum\limits_{j=0}^7 6e_{p_j} \text{ and } f(e_i)=6e_0-e_{p_i}-2\sum\limits_{j=0}^7 e_{p_{j}}\text{ for }i=0,\dots,7.\end{array}\]
\item \label{Bertiniexample1}
For each general point $q \in \p^2$, there exists a Jonqui\`eres element
$\varphi_q$ such that
\[\varphi_q^{-1}(e_0)=5e_0-4e_{p_0}-e_{p_1}-\dots-e_{p_7}-e_q.\]
\item \label{Bertiniexample2}
For all $f$ and $\varphi_q$ as in \ref{Existence-of-the-Bertini-involution} and \ref{Bertiniexample1}, the birational map $f_q=f\circ \varphi_q^{-1}$ satisfies:
\[ \lgth(f_q) = 4 < 5 =\lgth(f)  \quad  \text{and} \quad \deg(f_q) = 19 > 17 = \deg (f).\]
\item\label{Bertiniexample3}
For any two distinct general points $q,q'\in \p^2$, and all choices of $f_q,f_q'$ as in \ref{Bertiniexample2}, we have 
\[f_{q'}\not\in f_{q}\Aut(\p^2).\]
\end{enumerate} 
\end{proposition}

\begin{proof}
\ref{Existence-of-the-Bertini-involution}-\ref{Bertiniexample1}:
Let $U_D\subseteq (\p^2)^8$ be the subset of $8$-uples $(p_0,\dots,p_7)$ such that the points $p_i$ are pairwise distinct and the blow-up $\pi\colon X\to \p^2$ of $p_0,\dots,p_7$ is a Del Pezzo surface. By Lemma~\ref{Lemm:DPopen} below, $U_D$ is a dense open subset of $(\p^2)^8$.
For each $(p_0,\dots,p_7)\in U_D$, we can then define $\hat{f}\in \Aut(X)$ to be the \emph{Bertini involution} $($see \cite[\S 8.8.2]{DolgachevBook}$)$ and obtain that $f=\pi\circ \hat{f}\circ\pi^{-1}\in \Bir(\p^2)$ has degree $17$ and satisfies the conditions given in \ref{Existence-of-the-Bertini-involution}.
Since $(5; 4, 1^8)$ is a homaloidal type, Corollary \ref{corollary:existence-of-a-Cremona-transformation-with-assigned-homaloidal-type-on-a-dense-open-subset} yields a dense open subset $V_9\subseteq(\p^2)^9$ such that for all  each $(p_0,\dots,p_7,q)\in V_9$ there exists $\varphi_q \in \Bir(\p^2)$ satisfying the condition given in \ref{Bertiniexample1}.
The element $\varphi_q$ is  Jonqui\`eres by Lemma~\ref{Lemm:JoCremWinfty}.
We denote by $V$ the open set $V=V_9\cap (U_D\times \p^2)\subseteq(\p^2)^9$, by $\kappa \colon (\p^2)^9\to (\p^2)^8$  the projection on the first eight factors, and by $U\subseteq(\p^2)^8$ the open set given by $U=\kappa(V)$.
For each $(p_0,\dots,p_7)\in U$, the assertions \ref{Existence-of-the-Bertini-involution} and \ref{Bertiniexample1} are then satisfied.

\ref{Bertiniexample2}: We take $f$ and $\varphi_q$ as in \ref{Existence-of-the-Bertini-involution} and \ref{Bertiniexample1}, write $f_q=f\circ \varphi_q^{-1}$, and observe that
\[f_q(e_0)=f(5e_0-4e_{p_0}-e_{p_1}-\dots-e_{p_7}-e_q)=19e_0-4e_{p_0}-7e_{p_1}-\dots-7e_{p_7}-f(e_q),\]
whence the homaloidal type of $f_q^{-1}$ is $(19;7^7,4,1)$. In particular, we have $\deg(f_q)=\deg(f_q^{-1})=19$. Recall that we also have $\lgth (f_q)=\lgth(f_q^{-1})$. 
Therefore, to show that $\lgth (f) = 5$ and $\lgth(f_q)=4$, it suffices to look at Examples~\ref{Example:SomeSequences} and~\ref{Example:SequenceBertini}, where we already observed that Algorithm~\ref{Algo:AlgorithInWeyl} applied to $(19;7^7,4,1)$ and $(17;6^8)$ yields the following sequences of homaloidal types:
\[\begin{array}{lllllllllll}
(17;6^8)&\predecessor &(14;6,5^6,3) &\predecessor &(8;3^7) &\predecessor &(5;2^6)& \predecessor &(3;2,1^4)& \predecessor&(1);
\\
(19;7^7,4,1) &\predecessor& (13;5^6,4,1^2)&\predecessor &(8;4,3^5,1^2)&\predecessor&(4;3,1^6)&\predecessor&(1).\end{array}\]
One could also check that the homaloidal type of $f_q$ is $(19;11,8,5^7)$. However, it is useless for the proof we propose.

\ref{Bertiniexample3}: Take two general points $q,q'$ and suppose that $f_{q'}= f_{q}\circ \alpha$ for some $\alpha\in \Aut(\p^2)$. Then, $f_q(e_0)=f_{q'}(e_0)$, which implies that $f(e_q)=f(e_{q'})$, whence $q=q'$.
\end{proof}

The following result is classical:

\begin{lemma}  \label{Lemm:DPopen}
Let $\pi\colon X\to \p^2$ be the blow-up of  $r$ distinct proper points, with $1\le r\le 8$. Then, $X$ is a Del Pezzo surface if and only if the following conditions are satisfied: no $3$ of the points are collinear, no $6$ lie on the same conic, and no $8$ of the points lie on the same cubic singular at one of the points. Moreover, these conditions correspond to a dense open subset of $(\p^2)^r$, and are thus satisfied for sufficiently general points.
\end{lemma}

\begin{proof}
Follows from \cite[Proposition 8.1.25]{DolgachevBook}.
\end{proof}

\subsection{The number of predecessors is not uniformly bounded}
Let $f \in \Bir(\p^2)$ be a Cremona transformation.
If $\deg (f) \leq 4$, one can check that $f$ admits a unique predecessor up to right multiplication by an element of $\Aut(\p^2)$.
If $\deg (f) \leq 5$, then $f$ may admit more than one predecessor modulo $\Aut(\p^2)$
(but all having the same homaloidal type, by Lemma~\ref{Lem:Pred}). Example~\ref{Example5} gives an example of degree $5$ with two predecessors having distinct configurations of base-points, and Lemma~ \ref{lemm:UnboundedNumberPred}
shows that the number of
predecessors modulo $\Aut(\p^2)$ is not uniformly bounded.

\begin{example}  \label{Example5}
Let us consider the birational involution $g\in \Bir(\p^2)$ given by 
\[g\colon [x:y:z]\mapsto \left[\frac{x}{-\frac{1}{x}+\frac{1}{y}+\frac{1}{z}}: \frac{y}{\frac{1}{x}-\frac{1}{y}+\frac{1}{z}}: \frac{z}{\frac{1}{x}+\frac{1}{y}-\frac{1}{z}}\right] = [xvw\!:\! yuw\!:\! zuv]\]
where $u =-yz + xz+ xy$, $v =yz - xz+ xy$, and $w= yz + xz - xy$. We
see that $g$ is of degree $5$. It has moreover $6$ base-points of multiplicity $2$, namely the $3$ points $p_1=[1:0:0]$, $p_2=[0:1:0]$, $p_3=[0:0:1]$, and $3$ other points $q_1,q_2,q_3$, where each $q_i$ is infinitely near to $p_i$.
The homaloidal type  of $g$ is therefore $(5 ; 2^6)$.

The algorithm consists of applying a cubic birational transformation whose linear system consists of cubics singular at one of the $p_i$ and passing through
$4$ of the remaining $5$ base-points.
The predecessors of $g$ are thus of degree $3$. However, we get distinct classes up to automorphism of $\p^2$, depending on
our choice of the $4$ points.

Denoting by $\rho_1,\rho_2\in \Bir(\p^2)$ the birational maps of degree $3$ given by
\[\begin{array}{rrcl}
\rho_1\colon& [x:y:z]&\mapsto& \left[\frac{z}{x(\frac{1}{x}+\frac{1}{y}-\frac{1}{z})}: y: z \right] = [yz^2: yw: zw]  \text{ and }   \vspace{0.15cm}\\
\rho_2 \colon& [x:y:z]&\mapsto& \left[\frac{1}{-\frac{1}{x}+\frac{1}{y}+\frac{1}{z}}: y: z\right]= [xyz: yu: zu], \end{array}\]
we observe that $\alpha_1=\rho_1g{\rho_1}^{-1}$ and $\alpha_2=\rho_2 g{\rho_2}^{-1}$ are the two linear involutions 
\[\alpha_1\colon [x:y:z]\mapsto [x:y:2x-z]\text{ and }\alpha_2\colon [x:y:z]\mapsto [x:2x-y:2x-z].\]

For each $i \in \{ 1,2\}$ the map $\rho_i\in \Bir(\p^2)$ is a Jonqui\`eres map, preserving a general line through $[1:0:0]$. This implies that $\psi_i = {\rho_i}^{-1} \circ \alpha_i=g \circ {\rho_i}^{-1}$ is a predecessor of $g$.

The linear system of $\rho_1$ consists of cubics singular at $p_1$ and passing through $p_2,p_3,q_1,q_2$. The linear system of $\rho_2$ consists of cubics singular at $p_1$ and passing through $p_2,p_3,q_2,q_3$. The configuration of the points being different (for $\rho_2$, there is a tangent direction fixed at the singular point, contrary to $\rho_1$), this shows that $\psi_1 \notin \Aut(\p^2) \circ \psi_2 \circ \Aut(\p^2)$. 
\end{example}

\begin{lemma}  \label{lemm:UnboundedNumberPred}
For each integer $i\ge 1$, there exists $f\in \Bir(\p^2)$ which has
at least $i$ predecessors up to left and right composition with
elements of $\Aut(\p^2)$.
\end{lemma}

\begin{proof}
We choose $n$ such that $2n-1\ge i$ and $2n-1\ge 5$, and then observe that
$\chi=(n^2+1;n^2-n+1,n^{2n-1},1^{2n-1})$ is a homaloidal type, whose predecessor is $\chi_1=(n;n-1,1^{2n-2})$. We take $4n-1$ general points $p_0,p_1,\dots,p_{2n-1},q_1,\dots,q_{2n-1}\in \p^2$ and choose $f\in \Bir(\p^2)$ such that $f^{-1}(e_0)=(n^2+1)e_0-(n^2-n+1)e_{p_0}-\sum_{i=1}^{2n-1} n e_{p_i}-\sum_{i=1}^{2n-1} e_{q_i}$ (which exists by Lemma~\ref{Lemm:SymBirSim}). 

For each $j\in \{1,\dots,2n-1\}$, there exists a Jonqui\`eres element $\varphi_j\in \Bir(\p^2)$ such that $(\varphi_j)^{-1}(e_0)=ne_0-(n-1)e_{p_0}-\sum_{i=1}^{2n-1} e_{p_i}-e_{q_j}$ (again by Lemma~\ref{Lemm:SymBirSim}). Then, $f_j=f\circ \varphi_j^{-1}$ is a predecessor of $f$ (follows from Lemma~\ref{Lemm:AlgoFirstStepS}\ref{eachsinS} and Lemma~\ref{Lemma:TwoTypesPred}).

As $\Base(\varphi_j)\subseteq \Base(f)$, we get $\Base(f_j^{-1})=\Base(\varphi_j\circ f^{-1})\subseteq \Base(f^{-1})$ (Corollary~\ref{corollary: inclusion of the base locus of a composition in a special case}). Moreover, $f_j$ is Jonqui\`eres of degree $n$ (because of the type of $\chi_1$), so the same holds for $f_j^{-1}$. 

It remains to see that if $j,k\in \{1,\dots,2n-1\}$ are such that $j\not=k$, then there are no elements $\alpha,\beta\in \Aut(\p^2)$ such that $f_j=\alpha\circ f_k\circ \beta$. Indeed, otherwise we would have $f_j(e_0)=\alpha(f_k(e_0))$, so $\alpha$ sends the $2n-1$ base-points of $f_k^{-1}$ onto those of $f_j^{-1}$, respecting the multiplicities. Note that $f_j(e_0)\not=f_k(e_0)$, because $\varphi_j^{-1} (e_0) \not= \varphi_k^{-1} (e_0)$.
The map $\alpha\in \Aut(\p^2)$ has then to send a sequence of $2n-1$ points of  $\Base(f^{-1})$ onto another sequence of $2n-1$ points of $\Base(f^{-1})$. This is impossible, as the points are general points and $2n-1 \ge 5 $.
\end{proof}

\begin{remark}\label{Rem:BoundPredDegree}
It directly follows from Lemma~\ref{Lem:Pred} that the number of predecessors of $f\in \Bir(\p^2)$ is bounded by some number depending only on $\deg(f)$. This follows from Lemma~\ref{Lem:Pred}\ref{Lem:Pred3} and from the fact that the number of base-points of $f$ is at most $\deg(f)+2$ if $f$ is not Jonqui\`eres (\cite[Lemma 39]{BCM15}). Giving a meaningful bound does not seem so easy. The number of points of maximal multiplicity is at most $8$ (this follows from the Noether inequality of Lemma~\ref{Lemma:Hudsonalgorithm} and  the Noether equalities of Lemma~\ref{Lemm:EasyWeyl}), and the number of predecessors for each base-point of maximal multiplicity is bounded by the choice of the base-points of multiplicity one (for the Jonqui\`eres element) among the remaining base points of $f$. This choice seems to be smaller when the number of base-points of maximal multiplicity is large.
\end{remark}

\subsection{Reduced decompositions of arbitrary lengths}
Recall that a reduced decomposition of an element $f\in \Bir(\p^2)$ is a product $f=\varphi_n \circ \cdots \circ \varphi_1$ of Jonqui\`eres elements such that $\varphi_{i+1}\circ \varphi_{i}$ is not Jonqui\`eres for $i=1,\dots,n-1$. One can of course always obtain a reduced decomposition by starting with any decomposition, simply replacing $\varphi_i$ and $\varphi_{i+1}$ with their product if this one is Jonqui\`eres. 
Proposition~\ref{Prop:Unboudedlengthreduceddec} shows that the length of reduced decompositions is unbounded.

We begin with the following classical result:

\begin{lemma}\label{Lemm:Quadraticstd}
Let $p_0,p_1,p_2$ be three non-collinear points of $\p^2$. There is a quadratic birational involution $\nu\in \Bir(\p^2)$ preserving the pencil of lines through any of the three base-points and satisfying $\Base(\nu)=\{p_0,p_1,p_2\}$. 
\end{lemma}

\begin{proof}
It suffices to change coordinates to have $\{p_0,p_1,p_2\}=\{[1:0:0],[0:1:0],[0:0:1]\}$ and to choose $\nu\colon [x:y:z]\dasharrow [yz:xz:xy]$.
\end{proof}

\begin{lemma}   \label{Lemm:Prodpsi}
Let $p_0,\dots,p_5\in \p^2$ be six distinct points such that no $3$ of them are collinear and not all lie on the same conic. Then, there exist three quadratic birational involutions $\psi_1,\psi_2,\psi_3\in \Bir(\p^2)$, each having
$($three$)$
proper base-points, such that the following hold:\begin{enumerate}
\item \label{Prodpsis1}
$\Base(\psi_1)=\{p_0,p_1,p_2\}$;
\item \label{Prodpsis2}
$\Base(\psi_2)\cap \Base(\psi_1)=\Base(\psi_3)\cap \Base(\psi_2)=\emptyset$;
 \item\label{Prodpsis3} $(\psi_3\circ \psi_2\circ \psi_1)^{-1}(e_0) =5e_0 -2 \sum_{i=0}^5 e_{p_i}$.
\end{enumerate}
\end{lemma}
 
\begin{proof}
The blow-up $\pi_1\colon X\to \p^2$ of the six points $p_0,\dots,p_5$ is a Del Pezzo surface (Lemma~\ref{Lemm:DPopen}) and there exists a quadratic birational involution $\psi_1\in \Bir(\p^2)$ with $\Base(\psi_1)=\{p_0,p_1,p_2\}$ (Lemma~\ref{Lemm:Quadraticstd}). We then write $q_i=\psi_1(p_i)\in \p^2$ for $i=3,4,5$. We now prove that $\pi_2=\psi_1\circ \pi_1\colon X\to \p^2$ is the blow-up of $p_0,p_1,p_2,q_3,q_4,q_5$: we have a commutative diagram
\[ \xymatrix@R=0.4cm@C=1.5cm{
&& X\ar@/_1.5pc/[ddll]_{\pi_1}\ar@/^1.7pc/[ddrr]^{\pi_2}\ar[dl]_{\tau_1}\ar[dr]^{\tau_2}\\
 &Y \ar[dl]_{\eta}\ar[rr]^{\hat\psi_1}_{\simeq} &&Y\ar[dr]^{\eta} \\ 
\p^2 \ar@{-->}[rrrr]^{\psi_1} & &&& \p^2} \]
where $\eta$ is the blow-up of $p_0,p_1,p_2$, $\hat\psi_1\in \Aut(Y)$ is an automorphism of order $2$, $\tau_1$ is the blow-up of $\{\eta^{-1}(p_i)\mid i=3,4,5\}$, and $\tau_2$ is the blow-up of $\{\eta^{-1}(q_i)=\hat\psi_1(\eta^{-1}(p_i))\mid i=3,4,5\}$.

Because $X$ is a Del Pezzo surface, the points $q_3,q_4,q_5$ are not collinear (Lemma~\ref{Lemm:DPopen}), so there is a quadratic birational involution $\psi_2\in \Bir(\p^2)$ with $\Base(\psi_2)=\{q_3,q_4,q_5\}$ (Lemma~\ref{Lemm:Quadraticstd}). We then write $q_i=\psi_2(p_i)\in \p^2$ for $i=0,1,2$ and obtain similarly that $q_0,\dots,q_5\in \p^2$ are such that no $3$ of them are collinear
(since $\psi_2 \circ \psi_1 \circ \pi_1 \colon X \to \p^2$ is the blow-up of $q_0, \ldots,q_5$).

We now choose a quadratic birational involution $\psi_3\in \Bir(\p^2)$ with $\Base(\psi_3)=\{q_0,q_1,q_2\}$ (again by Lemma~\ref{Lemm:Quadraticstd}). It remains to calculate
\[\begin{array}{rcl}
 (\psi_3\circ \psi_2\circ \psi_1)^{-1}(e_0)&=&\psi_1(\psi_2(2e_0-e_{q_0}-e_{q_1}-e_{q_2}))\\
&=&\psi_1(4e_0-e_{p_0}-e_{p_1}-e_{p_2}-2e_{q_3}-2e_{q_4}-2e_{q_5} 
)\\
&=&5e_0-2 \sum_{i=0}^5e_{p_i},
\end{array}\]
where we have used the fact that $\psi_1 (e_{p_0}) =e_0 - e_{p_1} - e_{p_2}$, $\psi_1 (e_{p_1}) =e_0 - e_{p_0} - e_{p_2}$, and $\psi_1 (e_{p_2}) =e_0 - e_{p_0} - e_{p_1}$ (see Example~\ref{ExampleSigma}).
\end{proof}

\begin{corollary}  \label{Coro:SixQuad}
There exist
quadratic birational maps $\varphi_1,\dots,\varphi_6$, each having $($three$)$ proper base-points, such that
\begin{enumerate}
\item
$\varphi_6\circ \varphi_5\circ \varphi_4\circ \varphi_3\circ \varphi_2\circ \varphi_1=\mathrm{id}$;
\item
$\Base(\varphi_i^{-1})\cap \Base(\varphi_{i+1})=\emptyset$ for $i=1,\dots,5$.
\end{enumerate}
\end{corollary}

\begin{proof}
Let $p_0,\dots,p_5\in \p^2$ be six distinct points such that no $3$ are collinear and not all lie on the same conic. Then, choose $\psi_1,\psi_2, \psi_3$ as in Lemma~\ref{Lemm:Prodpsi}. Recall that we have in particular $(\psi_3\circ \psi_2\circ \psi_1)^{-1}(e_0) =5e_0 -2 \sum_{i=0}^5 e_{p_i}$.
Applying Lemma~\ref{Lemm:Prodpsi} to the same points, but taken in the order $p_3,p_4,p_5, p_0,p_1,p_2$,
we get quadratic birational involutions $\psi_1',\psi_2',\psi_3'\in \Bir(\p^2)$, each having three proper base-points, such that $\Base(\psi_1')=\{p_3,p_4,p_5\}$,  $\Base (\psi_2') \cap \Base (\psi_1') = \Base (\psi_3') \cap \Base (\psi_2') = \emptyset$ and 
$(\psi_3' \circ \psi_2' \circ \psi_1')^{-1}(e_0) = 5e_0-2 \sum_{i=0}^5 e_{p_i}. 
$ 
  
The birational map $\alpha=\psi_3'\circ \psi_2'\circ \psi_1' \circ \psi_1\circ \psi_2\circ \psi_3$ satisfies then $\alpha^{-1}(e_0)=e_0$, so that it is
an automorphism of $\p^2$. It remains to choose $(\varphi_6,\dots,\varphi_1):=(\psi_3',\psi_2',\psi_1',\psi_1,\psi_2,\psi_3\circ \alpha^{-1})$ to obtain the result.
\end{proof}

\begin{proposition}\label{Prop:Unboudedlengthreduceddec}
For each $f\in \Bir(\p^2)$ the set of reduced decompositions of $f$ has unbounded length.
\end{proposition}

\begin{proof}
To prove the result, we start with a reduced decomposition $f=\varphi_n \circ \cdots \circ \varphi_1$, and construct another one, with length $\ge n+5$.  To do this, we take quadratic birational maps $\psi_1,\dots,\psi_6\in \Bir(\p^2)$, each having three proper base-points, such that $\psi_6 \circ \dots \circ \psi_1= \mathrm{id}$ and $\Base(\psi_i^{-1})\cap \Base(\psi_{i+1})=\emptyset$ for $i=1,\dots,5$ (which exist by Corollary~\ref{Coro:SixQuad}). For $i=1,\dots,5$, we observe that $\lgth(\psi_{i+1}\circ \psi_i)=2$ (Lemma~\ref{Lemm:CompositionWithoutCommonBasepts}), so $\psi_{i+1}\circ \psi_i$ is not Jonqui\`eres.  Replacing all $\psi_i$ with $\alpha \circ \psi_i \circ \alpha^{-1}$ for some general $\alpha\in \Aut(\p^2)$, we can assume that $\Base(\varphi_n^{-1})\cap \Base(\psi_1)=\emptyset$, which implies that $\lgth(\psi_1\circ \varphi_n)=\lgth(\varphi_n)+1$ (Lemma~\ref{Lemm:CompositionWithoutCommonBasepts}) and is thus not Jonqui\`eres if $\varphi_n\not\in \Aut(\p^2)$. In this latter case, we obtain a reduced decomposition of $f$ of length $n+6$ as $f=\psi_6\circ \dots \circ \psi_1\circ \varphi_n \circ \cdots \circ \varphi_1$. The last case is when $\varphi_n\in \Aut(\p^2)$. This implies that $n=1$, as otherwise $\varphi_{n}\circ \varphi_{n-1}$ would be Jonqui\`eres. Hence, $f=\varphi_n\in \Aut(\p^2)$. In this case, it suffices to write $f=  \psi_6'\circ \psi_5 \dots \circ \psi_1$ with $\psi_6'=f\circ \psi_6$, to get a reduced decomposition of
length $6=n+5$.
\end{proof}

\subsection{Examples of dynamical lengths}

\begin{lemma}  \label{Lemma:AutA2}
The element $\kappa\in \Bir(\p^2)$ given by $[x:y:z]\dasharrow [yz+x^2:xz:z^2]$ satisfies
\[\lgth(\kappa^a)=a,\quad \deg(\kappa^a)=2^a,\quad \lvert \Base(\kappa^a) \lvert =3a\] for each integer $a\ge 1$.
 In particular, we have
  \[\dlgth(\kappa^m)=m,\quad \lambda(\kappa^m)=2^m,\quad \mu(\kappa^m)=3m\] for each $m\ge 1$.
\end{lemma}

\begin{proof}
As $\kappa$ is a Jonqui\`eres element of degree $2$, we have $\lgth(\kappa)=1$, $\deg(\kappa)=2$ and $\lvert \Base(\kappa^a) \lvert =3$ (this last assertion follows from Noether equalities, see Lemma~\ref{Lemm:EasyWeyl}). 

Denoting by $L\subset \p^2$ the line given by $z=0$, the restriction of $\kappa$ is automorphism of  $\p^2\setminus L\simeq \A^2$, so the same holds for $\kappa^a$, for each $a\in \Z$. There can then be at most one proper base-point of $\kappa^a$, namely the image by $\kappa^{-a}$ of the line  $z=0$. We check that $\kappa$ contracts $L$ onto $q=[1:0:0]$, which is then the unique proper base-point of $\kappa^{-1}$, and that $p=[0:1:0]$ is the unique proper base-point of $\kappa$. This implies that $\kappa^a$ contracts the line $L$ onto $q$ for each $a\ge 1$, and thus $q$ is the unique proper base-point of $\kappa^{-a}$ for each $a\ge 1$. We obtain, for each $a\ge 1$, that $\Base(\kappa)\cap \Base(\kappa^{-a})=\emptyset$. Lemma~\ref{Lemm:CompositionWithoutCommonBasepts} then yields 
$\deg(\kappa^{a+1})=\deg(\kappa)\cdot \deg(\kappa^a)$, $\lvert \Base(\kappa^{a+1})\rvert=\lvert \Base(\kappa)\rvert+\lvert \Base(\kappa^a)\rvert$ and $\lgth(\kappa^{a+1})=\lgth(\kappa)+ \lgth(\kappa^a)$. This provides the result, by induction over $a$.
\end{proof}

\begin{lemma}\label{Lemma:1213} Choosing $\sigma,\alpha_2,\alpha_3\in \Bir(\p^2)$ as $\sigma\colon [x:y:z]\dasharrow [yz:xz:xy]$,
\[\alpha_2\colon [x:y:z]\mapsto [z+y:x+z:y]\text{ and }\alpha_3 \colon [x:y:z]\mapsto [y+z: x: y]\]
we obtain
$\dlgth(\alpha_2\sigma)=\frac{1}{2}$ and $\dlgth(\alpha_3\sigma)=\frac{1}{3}$.
\end{lemma}

\begin{proof}
The birational involution $\sigma=\sigma^{-1}$ is quadratic with base-points $p_1=[1:0:0]$, $p_2=[0:1:0]$, $p_3=[0:0:1]$ and its action on $\ZZ_{\p^2}$ satisfies
\[\sigma(e_0)=2e_0-e_{p_1}-e_{p_2}-e_{p_3},\ \sigma(e_{p_i})=e_0-e_{p_1}-e_{p_2}-e_{p_3}+e_{p_i}, i=1,2,3,\]
as in Example~\ref{ExampleSigma}. Writing $p_4=[1:0:1]$, $p_5=[1:1:0]$, one gets 
\[\alpha_2(p_1)=p_2,\ \alpha_2(p_2)=p_4,\ \alpha_2(p_3)=p_5\text{ and }\alpha_3(p_1)=p_2,\ \alpha_3(p_2)=p_4 ,\ \alpha_3(p_3)=p_1.\]
Since $p_4$ and $p_5$ are general points of the lines contracted by $\sigma$ onto respectively $p_2$ and $p_3$, we find $\sigma(e_{p_4})=e_{q_2}$ and $\sigma(e_{p_5})=e_{q_3}$, where $q_2,q_3$ are points infinitely near to $p_2$ and $p_3$ respectively. In particular, writing 
\[\begin{array}{rcll}
T_2&=&\left\{\bigoplus \Z e_q\mid q\in \B(\p^2),\ q\text{ is infinitely near or equal to }p_4 \text{ or }p_5\right\}&\subseteq \ZZ(\p^2),\\
T_3&=&\left\{\bigoplus  \Z e_q\mid q\in \B(\p^2),\ q\text{ is infinitely near or equal to }p_4\right\}&\subseteq \ZZ(\p^2),\end{array}\] and writing $V_i=\Z e_0 \oplus \Z e_{p_1} \oplus \Z e_{p_2} \oplus \Z e_{p_3} \oplus T_i$, one observes that 
\[\alpha_i\sigma(T_i)\subseteq T_i\text{ and }\alpha_i\sigma(V_i)\subseteq V_i\text{ for }i=2,3.\]
We then get, for $i=2,3$, a linear map $V_i/T_i\to V_i/T_i$ given by a matrix $M_i$ with respect to the basis $e_0,e_{p_1},e_{p_2},e_{p_3}$, as follows :
\[M_2=\left(\begin{array}{rrrr}
2 & 1 & 1 & 1\\
0 & 0 & 0 & 0\\
-1 & 0 & -1 & -1 \\
0 & 0 & 0& 0\end{array}\right)\text{ and }M_3=\left(\begin{array}{rrrr}
2 & 1 & 1 & 1\\
-1 & -1 & -1 & 0\\
-1 & 0 & -1 & -1 \\
0 & 0 & 0& 0\end{array}\right).\]
We then compute 
\[ A:= (M_2)^2=\left(\begin{array}{rrrr}
3 & 2 & 1 & 1\\
0 & 0 & 0 & 0\\
-1 & -1 & 0 & 0 \\
0 & 0 & 0& 0\end{array}\right)\text{ and }
B:= (M_3)^3 = \left(\begin{array}{rrrr}
3 & 1 & 1 & 2 \\
-1 & 0 & -1 & -1 \\
-1 & 0 & 0 & -1 \\
0 & 0 & 0& 0\end{array}\right).\]
Let us check that $f:= (\alpha_2 \sigma)^2$ and $g:= (\alpha_3 \sigma)^3$ have dynamical lengths equal to $1$.

The expression of $A$ shows us that $\deg f = 3$ and $\comult (f) \leq 3 -2 =1$, so that $f$ is a Jonqui\`eres transformation. Set $d_{-1} =0$, $d_0=1$ and $d_n = 3 d_{n-1} - d_{n-2}$ for $n \geq 1$. A straightforward induction on $n$ would show that for each
non-negative
integer $n$, the coefficients $(1,1)$ and $(1,3)$ of $A^n$ satisfy:
\[ (A^n)_{1,1} = d_n, \quad (A^n)_{1,3} = d_n - d_{n-1}.\]
Since $2 (d_n- d_{n-1}) +1 > d_n$, it follows from Corollary~\ref{Coro:PositivityOrbit} that $d_n-d_{n-1}$ is the highest multiplicity of $f^n$. Therefore,
a predecessor $g_n$ of $f^n$ satisfies
\[ \deg g_n \geq \comult (f^n)  = d_n - (d_n-d_{n-1}) = d_{n-1} = \deg f^{n-1}.\]
This proves that $f^{n-1}$ is a predecessor of $f^n$, so that $\lgth (f^n) =\lgth (f^{n-1}) + 1$, proving that $\lgth (f^n) = n$ and $\dlgth(f) = 1$.

One would prove analogously that $g$ has dynamical degree $1$, since the coefficients $(1,1)$ and $(1,4)$ of $B^n$ satisfy:
\[ (B^n)_{1,1} = d_n, \quad (B^n)_{1,4} = d_n - d_{n-1}. \qedhere \] 
\end{proof}

\begin{corollary}\label{Cor:1213LengthSpectrum}
We have
\[\frac{1}{2}\Z_{ \geq 0} \cup \frac{1}{3}\Z_{ \geq 0}\subseteq \dlgth(\Bir(\p^2))=\{\dlgth(f)\mid f\in \Bir(\p^2)\}\]
\end{corollary}

\begin{proof}
Lemma~\ref{Lemma:1213} yields elements $f_2 ,f_3 \in \Bir(\p^2)$ such that 
$\dlgth(f_2)=\frac{1}{2}$ and $\dlgth(f_3)=\frac{1}{3}$. We then get $\dlgth({f_2}^m)=\frac{m}{2}$ and $\dlgth({f_3}^m)=\frac{m}{3}$ for each $m\ge 0$.
\end{proof}

\subsection{Length of monomial transformations}\label{SubSec:LengthMonomial}
Recall that the group $\GL_2(\Z)$ can be viewed as
the subgroup of monomial transformations of $\Bir(\p^2)$:
a matrix $\left( \begin{array}{rr} a & b \\ c & d \end{array} \right)$ corresponds to
the transformation $[x:y:1]\dasharrow [x^ay^b:x^cy^d:1]$.  In this
section, we give an algorithm to compute the length and  dynamical length in $\Bir(\p^2)$ of all monomial transformations.

In $\S$\ref{SubSec:LengthSL2p}, we first introduce the submonoids $S_R \subseteq \SL_2(\Z)_{\ge 0}$ of $\SL_2 (\Z)$ (see Lemma~\ref{Lem:DecSL2p}) and explain how to compute the lengths of their elements. In $\S\ref{SubSec:OrderedFracCont}$, we deal with the particular case of \emph{ordered} elements (see Definition~\ref{Def:ordered}) and relate the computation or their lengths with continued fractions (this relation is not needed in the sequel).  In $\S\ref{SubSec:LenghtGL2}$, we give the length of every element of $\GL_2(\Z)$ by reducing to the case of elements of $S_R$ (Lemma~\ref{lemma: reduction to a matrix with non-negative coefficient for computing the length}). The dynamical length of every element of $\GL_2(\Z)$ is then computed in $\S\ref{SubSec:DynLenghtGL2}$.

\subsubsection{The length of elements of $\SL_2(\Z)_{\ge 0}$} \label{SubSec:LengthSL2p}\label{Subsec:NonnegativeSL2}
In our first lemma we introduce some piece of notation and recall basic results:

\begin{lemma} \label{Lem:DecSL2p}
Writing $\SL_2(\Z)_{\ge 0}=\left.\left\{\left( \begin{array}{rr} a & b \\  c & d \end{array}\right) \in \SL_2(\Z)\right| 
a,b,c,d\ge 0\right\}$, we get:

\begin{enumerate}
\item\label{SL2Dec}
$\SL_2(\Z)_{\ge 0}=S_L \uplus S_R \uplus \{ \mathrm{id}\}$ $($disjoint union$)$, where 
\[\begin{array}{cccccccc}
S_L&=&\left.\left\{\left( \begin{array}{rr} a & b \\  c & d \end{array}\right) \in \SL_2(\Z)\right| \begin{array}{ll}
a\ge b\ge 0\\
 c\ge d\ge 0\end{array}\right\}&=&\SL_2(\Z)_{\ge 0}\cdot L,\vspace{0.1cm}\\ 
S_R&=&\left.\left\{\left( \begin{array}{rr} a & b \\  c & d \end{array}\right) \in \SL_2(\Z)\right| \begin{array}{ll}
0\le a\le b\\
0\le c\le d\end{array}\right\}&=&\SL_2(\Z)_{\ge 0}\cdot R.\end{array}\]
\item\label{SL2Monoid}
$\SL_2(\Z)_{\ge 0}$ is the free monoid generated by $L = \left( \begin{array}{rr} 1 & 0 \\  1 & 1 \end{array} \right)$ and $R= \left( \begin{array}{rr} 1 & 1 \\  0 & 1 \end{array} \right)$.
\item \label{LR}
$L\cdot\SL_2(\Z)_{\ge 0}\cdot R=\left.\left\{\left( \begin{array}{rr} a & b \\  c & d \end{array}\right) \in \SL_2(\Z)\right| \begin{array}{ll}
0\le a\le b\le d\\
 0\le a\le c\le d \end{array}\right\}.$
\end{enumerate}
\end{lemma}

\begin{proof}
We write \[S_L=\left.\left\{\left( \begin{array}{rr} a & b \\  c & d \end{array}\right) \in \SL_2(\Z)\right| \begin{array}{ll}
a\ge b\ge 0\\
 c\ge d\ge 0\end{array}\right\}, S_R=\left.\left\{\left( \begin{array}{rr} a & b \\  c & d \end{array}\right) \in \SL_2(\Z)\right| \begin{array}{ll}
0\le a\le b\\
0\le c\le d\end{array}\right\}\]
and obtain $S_L\cup S_R\subseteq \SL_2(\Z)_{\ge 0}$. The sets $S_L$, $S_R$ and $\{ \mathrm{id}\}$ are pairwise disjoint. To show $\SL_2(\Z)_{\ge 0}=S_L\cup S_R\cup\{ \mathrm{id}\}$, we take $M=\left(\begin{array}{rr} a & b \\  c & d \end{array}\right)\in \SL_2(\Z)_{\ge 0}\setminus \{\mathrm{id}\}$ and show that $M\in S_L\cup S_R$. As $a,b,c,d\ge 0$ and $ad-bc=1$, we have $a,d>0$. If $b=0$, then $a=d=1$ and $c>0$, so $M=L^c\in S_L$. Similarly, if $c=0$, then $M=R^b\in S_R$.
We can thus assume that $a,b,c,d>0$. The equality $1=ad-bc=(a-b)d-b(c-d)$ shows us that if $a-b$ is positive (resp.~negative), then $c-d$ is non-negative (resp. non-positive). We have therefore $(a-b)\cdot (c-d)\ge 0$, which yields $M\in S_L\cup S_R$.

We then observe that $\SL_2(\Z)_{\ge 0}\cdot L\subseteq S_L$ and $S_L\cdot L^{-1}\subseteq \SL_2(\Z)_{\ge 0}$, which yield $\SL_2(\Z)_{\ge 0}\cdot L= S_L$. We similarly obtain $\SL_2(\Z)_{\ge 0}\cdot R= S_R$. This yields \ref{SL2Dec}, which implies \ref{SL2Monoid}. Assertion~\ref{LR} follows from
\[L\cdot\SL_2(\Z)_{\ge 0}\cdot R= ( L\cdot\SL_2(\Z)_{\ge 0} ) \cap ( \SL_2(\Z)_{\ge 0}\cdot R ) =\vphantom{)}^t(\SL_2(\Z)_{\ge 0}\cdot R)\cap (\SL_2(\Z)_{\ge 0}\cdot R).\qedhere\]
\end{proof}

\begin{definition}For each sequence $(s_1,\dots,s_n)$ of positive integers with $n\ge 1$, we denote by
$M(s_1,\dots,s_n) \in S_R \subseteq \SL_2(\Z)_{\geq 0}$
the element given by 
\[M(s_1,\dots,s_n)=\left\{\begin{array}{ll}
 R^{s_n}L^{s_{n-1}} \cdots R^{s_3} L^{s_2} R^{s_1}&\text{ if $n$ is odd,}\\
 L^{s_n}R^{s_{n-1}} \cdots R^{s_3} L^{s_2} R^{s_1}&\text{ if $n$ is even,}\end{array}\right.\]
where $L = \left( \begin{array}{rr} 1 & 0 \\  1 & 1 \end{array} \right)$, $R= \left( \begin{array}{rr} 1 & 1 \\  0 & 1 \end{array} \right)\in \SL_2(\Z)$.\\
The length of $M(s_1,\dots,s_n)$ in $\Bir (\p^2)$ is denoted by $\ell(s_1,\dots,s_n)$.
\end{definition}

\begin{remark}
A matrix belongs to $S_R$ if and only if it is of the form $M(s_1, \ldots, s_n)$ where $n$ and $s_1,\ldots,s_n$ are positive integers.
\end{remark}

The next proposition gives the length of an element of $S_R$:

\begin{proposition} \label{Prop:LR}
Let $n\ge 1$. For each sequence $(s_1,\dots,s_n)$ of positive integers, $M(s_1,\dots,s_n)$ is a product of $\ell(s_1,\dots,s_n)$ elements of length $1$ that are of the form $R^s,L^s,LR^s$ or $RL^s$ for some $s\ge 1$ and thus belong to $\SL_2(\Z)_{\ge 0}$. Moreover, we have:
\[\ell (s_1, \ldots,s_n)=\left\{\begin{array}{ll} 
1 & \text{ if } n=1,\\
1 & \text{ if } n=2 \text{ and } s_2=1,\\
2 & \text{ if } n=2 \text{ and } s_2\not=1,\\
\ell(s_2-1, s_3, \ldots,s_n) + 1 & \text{ if } n\ge 3, s_2\ge 2,\\
\ell(s_3, \ldots,s_n) + 1 & \text{ if } n\ge 3, s_2=1.\end{array}
\right.\]
\end{proposition}

\begin{proof} We write $M=M(s_1,\dots,s_n)=\cdots  L^{s_2}R^{s_1}$. For each $s\ge 1$,   $R^s, LR^s$ are
\[\begin{array}{llll}
R^s\colon & [x:y:z]\mapsto [xy^s:yz^s:z^{s+1}],& LR^s\colon &[x:y:z]\dasharrow [xy^sz:xy^{s+1}:z^{s+2}],
\end{array}
\]
and thus have length equal to $1$. The same
holds for $L^s, RL^s$ (by conjugating with $\tau$, see Remark~\ref{Rem:ExchangeLR}). This gives the proof when $n=1$ or $(n,s_2) = (2,1)$.
We  thus assume  $n\ge 3$, or $n=2$ and $s_2\ge 2$. Since $\lgth(LR^{s_1})=1$, we have $\lgth(M)\le \lgth(M')+1$, where $M'=M(LR^{s_1})^{-1}$. It remains to show that equality holds to obtain the result (using Remark~\ref{Rem:ExchangeLR}). Applying Lemma~\ref{Lem:DecSL2p}, we can  write 
\[M =  \left( \begin{array}{rr} a & b \\  c & d \end{array} \right), M'=  \left( \begin{array}{rr} a' & b' \\  c' & d' \end{array} \right)=  \left( \begin{array}{rr} a+(as_1-b) & b-as_1 \\  c+(cs_1-d) & d-cs_1 \end{array} \right)\]
 with  $b\ge a\ge 0, d\ge c\ge 0$ and $a',b',c',d'\ge 0$. The degrees of $M$ and $M'$, as birational maps of $\p^2$, are respectively $D=\max\{a+b,c+d\}\ge 2$ and $D'=\max\{a'+b',c'+d'\}=\max\{a,c\}$. The element $M$ corresponds to the birational map
\[M\colon[x:y:z]\dasharrow [x^ay^bz^{D-a-b}:x^cy^dz^{D-c-d}:z^D],\]
which has degree $D$ and exactly two proper base-points, namely $p_1=[1:0:0]$ and $p_2=[0:1:0]$, having multiplicity $m_1=D-\max\{a,c\}$ and $m_2=D-\max\{b,d\}$ respectively. Hence, $p_1$ is a base-point of maximal multiplicity and every predecessor of $M$ has degree at least $D-m_1=D'$.
\end{proof}

\begin{corollary}  \label{Cor:EllIndFirst}
Let $n\ge 1$. If $s_1, \ldots,s_n$ and $s'_1$ are positive integers, we have
\[ \ell(s_1,\dots,s_n)=\ell(s'_1,s_2,\dots,s_n). \]
\end{corollary}

\begin{proof}
Directly follows from Proposition~\ref{Prop:LR}.
\end{proof}

\begin{remark} \label{Rem:ExchangeLR} The conjugation by $\tau=\left( \begin{array}{rr} 0 & 1 \\  1 & 0 \end{array} \right)\in \GL_2(\Z)$ 
exchanges $L$ and $R$, so $\ell(s_1,\dots,s_n)$ is also the length of
\[\tau M(s_1,\dots,s_n)\tau=\left\{\begin{array}{ll}
L^{s_n}R^{s_{n-1}} \cdots L^{s_3} R^{s_2} L^{s_1}&\text{ if $n$ is odd,}\\
R^{s_n}L^{s_{n-1}} \cdots L^{s_3} R^{s_2} L^{s_1}&
\text{ if $n$ is even.}
\end{array}\right.\]
Hence, Proposition~\ref{Prop:LR} allows to compute the length of any element of $\SL_2(\Z)_{\ge 0} $.
\end{remark}

\subsubsection{Ordered elements and continued fractions}\label{SubSec:OrderedFracCont}

\begin{definition}\label{Def:ordered}
We say that an element
$M=\left(\begin{array}{rr} a & b \\  c & d \end{array}\right)$  of $\SL_2(\Z)_{\ge 0}$ is \emph{ordered} if we have $0\le a\le b\le d$ and $0\le a\le c\le d$.
Equivalently, this means that $M$ belongs to $L\cdot\SL_2(\Z)_{\ge 0}\cdot R$ (see Lemma~\ref{Lem:DecSL2p}\ref{LR}).
\end{definition}

\begin{remark}
In the above definition, as $ad-bc=1$, we have $a>0$, so that all coefficients of $M$ are positive.
\end{remark}

We recall the following very classical result, whose proof is easy and well-known. We keep it as it is short, and for
self-containedness. See also
\cite[Equation (22), page~102]{Frame}
or \cite[\S 2.1]{BPSZ2014}.

\begin{proposition}\label{Prop:Ordered}
A matrix $M\in \SL_2(\Z)=\left(\begin{array}{rr} a & b \\  c & d \end{array}\right)$ is ordered if and only if it may be written
in the form
$M=L^{s_n}R^{s_{n-1}} \cdots R^{s_3} L^{s_2} R^{s_1}$ for some integers $s_1,\dots,s_n\ge 1$ with $n\ge 2$ even. In this case, the integers $s_1,\dots,s_n$ and $a,b,c,d$ are linked by the continued fractions \[ \frac{b}{a} = s_1 + \frac{1}{s_2 +   {\displaystyle \frac{1}{ \;  \ddots \, {\displaystyle   + \frac{1}{s_{n-1}}   } }    }  }\hspace{0.5cm}\text{ and }\hspace{0.5cm} \frac{d}{c} = s_1 + \frac{1}{s_2 +   {\displaystyle \frac{1}{ \;  \ddots \, {\displaystyle   + \frac{1}{s_n}   } }    }  }      .\]
\end{proposition}

\begin{proof}
The fact that a matrix $M\in \SL_2(\Z)$ is ordered if and only if it can be written $M=L^{s_n}R^{s_{n-1}} \cdots R^{s_3} L^{s_2} R^{s_1}$ with $n\ge 2$ even and $s_1,\dots,s_n\ge 1$ follows from Lemma~\ref{Lem:DecSL2p}. We then prove the equalities given by the continued fractions by induction on $n$. 

If $n=2$, then $L^{s_2} R^{s_1}=\left( \begin{array}{cc} 1 & s_1 \\  s_2 & s_1s_2+1 \end{array} \right)$, so $\frac{b}{a}=s_1$ and $\frac{d}{c}=\frac{s_1s_2+1}{s_2}=s_1+\frac{1}{s_2}$.

If $n>2$, then $\left( \begin{array}{cc} a& b \\  c & d \end{array} \right)=\left( \begin{array}{cc} a'& b' \\  c' & d' \end{array} \right) L^{s_2} R^{s_1}$, where  $\frac{b'}{a'} =s_3 +   { \frac{1}{ \;  \ddots \, {   + \frac{1}{s_{n-1}}   } }    }$ and $\frac{d'}{c'} =s_3 +   { \frac{1}{ \;  \ddots \, {   + \frac{1}{s_n}   } }    }$. We replace these in $\frac{c}{d}=\frac{c's_1+d'(s_1s_2+1)}{c'+d's_2}=s_1+\frac{1}{s_2+\frac{c'}{d'}}$ and $\frac{b}{a}=s_1+\frac{1}{s_2+\frac{b'}{a'}}$.
\end{proof}

\begin{remark}
The above result, together with Proposition~\ref{Prop:LR}, gives a way to compute the length of an ordered element $A=\left(\begin{array}{rr} a & b \\  c & d \end{array}\right)$ of  $\SL_2(\Z)$ by writing $\frac{d}{c}$ as a continued fraction with an even number of terms. Let us for example take $A=\left(\begin{array}{rr} 36 & 115 \\  41 & 131 \end{array}\right)$. Since $\frac{\displaystyle 131}{\displaystyle 41} = 3 + \frac{\displaystyle 1}{\displaystyle 5 +   {\displaystyle \frac{1}{ \;  8 }    }  }  = 3 + \frac{\displaystyle 1}{\displaystyle 5 +   {\displaystyle \frac{1}{ \;  7 \, {\displaystyle   + \frac{1}{1}   } }    }  } $, we have $A=M(3,5,7,1)$ by Proposition~\ref{Prop:Ordered}. In particular, the length of $A$ is equal to $\ell(3,5,7,1)=\ell(4,7,1)+1=\ell(6,1)+2=3$ by Proposition~\ref{Prop:LR}.
\end{remark}

\subsubsection{The length of elements of $\GL_2(\Z)$}\label{SubSec:LenghtGL2}

We now give a way to compute the length of any element of $\GL_2(\Z)$ by reducing to
the case of elements of $S_R$, i.e.~elements of the form $M(s_1,\dots,s_n)$.

\begin{lemma}   \label{lemma: reduction to a matrix with non-negative coefficient for computing the length}
For each $M\in \GL_2(\Z)$, the following hold:
\begin{enumerate}
\item\label{GL2lgth0}
$\lgth(M)=0$ $\Leftrightarrow$ $M\in \GL_2(\Z)\cap \Aut(\p^2)=\left\langle\left( \begin{array}{rr} 0 & 1 \\ 1 & 0 \end{array} \right), \left( \begin{array}{rr} 1 & -1 \\ 0 & -1 \end{array} \right)\right\rangle\simeq \Sym_3$.
\item\label{ABMSL2big}
There exist $A,B\in \GL_2(\Z)\cap \Aut(\p^2)$ such that either $AMB$ or $-AMB$ belongs to $\SL_2(\Z)_{\ge 0}$.
\item\label{Lgthbig1ordered}
If $\lgth(M)\ge 1$, there exist $A,B\in \GL_2(\Z)\cap \Aut(\p^2)$ such that either $AMB$ or $-AMB$ is equal to $M'=M(s_1,\dots,s_n)$ for some $n\ge 1$ and some positive integers $s_1,\dots,s_n\ge 1$. We then have $\lgth(M)=\lgth(M')$ $($which can be computed directly by Proposition~$\ref{Prop:LR})$.
\end{enumerate}
\end{lemma}

\begin{proof}
\ref{GL2lgth0}: We observe that the group $\GL_2(\Z)\cap \Aut(\p^2)$ corresponds to the group $\Sym_3$ of permutations of the coordinates, generated by $[x:y:z]\mapsto [y:x:z]$ and $[x:y:z]\mapsto [x:z:y]$, which correspond to $\tau=\left( \begin{array}{rr} 0 & 1 \\ 1 & 0 \end{array} \right)$ and $\nu=\left( \begin{array}{rr} 1 & -1 \\ 0 & -1 \end{array} \right)$.

\ref{ABMSL2big}: We consider the natural action of $\GL_2(\Z)$ on the circle $\p^1(\R)\simeq \mathbb{S}^1$, via $\GL_2(\Z)\to \PGL_2(\R)$.
This action induces an isomorphism between $\Sym_3= \GL_2(\Z)\cap \Aut(\p^2)$ and
the group of permutations of the set $\Delta=\{[1:0],[0:1],[1:1]\}$. These three points delimit the three closed intervals of $\p^1(\R)$ given by
\[I_1=\{[\alpha:\beta]\mid \alpha\ge \beta\ge 0\}, I_2=\{[\alpha:\beta]\mid 0\le \alpha\le \beta\}, I_3=\{[\alpha:\beta]\mid \alpha\ge0,\beta\le 0\}.\]
\[
\begin{tikzpicture}
\def \n {3} 
\def \bigradius {6mm}  
\def \smallradius {0.5mm} 
\def \margin {3} 

\foreach \s in {1,...,\n}
{
\draw ({360/\n * (\s - 1)}:\bigradius) circle (\smallradius);
\draw[-, >=latex] ({360/\n * (\s - 1)+\margin}:\bigradius) 
arc ({360/\n * (\s - 1)+\margin}:{360/\n * (\s)-\margin}:\bigradius);
}
\fill[black]   ({360/\n * (1 - 1)}:\bigradius) circle (\smallradius); 
\fill[black]   ({360/\n * (2 - 1)}:\bigradius) circle (\smallradius); 
\fill[black]   ({360/\n * (3 - 1)}:\bigradius) circle (\smallradius); 
\node at (1.2,  0) {$[ 1 : 1 ]$};
\node at (-0.8,  0.75) {$[ 0 : 1 ]$};
\node at (-0.8,  -0.75) {$[ 1 : 0 ]$};
\node at (0.4, -0.75) {$I_1$};
\node at (0.4, 0.7) {$I_2$};
\node at (-0.85,  0) {$I_3$};
\end{tikzpicture}
\]
Suppose first that $M(\Delta)$ is contained in the union of two of these three intervals. Replacing $M$ with $AM$ where $A\in \Sym_3$, we can assume that $M(\Delta)$ is contained in the interval $I_1\cup I_2=\{[\alpha:\beta]\mid \alpha,\beta\in \R_{\ge 0}\}$.

The open interval $\mathring{I}_3$ being infinite, $M^{-1}(\mathring{I}_3)$ contains elements of $\p^1(\R)\setminus \Delta$. Replacing then $M$ with $MB$, with $B \in \Sym_3$, we can assume that $M^{-1}( \mathring{I}_3) \cap \mathring{I}_3 \neq \emptyset $ or equivalently $M( \mathring{I}_3) \cap \mathring{I}_3 \neq \emptyset $.

We finish by replacing $M$ with $\pm M$ or $\pm \tau M$, to assume moreover that $\det(M)=1$ and that the first column of $M$ has non-negative coefficients. It remains to observe that $M\in \SL_2(\Z)_{\ge 0}$. Indeed, we have $M=\left( \begin{array}{rr} a & b \\ c & d \end{array} \right)\in \SL_2(\Z)$ with $a,c\ge 0$ and $bd\ge 0$ since $[b:d]\in I_1\cup I_2$. If $b,d\ge 0$, we are done. Otherwise $b,d\le 0$, which yields $M(I_3)\subseteq I_1\cup I_2$, contradicting $M( \mathring{I}_3) \cap \mathring{I}_3 \neq \emptyset $.

To finish the proof of \ref{ABMSL2big}, we suppose that $M(\Delta)$ is not contained in the union of two of the three intervals $I_1,I_2,I_3$ and derive a contradiction. This implies that the three points of $M(\Delta)$ are in the interiors of three distinct intervals. Replacing $M$ with $\pm AM$, with $A\in \Sym_3$, we can assume that $M([0:1])\in I_1, M([1:0])\in I_2, M([1:1])\in I_3$, and that the coefficients of the first column of $M$ are positive. The second column has then negative coefficients. We get $M=\left( \begin{array}{rr} a & -b \\ c & -d \end{array} \right)$ with $0<a<c$ and $b>d>0$. This yields $\det(M)=-ad+bc=a(b-d)+b(c-a)\ge 2$, a contradiction.

\ref{Lgthbig1ordered}: Using \ref{ABMSL2big}, we find $A,B\in \GL_2(\Z)\cap \Aut(\p^2)$ such that $M'=\pm AMB$ belongs to $\SL_2(\Z)_{\ge 0}$. Since $\lgth(M)\ge 1$, then $M'$ is not the identity. We can thus replace $M'$ with $\tau M'\tau$ if needed and assume that $M'\in S_R=\SL_2(\Z)_{\ge 0}\cdot R$ is an ordered matrix (follows from Lemma~\ref{Lem:DecSL2p}). This implies that $M'$ has the desired form. It remains to prove that $\lgth(M)=\lgth(M')$. If $M'=ABM$, this is because $\lgth(A)=\lgth(B)=0$. If $M'=-ABM$, we observe that $-M'$ is a product of $\lgth(-M')$ elements of length $1$ of the form $R^s,L^s,LR^s$ or $RL^s$, $s\ge 1$ (Proposition~\ref{Prop:LR}). Since
\[\begin{array}{llll}
-R^s\colon & [x:y:z]\mapsto [z^{s+1}:xy^{s-1}z:xy^s],& -LR^s\colon &[x:y:z]\dasharrow [z^{s+1}:yz^s:xy^{s}]
\end{array}\]
have length $1$, the same hold for $-L^s$, $-RL^s$ (using conjugation by $\tau$ as in Remark~\ref{Rem:ExchangeLR}). We thus get $\lgth(M')=\lgth(-M')=\lgth(M)$.
\end{proof}

\subsubsection{The dynamical length of elements of $\GL_2(\Z)$} \label{SubSec:DynLenghtGL2}

We begin to compute the dynamical length of an ordered element (see Corollary~\ref{Cor:Ordereddynlength} and Remark~\ref{remark:cyclic-permutation}), then extend to the case of a general  element of $\SL_2(\Z)$ (see Proposition~\ref{proposition: Three cases for a unimodular matrix}). This provides the dynamical length of every element of $\GL_2(\Z)$, as $\dlgth(M)=\frac{1}{2}\dlgth(M^2)$ for each $M\in \GL_2(\Z)$.

It will follow from our computation that  $\dlgth(\SL_2(\Z) ) = \Z_+$ and $\dlgth(\GL_2(\Z) ) = \frac{1}{2}\Z_+$.

Finally, at the end of the section, we prove in Corollary~\ref{Coro:Smalldynamiclength} that an element of $\GL_2(\Z)$ has dynamical length $\frac{1}{2}$ if and only if it is conjugate in $\GL_2(\Z)$ to $\pm \left(\begin{array}{rr}0 & 1 \\  1 & 1 \end{array} \right)$.

\begin{lemma}\label{Lem:LengthComposition}
Let $m,n\ge 1$ and let $(s_1,\dots,s_n)$, $(t_1,\dots,t_m)$ be two sequences of positive integers, such that $t_1\ge 2$ and $m\ge 2$. We then have 
\[\ell(s_1,\dots,s_n,t_1,\dots,t_m)=\ell(s_1,\dots,s_n)+\ell(t_1,\dots,t_m).\]
\end{lemma}
\begin{proof}
We prove the result by induction on $n$.

If $n=1$, then Proposition~\ref{Prop:LR} yields $\ell(s_1,t_1,\dots,t_m)=\ell(t_1-1,t_2,\dots,t_m)+1\stackrel{\text{Corollary~\ref{Cor:EllIndFirst}}}{=}\ell(t_1,t_2,\dots,t_m)+1=\ell(s_1)+\ell(t_1,\dots,t_m).$

If $n= 2$ and $s_2=1$, then Proposition~\ref{Prop:LR} yields $\ell (s_1,s_2,t_1,\dots,t_m)=\ell(t_1,\dots,t_m) + 1=\ell(t_1,\dots,t_m)+\ell(s_1,s_2)$. If $n=2$ and $s_2\ge 2$, then Proposition~\ref{Prop:LR} yields $\ell (s_1,s_2,t_1,\dots,t_m)=\ell(s_2-1,t_1,\dots,t_m) + 1$, which is equal to $\ell(t_1,\dots,t_m) + 2$ by induction hypothesis. This achieves the proof since $\ell(s_1,s_2)=2$ by Proposition~\ref{Prop:LR}.

If $n\ge 3$, then Proposition~\ref{Prop:LR} yields
\[\begin{array}{rcl}
\ell (s_1,\dots,s_n,t_1,\dots,t_m)&=&\left\{\begin{array}{ll} 
\ell(s_2-1, s_3, \ldots,s_n,t_1,\dots,t_m) + 1 & \text{ if } s_2\ge 2,\\
\ell(s_3, \ldots,s_n,t_1,\dots,t_m) + 1 & \text{ if }s_2=1.\end{array}
\right.\\
\ell (s_1,\dots,s_n)&=&\left\{\begin{array}{ll} 
\ell(s_2-1, s_3, \ldots,s_n) + 1 & \text{ if } s_2\ge 2,\\
\ell(s_3, \ldots,s_n)+ 1 & \text{ if }s_2=1,\end{array}
\right.\end{array}
\]
so the result follows by induction.
\end{proof}

\begin{corollary} \label{Cor:Ordereddynlength} Let $n\ge 2$ be an even integer and let $(s_1,\dots,s_n)$ be a sequence of positive integers such that either $s_1\ge 2$ or $s_1=\dots=s_n=1$. 
Then, the ordered element $M=M(s_1,\dots,s_n)\in \SL_2(\Z)$ satisfies \[\dlgth(M)=\lgth(M).\]
\end{corollary}

\begin{proof}
If $s_1=2$, then Lemma~\ref{Lem:LengthComposition} yields $\lgth(M^m)=m\cdot\lgth(M)$ for each $m\ge 1$, which yields $\dlgth(M)=\lgth(M)$.

If $s_1=\dots=s_n=1$, then $M^m=((LR)^n)^m$. It then suffices to show that $\lgth((LR)^n)=n$ for each $n\ge 1$. For $n=1$, this is directly given by Proposition~\ref{Prop:LR}. For $n\ge 1$, we also apply Proposition~\ref{Prop:LR} and get $\lgth((LR)^n)=\lgth((LR)^{n-1})+1$, which yields the result by induction.
\end{proof}

\begin{remark}  \label{remark:cyclic-permutation}
Note that $M(s_1, \ldots,s_n)$ is conjugate in $\SL_2(\Z)$ to $M(s_2, \ldots,s_n, s_1)$. Hence, each element $M(s_1,\dots,s_n)$ admits a conjugate which satisfies the hypotheses of Corollary~\ref{Cor:Ordereddynlength}.
\end{remark}

\begin{proposition} \label{proposition: Three cases for a unimodular matrix}
Let $M\in \SL_2(\Z)$.
\begin{enumerate}
\item\label{Trace01}
If $\trace(M)\in \{0,\pm 1\}$ then $M$ has order $m\in \{3,4,6\}$ so $\dlgth(M)=0$.
\item \label{Trace2}
If $\trace(M)\in \{\pm 2\}$ then $M$ is conjugate to $\pm\left( \begin{array}{rr} 1 & a \\  0 & 1 \end{array} \right)$ for some $a\in \Z$, so $\dlgth(M)=0$.
\item  \label{Trace3}
If $\lvert\trace(M)\rvert\ge 3$ then $\pm M$ is conjugate to an ordered element $M'$, so $\dlgth(M)=\dlgth(M')$ is a positive integer.
\end{enumerate}
\end{proposition}

\begin{proof}
Writing $\lambda=\trace(M)$, the characteristic polynomial of $M$ is equal to $\chi_M=X^2-\lambda X+1$. If $\lambda\in \{0,\pm 1\}$, we obtain then orders $3,4,6$, yielding \ref{Trace01}. If $\lambda=\pm 2$, then $\chi_M=(X\pm 1)^2$, so there is an eigenvector of eigenvalue $\pm 1$, which can be choosen in $\Z^2$ with coprime coefficients. This yields \ref{Trace2}.

In case \ref{Trace3} we can replace $M$ with $-M$ if needed and assume that $\lambda\ge 3$. We will then show that $M$ is conjugate
to an ordered matrix. The fact that $\lambda=\trace(M)\ge 3$ implies that $M$ has distinct  positive real eigenvalues
$\mu, \mu^{-1}$ with $\max \{ \mu, \mu^{-1} \} = \frac{\lambda+\sqrt{\lambda^2-4}}{2}>2$.
Write $M= \left(\begin{array}{cc} a & b \\  c & d \end{array} \right)$.
The two eigenspaces are spanned by  two vectors $(1, \xi_1)$ and $(1, \xi_2)$, where $\xi_1, \xi_2$ are nonzero reals (note that $(1,0)$ and $(0,1)$ are not eigenvectors of $M$ since $bc \neq 0$).

$(a)$ If $\xi_1\xi_2<0$, then we may assume without loss of generality that $\xi_1 >0$ and $\xi_2 <0$ (by exchanging the names of $\xi_1$ and $\xi_2$). Up to replacing $M$ with $M^{-1}$, we may furthermore assume that we have $\mu > 1$, where we have used the two following facts:
\begin{enumerate}[(i)]
\item We have $\trace(M^{-1}) = \trace(M)$;
\item The inverse of an ordered matrix is conjugate to an ordered matrix since the matrix $P= \left(\begin{array}{rr} 0 & 1 \\  -1 & 0 \end{array} \right)$ satisfies $PR^{-1}P^{-1} = L$ and $PL^{-1}P^{-1} = R$.
\end{enumerate}
Then, we have
$\left( \begin{array}{l} a +b\xi_1 = \mu \\ c+ d \xi_1 = \mu \xi_1 \\ a + b \xi_2 = \mu^{-1} \\ c +d \xi_2 = \mu^{-1} \xi_2 \end{array} \right.$ which yields $\left( \begin{array}{l} b (\xi_1 - \xi_2) = \mu - \mu^{-1} \\ c (\xi_1^{-1} - \xi_2^{-1}) = \mu - \mu^{-1} \\ a =\mu^{-1} - b \xi_2 \\  d = \mu^{-1} - \xi_2^{-1} c  \end{array} \right.$ which in turn proves that $b,c>0$, and then $a,d>0$. Therefore, $M$ belongs to the monoid generated by $L$ and $R$ (Lemma~\ref{Lem:DecSL2p}). As $\trace(M)\not=2$, $M$ is not conjugate
to a matrix of the form $L^s$ or $R^s$ for some $s\in \Z$
and is thus conjugate  to an element which starts with $L$ and ends with $R$, hence to an ordered element (Lemma~\ref{Lem:DecSL2p}). The result then follows from Corollary~\ref{Cor:Ordereddynlength}.

$(b)$ If there exists $s\in \Z$ such that $(\xi_1+s)(\xi_2+s)<0$, we conjugate $M$ with $L^s$. This replaces $\xi_i$ with $\xi_i+s$ and reduces to case $(a)$.

$(c)$ If $(b)$ is not possible, there exists $s\in \Z$ such that $0< \xi_i+s<1$ for $i=1,2$. Replacing $s$ with $s-1$ if needed, we can rather assume that $\lvert \xi_i+s\rvert<1$ for both $i$ and that $\lvert \xi_i +s \rvert <\frac{1}{2}$ for at least one $i$, which we will assume to be $1$ (by exchanging the names of $\xi_1$ and $\xi_2$ if needed).
Therefore, by conjugating $M$ with $L^s$ we may assume that $\lvert \xi_1\rvert<\frac{1}{2}$ and $\lvert \xi_2 \rvert<1$.
We then conjugate $M$ with $\left(\begin{array}{rr}0 & 1 \\  -1 & 0 \end{array} \right)$. This replaces $\xi_i$ with  $\xi_i'=-\frac{1}{\xi_i}$ for $i=1,2$
and we have
$\lvert\xi_1'-\xi_2'\rvert=\lvert\xi_1-\xi_2\rvert\cdot \lvert\frac{1}{\xi_1}\rvert\cdot \lvert\frac{1}{\xi_2}\rvert>2\cdot \lvert\xi_1-\xi_2\rvert$.
 After finitely many such steps we obtain $\lvert\xi_1-\xi_2\rvert>1$, which then gives case $(b)$.
\end{proof}

\begin{example}
The ordered matrices $M(5,1) = \left( \begin{array}{rr} 1 & 5 \\  1& 6 \end{array} \right)$ and $M(1,1,1,1) = \left( \begin{array}{rr} 2 & 3 \\  3 & 5 \end{array} \right)$ have lengths $\lgth ( M(5,1) ) =1$ and $\lgth(M(1,1,1,1))=2$  (Proposition~\ref{Prop:LR}), and then dynamical lengths
\[ \dlgth(M(5,1)) = \lgth ( M(5,1) ) =1 \text{ and }  \dlgth(M(1,1,1,1))=\lgth(M(1,1,1,1)) =2 \]
(Corollary~\ref{Cor:Ordereddynlength}). In particular, the matrices $M(5,1),M(1,1,1,1)$ have different dynamical lengths, even if they have the same trace and determinant, and thus the same eigenvalues and 
dynamical degrees.
\end{example}

\begin{corollary} \label{Coro:Smalldynamiclength}
For each $M \in \GL_2(\Z)$, the following conditions are equivalent:
\begin{enumerate}
\item  \label{Dynamical lenght is one half} We have  $\dlgth(M) = \frac{1}{2}$.
\item  \label{M is conjugate to plus or minus A} The matrix $M$ is conjugate in $\GL_2(\Z)$ to $\pm \left(\begin{array}{rr}0 & 1 \\  1 & 1 \end{array} \right)$.
\item \label{Determinant is minus one and Trace is plus or minus one} We have $\det (M) = -1$ and  $\trace(M)\in \{\pm 1 \}$.
\end{enumerate}
\end{corollary}

\begin{proof}
\ref{M is conjugate to plus or minus A} $\Rightarrow$  \ref{Determinant is minus one and Trace is plus or minus one} is obvious.

\ref{Determinant is minus one and Trace is plus or minus one} $\Rightarrow$ \ref{Dynamical lenght is one half}. If $M$ satisfies \ref{Determinant is minus one and Trace is plus or minus one}, we have $\trace(M^2)= (\trace(M)) ^2 - 2 \det (M) = 3$,
so that $M^2$ is conjugate to an ordered element of trace $3$ by Proposition~\ref{proposition: Three cases for a unimodular matrix}\ref{Trace3}. The unique ordered element of trace $3$ being the matrix  $\left(\begin{array}{rr} 1 & 1 \\  1 & 2 \end{array} \right) = LR = M(1,1)$ of dynamical length $1$, this proves \ref{Dynamical lenght is one half}.

\ref{Dynamical lenght is one half} $\Rightarrow$ \ref{M is conjugate to plus or minus A}. If $M$ satisfies \ref{Dynamical lenght is one half}, we necessarily have $\det (M) = -1$ and $\dlgth(M^2) = 1$ (Proposition~\ref{proposition: Three cases for a unimodular matrix}). By Proposition~\ref{proposition: Three cases for a unimodular matrix} and up to conjugation of $M$ into $\GL_2(\Z)$, there exists an element $\varepsilon \in \{ \pm 1 \}$ such that $M' := \varepsilon M^2$ is an ordered matrix satisfying $\dlgth(M') = \lgth (M') = 1$. This yields the existence of an integer $s \geq 1$ such that $M' = LR^s = \left(\begin{array}{cc} 1 & s \\  1 & s+1 \end{array} \right)$ . Writing $M =  \left(\begin{array}{cc} a & b \\  c & d \end{array} \right)$, we obtain
\[  M^2 = \left(\begin{array}{cc} a^2 + bc & b(a+d) \\  c(a+d) & d^2 +bc \end{array} \right) = \varepsilon  \left(\begin{array}{cc} 1 & s \\  1 & s+1 \end{array} \right) . \]
The equality $c(a+d) = \varepsilon$ gives us $\trace(M) = a+d = \pm 1$, so that (as above) $\trace(M^2)= ( \trace(M) ) ^2 - 2 \det (M) = 3$. This implies $\varepsilon (s+2) =3$, so that $\varepsilon = s = 1$. This proves that $b= c = a+d \in \{ \pm 1 \}$ and since $a^2 + bc =1$, we finally obtain $a= 0$ and $b=c=d \in \{ \pm 1 \}$, proving that $M = \pm \left(\begin{array}{rr}0 & 1 \\  1 & 1 \end{array} \right)$.
\end{proof}

\subsection{Dynamical length of regularisable elements and the proof of Theorem~\ref{Thm:DistAutP2}}  \label{Subsection:Dynamical-length-of-regularisable-elements}

Recall the two following definitions:

\begin{definition}
An element $f\in \Bir(\p^2)$ is said to be \emph{regularisable} if there exists a birational map $\eta \colon X \dasharrow \p^2$,  where $X$ is a smooth projective surface, such that $\eta^{-1}\circ f \circ \eta\in \Aut(X)$.
By \cite[Theorem~B]{BlaDes}, this is equivalent to $\mu(f) = 0$, where $\mu$ denotes the dynamical number of base-points, as explained in the introduction. (The statement of \cite[Theorem~B]{BlaDes} is made over $\C$ but its proof works over any algebraically closed field).
\end{definition}

\begin{definition}
An element $f\in \Bir(\p^2)$ is said to be \emph{loxodromic} if  $\log(\lambda(f)) >0$ (where $\lambda(f)=\lim\limits_{n\to \infty} (\deg(f^n))^{1/n}$ is the dynamical degree of $f$, as explained in the introduction).
\end{definition}

It follows from \cite{GizHalphen,DillerFavre,BlaDes} that
a Cremona transformation $f \in \Bir (\p^2)$ belongs to exactly one of the five following categories:

\begin{enumerate}
\item
{\bf Algebraic elements} ;
\item
{\bf Jonqui\`eres twists:} $f\in \Bir(\p^2)$ is \emph{a Jonqui\`ere twist} if
the sequence $n \mapsto \deg(f^n)$ grows linearly, 
i.e.~if the sequence $n \mapsto \frac{\deg (f^n)}{n}$ admits a nonzero limit when $n$ goes to infinity. Equivalently, $f$ preserves a rational fibration $\p^2\dasharrow \p^1$ and is not algebraic;
\item
{\bf Halphen twists:} $f\in \Bir(\p^2)$ is \emph{a Halphen twist} if
the sequence $n \mapsto \deg(f^n)$
grows quadratically, i.e.~if the sequence $n \mapsto \frac{\deg (f^n)}{n^2}$ admits a nonzero limit when $n$ goes to infinity. Equivalently, $f$ preserves an elliptic fibration $\p^2\dasharrow \p^1$ and is not algebraic;

\item
 {\bf Regularisable loxodromic elements:} In this case, $\lambda(f)$ is a Salem number or a reciprocical quadratic integer (see \cite{DillerFavre,BlaCan});

\item
{\bf Non-regularisable loxodromic elements:} In this case, $\lambda(f)$ is a Pisot number by \cite{BlaCan}.
\end{enumerate}

If $f$ is a Halphen twist or a regularisable loxodromic element, we will prove that the dynamical length $\dlgth(f) $ is positive in Corollary~\ref{Cor:Halphentwist} and Proposition~\ref{Prop:RegLox}. In Lemma~\ref{Lemma:AutA2}, an example of non-regularisable loxodromic element $f\in \Bir(\p^2)$ whose dynamical length $\dlgth(f) $ is positive was given. These results are summarised in Figure~\ref{FigureSubadditive} and achieve the proof of Theorem~\ref{Thm:DistAutP2}.

The following result follows from the proof of \cite[Lemma 5.10]{BlaCan}.

\begin{lemma}\label{Lemm:9ptssqrt}
Let $f_1,f_2\in \Weyl$ be two elements such that $\Base(f_1)\cup \Base(f_2)$ contains at most $9$ points. Then, 
\[\sqrt{\deg(f_1\circ f_2^{-1})}\le\sqrt{\deg(f_1)}+\sqrt{\deg(f_2)}\qedhere\]
\end{lemma}

\begin{lemma}\label{Lemm:BasePtsbounded}
Let $\psi\in \Bir(\p^2)\setminus \Aut(\p^2)$ be a birational map,
and let $r$ be its number of base-points. Recall that  we have $r \geq 3$ by Lemma~$\ref{Lemma:Hudsonalgorithm}$. Then, the following hold:
\begin{enumerate}
\item\label{Ineq1bspt}
$\lgth(\psi)\ge \frac{\log(\deg(\psi))}{\log(\frac{r+1}{2})}$.
\item\label{Ineq2bspt}
If $r\le 9$, then  $\lgth(\psi)\ge \sqrt{\frac{\deg(\psi)}{5}}$.
\end{enumerate}
\end{lemma}
\begin{proof}
We prove the result by induction on $\lgth(\psi)$. When $\lgth(\psi)=1$, then $\psi$ is a Jonqui\`eres transformation of degree $d=\deg(\psi)>1$, which implies that $r=2\deg(\psi)-1$. Hence, \ref{Ineq1bspt} is an equality and \ref{Ineq2bspt} holds, as $r\le 9$ yields $\deg(\psi)\le 5$.

Suppose now that $\lgth(\psi)>1$, and let $\varphi\in \Bir(\p^2)$ be a Jonqui\`eres element such that $\psi'=\psi\circ \varphi$ is a predecessor of $\psi$ (which implies in particular that $\lgth(\psi')=\lgth(\psi)-1$). We then have $\Base(\varphi^{-1})\subseteq\Base(\psi)$ (Lemma~\ref{Lem:Pred}), which yields $\Base(\psi'^{-1})\subseteq \Base(\psi^{-1})$ (Corollary~\ref{corollary: inclusion of the base locus of a composition in a special case}). 
This proves that $\varphi$ and $\psi'$ have at most $r$ base-points. In particular, we have $\deg(\varphi)\le \frac{r+1}{2}$.

To prove \ref{Ineq1bspt} we start with
$\deg(\psi)\le \deg(\psi')\cdot \deg(\varphi)\le \deg(\psi')\cdot \frac{r+1}{2}$, which yields $\log(\deg(\psi))\le \log(\deg(\psi'))+\log(\frac{r+1}{2})$. Applying induction to $\psi'$, we then get $\frac{\log(\deg(\psi))}{\log(\frac{r+1}{2})}\le \frac{\log(\deg(\psi'))}{\log(\frac{r+1}{2})}+1\le \lgth(\psi')+1=\lgth(\psi)$, as desired.

To prove \ref{Ineq2bspt}, we denote by $\pi\colon Z\to \p^2$ the blow-up of the base-points of $\varphi^{-1}$. As $\Base(\varphi^{-1})\subseteq \Base(\psi)$, there is a birational morphism $\epsilon\colon X\to Z$, such that $\pi \circ \epsilon$ is the blow-up of the base-points of $\psi$. We then get a commutative diagram
\[\xymatrix@R=0.4cm@C=2cm{
& X\ar[d]^{\epsilon}\ar[rdd]^{\eta'} \\
 & Z\ar[d]^{\pi}\ar[ld]_{\eta} \\
\p^2\ar@/_0.7pc/@{-->}[rr]_{\psi'}\ar@{-->}[r]^{\varphi}&\p^2\ar@{-->}[r]^{\psi}&\p^2,
}\]
where $\eta,\eta'$ are the blow-ups of the base-points of $\varphi$ and $\psi^{-1}$ respectively. As $\eta'$ blows-up $r$ points, the same holds for $\eta \circ \epsilon$. Hence, we find that
$\Base(\psi')\cup \Base(\varphi)\subseteq \Base((\eta \circ \epsilon)^{-1})$
contains at most $r\le 9$ points. We can then apply Lemma~\ref{Lemm:9ptssqrt}, which yields $\sqrt{\deg(\psi)} = \sqrt{ \deg (\psi' \circ \varphi^{-1} ) }     \le \sqrt{\deg(\psi')} +\sqrt{\deg(\varphi)} \le  \sqrt{\deg(\psi')} +\sqrt{5}$.  Applying induction to $\psi'$, we find 
\[\sqrt{\frac{\deg(\psi)}{5}} \le \sqrt{\frac{\deg(\psi')}{5}} +1\le \lgth(\psi')+1=\lgth(\psi).\qedhere\]
\end{proof}

\begin{proposition}\label{Prop:9pts}
Let $\pi\colon X\to \p^2$ be a birational morphism which is
the blow-up of at most $9$ points
and let $\varphi \in  \pi \Aut(X)\pi^{-1} \subseteq \Bir(\p^2)$ be a Cremona transformation. Then, the sequence $n \mapsto \frac{\deg(\varphi^n)}{n^2}$ admits a limit $L \in \mathbb{R}$ when $n$ goes to infinity and we have
\[\begin{array}{c}\dlgth(\varphi) \ge \sqrt{\frac{L}{5}}.\end{array}\] 
\end{proposition}

\begin{proof}
Denote by $\Delta\subseteq \B(\p^2)$ the set of points
blown-up by $\pi$. For each $n\in\Z$, we have $\Base(\varphi^n)\subseteq\Delta$. In particular we have $\sqrt{\deg(\varphi^{m+n})}=\sqrt{\deg(\varphi^{m}\circ (\varphi^{-n})^{-1})}\le \sqrt{\deg(\varphi^m)}+\sqrt{\deg(\varphi^{-n})}=\sqrt{\deg(\varphi^m)}+\sqrt{\deg(\varphi^n)}$ for all $m,n\ge 1$ (Lemma~\ref{Lemm:9ptssqrt}). This means that the sequence $n\mapsto \sqrt{\deg(\varphi^n)}$ is subadditive, so $\lim\limits_{n\to \infty} \frac{\sqrt{\deg(\varphi^n)}}{n}$ exists, which is equivalent to saying that $\lim\limits_{n\to \infty} \frac{\deg(\varphi^n)}{n^2}$ exists.

For each $n\ge 1$, the number of base-points of $\varphi^n$ is at most $9$. This yields $\lgth(\varphi^n)\ge \sqrt{\frac{\deg(\varphi^n)}{5}}$ (Lemma~\ref{Lemm:BasePtsbounded}\ref{Ineq2bspt}), whence
$\frac{\lgth(\varphi^n)}{n} \ge \sqrt{\frac{\deg(\varphi^n)}{5n^2}}$ and the result follows by taking the limit when $n$ goes to infinity.
\end{proof}

\begin{corollary}\label{Cor:Halphentwist}
If $\varphi\in \Bir(\p^2)$ is a birational transformation such that $\left(\deg(\varphi^n)\right)_{n\ge 1}$ grows quadratically $($i.e.~$\varphi$ is a Halphen twist$)$,
then $\dlgth(\varphi)\ge \sqrt{\frac{1}{5}\cdot \lim\limits_{n\to \infty} \frac{\deg(\varphi ^n)}{n^2}}>0$.
\end{corollary}

\begin{proof}
A birational transformation of $\p^2$ has a quadratic growth if and only if it is conjugate to an automorphism of a Halphen surface, obtained by blowing-up $9$ points of $\p^2$ \cite{GizHalphen}. The result then follows from Proposition~\ref{Prop:9pts}.
\end{proof}

\begin{remark}
Using an analogue method as in \cite[Proposition 5.1]{BlaDes} we are able to give a uniform lower bound $C >0$ such that $\dlgth(\varphi)\ge C$ for all Halphen twists.
However, this bound is far from being reached from the known examples.
\end{remark}

\begin{proposition}  \label{Prop:RegLox}
Let $\varphi\in \Bir(\p^2)$ be a loxodromic birational map which is regularisable,
i.e.~such that there exists a birational map $\kappa \colon \p^2\dasharrow X$ that conjugates $\varphi$ to an automorphism $g = \kappa \circ  \varphi \circ  \kappa^{-1} \in \Aut(X)$ where $X$ is a smooth projective surface.

Then, each $X$ as above is isomorphic to the blow-up of  finitely many points $p_1, \ldots,p_r \in \B( \p^2)$ with $r\ge 10$, and  the dynamical length of $\varphi$ satisfies
$\dlgth(\varphi) \ge \frac{\log(\lambda(\varphi))}{\log(\frac{r+1}{2})} > 0.$
\end{proposition}

\begin{proof}
We first show that there
exists a birational morphism $\eta\colon X\to \p^2$. Suppose the converse, for contradiction. Then, \cite[Corollary 1.2]{Har1987} implies that the action of $\Aut(X)$ on $\Pic(X)$ is finite (i.e.~factorises through the action of a finite group). 
This is impossible: the dynamical degree of $g$, equal to the one of $\varphi$ as both are conjugate, is the spectral radius of the action of $g$ on $\Pic(X)\otimes_\Z \C$ (see the introductions of \cite{DillerFavre} and \cite{BlaCan}
for details on these two facts). This dynamical degree is therefore equal to~$1$, contradicting the fact that $\varphi$ is loxodromic.

We then obtain the existence of  a birational morphism $\eta \colon X \to \p^2 $ which is the blow-up of finitely many points $p_1,\dots,p_{r}\in \B(\p^2)$. Observe moreover that $r\ge10$ because $\varphi$ is loxodromic. One way to see this classical fact is to use Proposition~\ref{Prop:9pts} which implies that $\{\deg(\tilde\varphi^n)\}_{n\ge 1}$ grows at most quadratically if $r\le 9$, where $\tilde\varphi=\eta \circ g\circ \eta^{-1}$. This
is impossible
since $\{\deg(\varphi^n)\}_{n\ge 1}$ grows exponentially and because $\varphi,\tilde\varphi \in \Bir(\p^2)$ are conjugate (by $\eta\circ \kappa\in \Bir(\p^2)$).

 We then replace $\varphi$ with its conjugate $\eta \circ g\circ \eta^{-1}\in \Bir(\p^2)$. After this, we get $\varphi^n  = \eta \circ g^n \circ \eta^{-1}$ for each $n\ge 1$, so $\Base(\varphi^n)\subseteq \{p_1,\dots,p_{r}\}$. By Lemma~\ref{Lemm:BasePtsbounded}\ref{Ineq1bspt} we have
$\lgth(\varphi^n)\ge \frac{\log(\deg(\varphi^n))}{\log(\frac{r+1}{2})}$.
From $\lim\limits_{n\to \infty} (\deg(\varphi^n))^{1/n}=\lambda(\varphi)$, we deduce
$\lim\limits_{n\to \infty} \frac{\log(\deg(\varphi^n))}{n}=\log(\lambda(\varphi))$, which then yields 
\[\dlgth(\varphi)=\lim\limits_{n \to \infty} \frac{\lgth(\varphi^n)}{n}\ge \lim\limits_{n \to \infty} \frac{\log(\deg(\varphi^n))}{n\log(\frac{r+1}{2})}=\frac{\log(\lambda(\varphi))}{\log(\frac{r+1}{2})}>0\]
as desired.
\end{proof}

\begin{lemma}\label{Lemm:AutP2distorted}
Every element of $\Aut(\p^2)$ is distorted in $\Bir(\p^2)$.
\end{lemma}

\begin{proof}
Let us first start with
a diagonalisable element of $\Aut(\p^2)$.
Up to conjugation, we may assume that this element is
locally given in the diagonal form
$g\colon (x,y)\mapsto (\alpha x,\beta y)$ for some $\alpha,\beta\in \k^*$.
Consider the monomial transformation
$\tau\colon (x,y)\dasharrow (y,xy)$ (which is a monomial transformation
of minimal positive dynamical length by Corollary~\ref{Coro:Smalldynamiclength}). For each $n\ge 1$ we obtain $\tau^n\colon (x,y)\dasharrow (x^{a_{n-1}}y^{a_n},x^{a_n}y^{a_{n+1}})$ where $(a_0,a_1,a_2,\cdots)=(0,1,1,2,3,5,8,\cdots)$ is the Fibonacci sequence.
We then set
$\rho \colon (x,y)\mapsto (y,x)$ and $\kappa \colon (x,y) \dasharrow (x^{-1},y)$, and get
$\varphi_1=\tau^n\circ \rho\colon (x,y)\dasharrow (x^{a_n}y^{a_{n-1}},x^{a_{n+1}}y^{a_n})$ and $\varphi_2=\kappa\circ \varphi_1\circ \kappa\colon (x,y)\dasharrow (x^{a_n}y^{-a_{n-1}},x^{-a_{n+1}}y^{a_n})$
so $\varphi_1\circ g\circ \varphi_1^{-1} \circ \varphi_2\circ g \circ \varphi_2^{-1}\colon
(x,y)\mapsto(\alpha^{2a_n}x,\beta^{2a_n}y)= g^{2a_n}$.
Hence, writing $F=\{g,\tau,\rho,\kappa\}\subseteq \Bir(\p^2)$,
we get $\lvert g^{2a_n} \rvert_F \le 2(\lvert \varphi_1\rvert_F+\lvert \varphi_2\rvert_F+\lvert g\rvert_F)\le 4n+10$. This implies that $g$ is distorted, since $\lim\limits_{n\to \infty} \frac{n}{a_n}=0$.

We now consider the case of a non-diagonalisable element of $\Aut(\p^2)$. This element is
either conjugate to $g\colon [x:y:z]\mapsto [\alpha x:y+z: z]$ for some $\alpha\in \k^*$ or to $h\colon [x:y:z]\mapsto [x+y:y+z:z]$. If $\car(\k)=p>0$, then $g^p$ and $h^{p^2}$ are diagonal and thus distorted, so $g$ and $h$ are distorted.
Hence, we may
assume that $\car(\k)=0$.
Write
$g,h$ locally as $g\colon (x,y)\mapsto (\alpha x,y+1)$ and $h\colon (x,y)\mapsto (x+y,y+1)$, and observe that we only need to show that $g$ is distorted, as $h$ is conjugate to $g$ (with $\alpha=1$) by $(x,y)\mapsto (x-\frac{y(y-1)}{2},y)$. As $g$ is the composition of the two commuting automorphisms  $q\colon (x,y)\mapsto (x,y+1)$ and $r\colon (x,y)\mapsto (\alpha x,y)$, it suffices to prove that both $q$ and $r$ are distorted. As $r$ is diagonal, we only need to show that $q$ is distorted. We set $\varphi\colon (x,y)\mapsto (x,2y)$ and $F=\{q,\varphi \}$. For each $n\ge 1$ we have $\varphi^n\circ q\circ \varphi^{-n}=q^{2^n}$. Hence $\lvert q^{2^n}\lvert_F\le 2n+1$, which implies that $q$ is distorted.
 \end{proof}
 
We finish this section by recalling the following result, stated in \cite{BlaDes} for $\k=\C$, but whose proof in fact works for each algebraically closed field, as we explain now:

\begin{proposition}\label{Prop:BlaDes}
An element $\varphi\in \Bir(\p^2)$ is algebraic if and only if it is of finite order or conjugate to an element of $\Aut(\p^2)$.
\end{proposition}

\begin{proof}
By definition, every element of finite order or conjugate to an element of $\Aut(\p^2)$ is algebraic. It then remains to show that an algebraic element $\varphi\in \Aut(\p^2)$ of infinite order is conjugate to an element of $\Aut(\p^2)$. As $\varphi$ is algebraic, we can conjugate $\varphi$ to an element $g\in \Aut(S)$, where $S$ is a smooth projective rational surface, and such that the action on $\Pic(S)$ is finite. There exists then a birational morphism $S\to X$, where $X=\p^2$ or $X$ is a Hirzebruch surface  $\F_n$ for some $n \geq 0$ (see \cite[Proposition 2.1]{BlaDes}, which is stated in the case where $\k=\C$ but whose proof works over any algebraically closed field ; it is proven there that we may assume $n \neq 1$, but we will not need it). If $X=\p^2$, we are done. The reduction to this case is then done in \cite[Proposition 2.1]{BlaDes} for $\k=\C$; we follow the proof checking what is dependent of the field.  If $X = \F_0 =\p^1\times \p^1$, we blow-up a fixed point (which exists as every automorphism of $\p^1 \times \p^1$ is of the form $(u,v) \mapsto ( \tau_1 (u), \tau_2 (v) )$ or $(u,v) \mapsto ( \tau_2 (v), \tau_1 (u) )$ where $\tau_1,\tau_2$ are automorphisms of $\p^1$ and as each automorphism of $\p^1$ fixes at least a point of $\p^1$) and contract the strict transforms of the two horizontal and vertical lines of self-intersection $0$ through the point, to get a birational map $X \dasharrow \p^2$. As the point is fixed and the union of the two curves contracted is invariant, we obtain an element of $\Aut(\p^2)$.

It remains to consider the case of the Hirzebruch surface $\F_n$ with $n \geq 1$. Recall that $\F_n$ is the quotient of $(\A^2\setminus \{0\})^2$ by the action of $(\mathbb{G}_m)^2$ given by
\[\begin{array}{ccc}
(\mathbb{G}_m)^2 \times (\A^2\setminus \{0\})^2 & \to & (\A^2\setminus \{0\})^2\\
((\mu,\rho), (y_0,y_1,z_0,z_1))&\mapsto& (\mu \rho
^{-n} y_0,\mu y_1,\rho z_0,\rho z_1).\end{array}\]
The class of $(y_0,y_1,z_0,z_1)$ is denoted by $[y_0:y_1;z_0:z_1]$ and the natural $\p^1$-fibration $\F_n \to \p^1$, $[y_0:y_1;z_0:z_1] \mapsto [z_0:z_1]$ by $\pi$. It is well-known that each automorphism of $\F_n$ exchanges the fibres of $\pi$ and is of the form 
\[[y_0:y_1;z_0:z_1]\mapsto [y_0:y_1+y_0P(z_0,z_1);az_0+bz_1:cz_0+cz_1]\]
for some $\left( \begin{array}{rr} a & b \\ c & d \end{array} \right)\in \GL_2(\k)$ and some polynomial $P\in \k[z_0,z_1]$, homogeneous of degree~$n$. If one point of $\F_n$ that does no lie on the exceptional section $s$ given by $y_0=0$ is fixed, we can perform the elementary link $\F_n\dasharrow \F_{n-1}$
(blowing up the point and contracting the strict transform of the fibre through that point) to
decrease the integer $n$. We can thus assume that each fixed point of $g$ is on the exceptional section. Conjugating by an element of $\GL_2(\k)$, we obtain two possibilities, namely
\[\begin{array}{rcl}
\ [y_0:y_1;z_0:z_1]&\mapsto& [y_0:y_1+y_0P(z_0,z_1);\lambda z_0:\mu z_1]   \text{ or}\\
\ [y_0:y_1;z_0:z_1]&\mapsto& [y_0:y_1+y_0P(z_0,z_1);\lambda z_0:\lambda z_1+z_0]\end{array}\]
for some $\lambda,\mu\in \k^*$. 

In the first case, the actions on the two fibres $z_0=0$ and $z_1=0$ having no fixed points outside of $s$, we have $\lambda^n=\mu^n=1$ and $P(0,1)P(1,0)\not=0$.

In the second case, the action on the fibre $z_0=0$ having no fixed points outside of $s$ implies that $\lambda^n=1$ and $P(1,0)\not=0$. 

In both cases, we look at the action on the image of the open embedding $\A^2\to \F_n$, $(x,y)\mapsto [1:x;1:y]$. This action is given respectively by 
\[ (x,y)\mapsto \left(x+P(1,y),\frac{\mu}{\lambda} y\right) \text{ or }
(x,y)\mapsto \left(x+P(1,y),y+\frac{1}{\lambda}\right).\]

In the first case, we write $p(y)=P(1,y)$ and $\alpha=\frac{\mu}{\lambda} \in \k^*$, 
where $\alpha$ is a primitive $k$-th root of unity for some integer $k\ge 1$ that divides $n$. We can conjugate with $(x,y)\mapsto (x+\gamma y^d,y)$ and replace $p(y)$ with $p(y)+\gamma ( \alpha ^d-1)y^d$, so we may assume that $p\in \k[y^k]$. We then conjugate with $(x,y)\mapsto (\frac{x}{p(y)},y)$ to obtain $(x,y)\mapsto (x+1,\alpha y)$ which actually induces an automorphism of $\p^2$.

In the second case, we necessarily have $\car(\k)=0$, as otherwise $g$ would be of finite order.
We write  $p(y)=P(1,y)$ and $\beta = \frac{1}{\lambda} \in \k^*$. It is enough to prove that the polynomial automorphism of $\A^2$ given by $(x,y) \mapsto (x+ p(y), y + \beta)$ is conjugate to the affine polynomial automorphism $a \colon (x,y) \mapsto (x, y + \beta)$. Conjugating $a$ with the polynomial automorphism $(x,y) \mapsto ( x + q(y), y)$, where $q$ is a polynomial, we get the polynomial automorphism $(x,y) \mapsto (x+ q(y+ \alpha) - q(y), y + \beta)$. Since we are in characteristic zero, there exists a polynomial $q$ such that $q(y+ \alpha) - q(y) = p(y)$.
\end{proof}

\section{Lower semicontinuity of the length: the proof of Theorem~\ref{Continuitylength}}
\label{section:lower-semicontinuity-of-the-length}

Throughout this section, $\Pi \colon \p^2 \dasharrow \p^1$ will denote the standard  linear projection \[\begin{array}{ccc} \p^2 & \dasharrow & \p^1 \\
\ [x:y:z]& \mapsto &  [x:y].\end{array}\]

\subsection{Variables}
The proof of Theorem~\ref{Continuitylength}  uses the notion of variables, that we now define.
\begin{definition}\label{Defi:variable}
A rational map $v \colon \p^2 \dasharrow \p^1$ is called a \emph{variable}, if there exists a birational map $f\in \Bir(\p^2)$ such that 
$\Pi\circ f=v$.
\[\xymatrix{
\p^2 \ar@{-->}[r]^{f}  \ar@{-->}[dr]_{v} & \p^2 \ar@{-->}[d]^{\Pi} \\
 & \p^1 }\]
Writing a variable $v \colon \p^2 \dasharrow \p^1$ as $[x:y:z]\dasharrow [v_0(x,y,z):v_1(x,y,z)]$ where $v_0,v_1\in \k[x,y,z]$ are homogeneous of the same degree, without common factor, we define its degree as the common degree of $v_0$ and $v_1$
(which is also the degree of a general fibre of~$v$).
\end{definition}

\begin{remark} \label{Remark:Easy-observations-on-variables}
Let us make the following observations:

\begin{enumerate}
\item
A  rational map $v \colon \p^2 \dasharrow \p^1$ is a variable if and only if there exists a rational map $w \colon \p^2 \dasharrow \p^1$ such that the rational map $(v,w) \colon  \p^2 \dasharrow \p^1 \times \p^1$ is birational.
\item
Writing a rational map $v \colon \p^2 \dasharrow \p^1$ as  $[x:y:z]\dasharrow [v_0(x,y,z):v_1(x,y,z)]$ where $v_0,v_1\in \k[x,y,z]$ are homogeneous of the same degree
$d$, then $v$ is a variable if and only if there exist 
an integer $D \geq d$ and homogeneous polynomials $h,v_2\in \k[x,y,z]$ of degrees $D-d$ and $D$, such that $[x:y:z]\dasharrow [hv_0 : h v_1 : v_2]$ is an element of $\Bir(\p^2)$.
\item
For each $p\in \p^2$, every projection $\pi_p\colon \p^2\dasharrow \p^1$ away from $p$ is a variable of degree $1$. Conversely, all variables of degree $1$ are obtained like this.
\item \label{natural-cations-on-variables}
The group $\Bir(\p^2)$ acts transitively on the set of variables by
right composition.
Similarly, $\Aut(\p^1)$ acts on the set of variables by left-composition.
\end{enumerate}
\end{remark}

\begin{definition}
Let $v\colon \p^2\dasharrow \p^1$ be a variable. We define the length of $v$, written $\lgth(v)$, to be the minimum of the lengths of the birational maps $\varphi\in \Bir(\p^2)$ such that $\Pi\circ \varphi=v$.
\end{definition}

\begin{lemma}  \label{Lem:VariableMapLength}
Let $v\colon \p^2\dasharrow \p^1$ be a variable. For each $\varphi\in \Bir(\p^2)$ such that $\Pi \circ \varphi=v$, we have 
\[\lgth(\varphi)\in \{\lgth(v),\lgth(v)+1\}.\]
\end{lemma}

\begin{proof}
By definition, there exists $\psi\in \Bir(\p^2)$ such that $\Pi \circ \psi=v$ and
$\lgth(\psi)= \lgth( v)$.
Since $\Pi \circ \varphi \circ (\psi)^{-1}=\Pi$, the map $\varphi\circ(\psi)^{-1}$ is a Jonqui\`eres transformation, which implies that the lengths of $\varphi$ and $\psi$ differ at most by one. As $\lgth(v)\le \lgth(\varphi)$ by definition, we get the result.
\end{proof}

\begin{lemma}\label{lemma:CompositionVariableEndoP1}
Let $v\colon \p^2\dasharrow \p^1$ be a variable, and let $\theta\colon \p^1\to \p^1$ be a morphism. Then, the following are equivalent:
\begin{enumerate}
\item\label{Thetav1}
The rational map $\theta\circ v\colon \p^2\dasharrow \p^1$ is a variable.
\item\label{Thetav2}
The morphism $\theta\colon \p^1\to \p^1$ is an automorphism.
\end{enumerate}
\end{lemma}

\begin{proof}

\ref{Thetav1}$\Rightarrow$\ref{Thetav2}:
If $\theta \circ v$ is a variable, its general scheme theoretic fibre is irreducible,
and hence the general scheme theoretic fibre of $\theta$ is irreducible as well.
This proves that $\theta$ has degree $0$ or $1$. As $\theta\circ v$ is non-constant, so is $\theta$, which is thus an automorphism.

\ref{Thetav2}$\Rightarrow$\ref{Thetav1}:  If $\theta$ is an automorphism of $\p^1$, we have already noted in Remark~\ref{Remark:Easy-observations-on-variables}\ref{natural-cations-on-variables} that $\theta \circ v$ is a variable.
\end{proof}

\begin{lemma} \label{lemma:NotDominantVariable}
Let $f\colon \p^2\dasharrow \p^2$ be a non-dominant and non-constant rational map and let $v\colon \p^2\dasharrow \p^1$ be a variable. Then, the following are equivalent:
\begin{enumerate}
\item\label{Fvar1}
There exists a morphism $\kappa\colon \p^1\to \p^2$ such that $\kappa\circ v=f$.
\item\label{Fvar2}
For each linear projection $\pi\colon \p^2\dasharrow \p^1$, there exists a morphism $\theta\colon \p^1\to \p^1$ such that 
$\pi\circ f=\theta\circ v.$
\item\label{Fvar3}
There exists a linear projection $\pi\colon \p^2\dasharrow \p^1$ and a non-constant morphism $\theta\colon \p^1\to \p^1$ such that 
$\pi\circ f=\theta\circ v.$
\end{enumerate}
\end{lemma}

\begin{proof}
As $f$ is non-constant, $\pi \circ f \colon \p^2 \dasharrow \p^1$ is a well-defined rational map for each linear projection $\pi\colon \p^2\dasharrow \p^1$.

$\ref{Fvar1}\Rightarrow \ref{Fvar2}$: For each linear projection $\pi\colon \p^2\dasharrow \p^1$ we get $\pi\circ f=(\pi\circ \kappa)\circ v$.

$\ref{Fvar2}\Rightarrow \ref{Fvar3}$: Since $f$ is not constant, there is a linear projection $\pi\colon \p^2\dasharrow \p^1$ such that $\pi\circ f$ is not constant.

$\ref{Fvar3}\Rightarrow \ref{Fvar1}$: Let us choose $\alpha\in \Aut(\p^2)$, $\beta\in \Aut(\p^1)$, and $g\in \Bir(\p^2)$ such that $\beta \circ \pi \circ \alpha = v\circ g=\Pi$. We then replace $(\pi,f,\theta,v)$ with $(\beta \circ \pi \circ \alpha,\alpha^{-1}\circ f\circ g,\beta\circ \theta,v\circ g)$ and can assume that $\pi =v= \Pi$.

We write locally the non-constant morphism $\theta\colon \p^1\to \p^1$ as 
$[1:t]\dasharrow [1:r(t)]$ for some $r(t)\in \k(t)\setminus \k$. The equation $\Pi \circ f=\theta\circ \Pi$ implies that $f$ is locally given as
\[[1:t:u]\dasharrow [1:r(t):s(t,u)]\]
for some rational function $s\in \k(t,u)$. As $f$ is not dominant, the two elements $s(t,u)$ and $r(t)$ are algebraically dependent over $\k$. Since $\k(t)$ is algebraically closed in $\k(t,u)$, this shows that $s\in \k(t)$. We can thus write $f$ as
\[[x:y:z]\dasharrow [f_0(x,y):f_1(x,y):f_2(x,y)]\] for some homogeneous polynomials $f_0,f_1,f_2\in \k[x,y]$, and can choose $\kappa\colon \p^1\to \p^2$ to be $[u:v]\mapsto [f_0(u,v):f_1(u,v):f_2(u,v)]$.
\end{proof}

\subsection{Definition of the Zariski topology on $\Bir(\p^2)$ and basic properties}\label{Subsec:ZarTop}

Following \cite{De,Se}, the notion of families of birational maps is defined, and used in Definition \ref{defi: Zariski topology} for describing the natural Zariski topology on $\Bir(X)$.

\begin{definition} \label{Defi:Family}
Let $A,X$ be irreducible algebraic varieties, and let $f$ be a $A$-birational map of the $A$-variety $A\times X$, inducing an isomorphism $U\to V$, where $U,V$ are open subsets of $A\times X$, whose projections on $A$ are surjective.

The rational map $f$ is given by $(a,x)\dasharrow (a,p_2(f(a,x)))$, where $p_2$ is the second projection, and for each $\k$-point $a\in A$, the birational map $x\dasharrow p_2(f(a,x))$ corresponds to an element  $f_a\in \Bir(X)$. The map $a\mapsto f_a$ represents a map from $A$ $($more precisely from the $\k$-points of $A)$ to $\Bir(X)$, and will be called a \emph{morphism} from $A$ to $\Bir(X)$.
\end{definition}

\begin{definition}  \label{defi: Zariski topology}
A subset $F\subseteq \Bir(X)$ is closed in the Zariski topology
if for any algebraic variety $A$ and any morphism $A\to \Bir(X)$ the preimage of $F$ is closed.
\end{definition}

If $d$ is a positive integer, we set $\b_d:= \{  f \in \b, \; \deg (f) \leq d \}$. We will use the following result, which is \cite[Proposition 2.10]{BF13}:

\begin{lemma}  \label{lemma:description-of-the-topology-of-Bir(P2)-as-an-inductive-limit}
A subset $F \subseteq \b$ is closed if and only if $F \cap \b_d$ is closed in $\b_d$ for any $d$.
\end{lemma}

The aim of this whole section~\ref{section:lower-semicontinuity-of-the-length} is to prove that for each non-negative integer $\ell$ the set
\[\b^{\ell}:= \{ f \in \b, \; \lgth (f) \leq \ell \}\]
is closed in $\b$. By Lemma~\ref{lemma:description-of-the-topology-of-Bir(P2)-as-an-inductive-limit}, this is equivalent to proving that $\b_d^{\ell}:= \b_d \cap \b^{\ell}$ is closed in $\b_d$ for any $d$. We will now describe the topology of $\b_d$. A convenient way to handle this topology is through the map $\pi_d \colon \bb_d \to \b_d$ that we introduce in the next definition and whose properties are given in Lemma~\ref{lemma:bridge-to-an-algebraic-variety} below.

Let us now fix the integer $d \geq 1$. We will constantly use the following piece of notation:

\begin{definition}
Denote by  $\rr_d$ the projective space associated with the vector space of triples $(f_0,f_1,f_2)$ where $f_0,f_1,f_2 \in \k [x,y,z]$ are homogeneous polynomials of degree $d$. The equivalence class of $(f_0,f_1,f_2)$ will be denoted by $[f_0:f_1:f_2]$.

For each $f=[f_0:f_1:f_2]\in \rr_d$, we denote by $\psi_f$ the rational map $\p^2\dasharrow \p^2$ defined by
\[  [x_0:x_1:x_2] \mapsto [f_0 (x_0,x_1,x_2): f_1 (x_0,x_1,x_2): f_2 (x_0,x_1,x_2) ].\]
Writing $\r$ the set of rational maps from $\p^2$ to $\p^2$ and setting $\r_d:=\{ h \in \r, \; \deg (h) \leq d \}$, we obtain a surjective map
\[ \Psi_d \colon \rr_d \to \r_d, \quad f \mapsto \psi_f.\]
This map induces a surjective map $\pi_d \colon \bb_d \to \b_d$, where $\bb_d$ is defined to be $\Psi_d^{-1} ( \b_d)$.
\end{definition}

\begin{remark}  \label{remark:psi-is-also-defined-over-a-field-extension-of-k}
For each field extension $\k\subseteq \k'$, we can similarly associate to each $f\in\rr_d(\k')$ a rational transformation
$\psi_f \colon \p^2_{\k'} \dasharrow \p^2_{\k'}$ defined
over $\k'$. This will be needed in the sequel to use a valuative
criterion.
\end{remark}

We will need the following result:

\begin{proposition} \label{proposition: B'_d-is-locally-closed-and-control-of-its-closure}
The set $\bb_d$ is locally closed in $\rr_d$ and thus inherits from $\rr_d$ the structure of an algebraic variety. Moreover, the following assertions hold:
\begin{enumerate}
\item  \label{ClosureNonDom}
For each $f \in \overline{\bb_d} \setminus \bb_d$, the rational map $\psi_f \colon \p^2 \dasharrow \p^2$ is non-dominant.
\item  \label{ExtensionStillBir}
For each field extension $\k\subseteq \k'$, the set $\bb_d(\k')$ of $\k'$-points of $\bb_d$ is equal to the set $\{f\in \rr_d(\k'),  \; \psi_f \colon \p^2_{\k'} \dasharrow \p^2_{\k'} \text{ is birational }\}$.
\end{enumerate}
\end{proposition}

\begin{proof}
Even if the first assertion is \cite[Lemma 2.4(2)]{BF13}, 
we recall the argument since it
will be used to prove the rest of the proposition. Denote by $ F \subseteq \rr_{d} \times \rr_d$ the closed algebraic variety corresponding to pairs $([g_0:g_1:g_2],[f_0:f_1:f_2])$ such that the ``formal composition'' \[g\circ f=[h_0:h_1:h_2]=[g_0(f_0,f_1,f_2):g_1(f_0,f_1,f_2):g_0(f_0,f_1,f_2)]\] is a ``multiple'' (maybe zero) of the identity. This corresponds to asking
that $h_0y=h_1x$, $h_0z=h_2x$, $h_1z=h_2y$. We then define $F_0\subseteq F$ to be the closed subset such that the formal composition is zero. If $(g,f ) \in F_0$, let us observe that the formal composition $g \circ f$ is zero, so that the rational map $\psi_f \colon \p^2 \dasharrow \p^2$ is non-dominant. Conversely, if $(g,f ) \in F \setminus F_0$, the formal composition $g \circ f$ is non-zero, so that the rational map $\psi_f \colon \p^2\dasharrow \p^2$ is birational.

The second projection $ \mathrm{pr_2} \colon \rr_d \times \rr_d \to \rr_d$  yields two closed subvarieties
\[ G_0= \mathrm{pr_2}(F_0) \subseteq G= \mathrm{pr_2}(F)\subseteq \rr_d.\]
By what has been said above, $\psi_f$ is non-dominant when $f \in G_0$ and birational when $f \in G \setminus G_0$. It follows that $(G \setminus G_0) \subseteq \bb_d$. Since $\deg (\varphi^{-1})  =\deg ( \varphi )$ for each $\varphi \in \b $ (Lemma~\ref{Lemm:EasyWeyl}\ref{L3}), we even get the equality $\bb_d = G \setminus G_0$. This shows that $\bb_d$ is locally closed in $\rr_d$, and also gives~\ref{ClosureNonDom}. To obtain \ref{ExtensionStillBir}, we observe that the construction made in the proof is defined over $\k'$, and that the inverse of any birational transformation of $\p^2$ defined over $\k'$ is still defined over $\k'$.
\end{proof}

The following result, which is \cite[Corollary 2.9]{BF13}, will be crucial for us since it provides us a bridge from the ``weird'' topological space $\b_d$ to the ``nice'' topological space $\bb_d$ which is an algebraic variety.

\begin{lemma} \label{lemma:bridge-to-an-algebraic-variety}
The map $\pi_d \colon \bb_d \to \b_d$ is continuous and closed. In particular, it is a quotient topological map: A subset $F \subseteq \b_d$ is closed if and only if its preimage $\pi_d^{-1}(F)$ is closed.
\end{lemma}

Recall that our aim is to prove that $\b_d^{\ell}= \{ f \in \b_d,  \; \lgth (f) \leq \ell \}$ is closed in $\b_d$. By Lemma~\ref{lemma:bridge-to-an-algebraic-variety}, this reduces to prove that $\bb_d^{\ell} := \pi_d^{-1} ( \b_d^{\ell} )$ is closed in (the algebraic variety) $\bb_d$.

We conclude this section by noting that the conjonction of Lemmas \ref{lemma:description-of-the-topology-of-Bir(P2)-as-an-inductive-limit} and \ref{lemma:bridge-to-an-algebraic-variety} gives us the following usefull characterisation (already contained in \cite[Corollary~2.7]{BF13})
of closed subsets of $\b$:

\begin{lemma} \label{lemma:useful-characterisation-of-closed-subsets-of-Bir(P2)}
A subset $F \subseteq \b$ is closed if and only if $\pi_d^{-1}(F \cap \b_d) \subseteq \bb_d$ is closed for any $d$.
\end{lemma}

\subsection{The use of a valuative criterion}

Let us set
\[ R := \k [[t]], \quad \text{and} \quad  K := \k ((t)), \]
where $\k [[t]]$ is the ring of formal power series and $\k ((t))$ its field of fractions (also known as formal Laurent series).

We will also write $\overline{K} = \bigcup_{a \ge 1} \k ((t^{1/a}))$ since this latter field is an algebraic closure of $K$ by Newton-Puiseux theorem \cite[Proposition~4.4]{Ruiz}.

\begin{definition}  \label{definition:definition-of-a(0)-when-a(t)-belongs-to-a-projective-space}
Let $n \ge 1$ be an integer and let $a=a(t) = [a_0 : \cdots : a_n] \in \p^n (K)$ be a $K$-point of the $n$-th projective space $\p^n$. Then, up to multiplying $(a_0, \ldots,a_n)$ with some power of $t$, we may assume that all coefficients $a_i$ belong to $R$ and that the evaluation $(a_0(0), \ldots,a_n(0) )$ at $t=0$ is nonzero. This enables us to define non ambigously the element $a(0) \in \p^n$, also denoted $\lim\limits_{t \to 0} a(t)$, by
\[  a(0) = \lim\limits_{t \to 0} a(t) := [a_0(0) : \cdots : a_n(0) ].\]
\end{definition}

\begin{remark}
More generally, if $X$ is a complete $\k$-variety and $x=x(t) \in X(K)$ is a $K$-point of $X$, one can define $x(0) = \lim\limits_{t \to 0} x(t) \in X$ in the following way: 
The morphism $x \colon \Spec (K) \to X$ admits a unique factorisation
of the form $x = \tilde{x} \circ \iota$ where $\iota \colon \Spec( K ) \into \Spec (R)$ is the open immersion induced by the natural injection $R \into K$ and where $\tilde{x} \colon \Spec (R) \to X$ is a $\k$-morphism (see the valuative criterion of properness given in \cite[(II, Theorem 4.7), page 101]{Hartshorne}).
\end{remark}

The following valuative criterion is classical, see e.g.~\cite[chap. 2, \S 1, pp 52-54]{MFK}. We refer to \cite{FurterPlanePolynomialAutomorphisms} for a proof in characteristic zero and to \cite{BlaConjugacyClasses} for a proof in any characteristic.

\begin{lemma} \label{lemma: valuative criterion}
Let $\varphi \colon X \to Y$ be a morphism between algebraic $\k$-varieties, $X$ being quasi-projective, and $Y$ being projective. Let $y_0$ be a $($closed$)$ point of $Y$.
Then, the two following assertions are equivalent:
\begin{enumerate}
\item We have $y_0 \in \overline{ \varphi (X)}$;
\item There exists a $K$-point $x= x(t) \in X(K)$ such that the $K$-point $y=y(t): = \varphi (x (t)) \in Y(K)$ satisfied $y_0= y (0)$.
\end{enumerate}
\end{lemma}

\begin{remark}
Lemma~\ref{lemma: valuative criterion} is analogue to the case of a continuous map $\varphi \colon X\to Y$ between metric spaces where a point $y_0$ of $Y$ belongs to $ \overline{  \varphi (X) }$ if and only if there exists a sequence $(x_n)_{n \geq 1}$ of $X$ such that ${\displaystyle y_0 = \lim_{n \to + \infty}  \varphi (x_n) }$.
\end{remark}

\begin{remark} \label{remark:f(0)-for-an-element-f-of-Bir(P2)_d(K)}
Applying Definition~\ref{definition:definition-of-a(0)-when-a(t)-belongs-to-a-projective-space} to an element $f=f(t) = [f_0:f_1:f_2]$ of $\rr_d(K)$ 
allows us to define $f(0) \in \rr_d$. If we assume furthermore that $f \in \bb_d(K) \subseteq \rr_d (K)$, note that $f(0)$ necessarily belongs to $\overline{\bb_d}$ by Lemma~\ref{lemma: valuative criterion}, so that $\psi_{f(0)} \colon \p^2 \to \p^2$ is either birational
or non-dominant by Proposition~\ref{proposition: B'_d-is-locally-closed-and-control-of-its-closure}.
Let us recall for clarity that for any $f \in \rr_d(K)$, we have defined a $K$-rational transformation $\psi_f \colon \p^2_K \dasharrow \p^2_K$ in Remark~\ref{remark:psi-is-also-defined-over-a-field-extension-of-k}, and that this transformation is moreover birational if we assume  that $f \in \bb_d(K)$ by Proposition~\ref{proposition: B'_d-is-locally-closed-and-control-of-its-closure}\ref{ExtensionStillBir}.
\end{remark}

We will prove that $\bb_d^{\ell}$ is closed in $\bb_d$. For this, we will prove that  its closure $\overline{\bb_d^{\ell} }$ in $\rr_d$ is such that $\overline{\bb_d^{\ell} } \cap \bb_d = \bb_d^{\ell} $. We begin with the following result which is just a (technical) application of the valuative
criterion
given above. If $\k'$ is an extension field of $\k$, $\Aut_{\k'} (\p^2) \simeq \PGL_3(\k')$, resp. $\Bir_{\k'}(\p^2)$, denotes the group of automorphisms, resp.~birational transformations, of $\p^2$ defined over $\k'$. Actually, we will only consider the cases where $\k' = K$ (since we will use the valuative
criterion
given in Lemma~\ref{lemma: valuative criterion}) and where $\k' =\overline{K}$ (since we need an algebraically closed field in order to apply the machinery about the length that we have 
developed).

\begin{proposition}   \label{proposition:application-of-the-valuative-criterion-to-our-case}
For any $ h \in \overline{\bb_d^{\ell} }$ there exists $f\in \bb_d(K)$ 
such that $h=f(0)$, and such that the birational map $\psi_f \in \Bir_{\overline{K}} (\p^2)$ associated to  $f \in \bb_d( \overline{K} ) $ has length at most $\ell$.
\end{proposition}

The proof of Proposition~\ref{proposition:application-of-the-valuative-criterion-to-our-case} relies on the two following lemmas:

\begin{lemma}  \label{Lemm:Jonqpclosed}
For each $p \in \p^2$, the set  $\JJonq_{p,d} := \Psi_d^{-1}(\Jonq_p \cap \b_d)$ is closed in $\bb_d$.
\end{lemma}

\begin{proof}
Up to applying an automorphism of $\p^2$,
we may assume that $p= [0:0:1]$.
Denote by  $\LL$ the projective space (of dimension 3) associated with the vector space of pairs $(g_0,g_1)$ where $g_0,g_1 \in \k [x,y]$ are homogeneous polynomials of degree $1$. The equivalence class of $(g_0,g_1)$ will be denoted by $[g_0:g_1]$. Denote by $Y\subseteq \bb_d \times \LL$ the closed subvariety given by elements $([f_0:f_1:f_2],[g_0:g_1])$ satisfying $f_0g_1=f_1g_0$. Since the first projection $\pr_1 \colon \bb_d \times \LL \to \bb_d$ is a closed morphism, the lemma follows from the equality $\JJonq_{p,d} = \pr_1 (Y)$.
\end{proof}

\begin{remark}
Lemma~\ref{Lemm:Jonqpclosed} asserts that $\JJonq_{p,d} =\pi_d^{-1}(\Jonq_{p} \cap \b_d )$ is closed in $\bb_d$ for each $d$. By Lemma~\ref{lemma:useful-characterisation-of-closed-subsets-of-Bir(P2)}, this means that $\Jonq_{p}$ is closed in $\Bir(\p^2)$.
\end{remark}

\begin{lemma}  \label{lemma:bound-on-the-degrees-of-the-Jonquieres}
Any Cremona transformation $g \in \b $ of length $\ell$ admits an expression of the form
\[  g = a_1 \circ \varphi_1  \circ \cdots \circ a_\ell \circ \varphi_\ell  \circ a_{\ell+1},\]
where $a_1,\dots,a_{\ell+1}\in \Aut (\p^2)$, $\varphi_1,\dots,\varphi_\ell \in \Jonq_p$,
and $ \deg ( \varphi_i ) \leq \deg (g) $ for each $i$.
\end{lemma}

\begin{proof}
This follows from Theorem~\ref{TheMainTheorem} and the fact that if $\varphi$ is a Jonqui\`eres transformation such that $g \circ \varphi$ is a predecessor of $g$, then $\Base(\varphi^{-1})\subseteq \Base(g)$ (Lemma~\ref{Lem:Pred}\ref{Lem:Pred3}), which implies that $g$ hast at least $2\deg(\varphi)-1$ base-points, so $\deg(g)\ge \deg(\varphi)$ (every element of $\Bir(\p^2)$ of degree $d\ge 2$ has at most $2d-1$ points, and equality holds if and only if the map is Jonqui\`eres \cite[Lemma 13]{BCM15}).
\end{proof}

\begin{proof}[Proof of Proposition~$\ref{proposition:application-of-the-valuative-criterion-to-our-case}$]

In order to use Lemma~\ref{lemma: valuative criterion}, we  realise $\bb_d^{\ell}$ as the image of a morphism of algebraic varieties. Let us fix $p=[0:0:1]\in \p^2$.  By Lemma~\ref{lemma:bound-on-the-degrees-of-the-Jonquieres},
an element $f$ of $\bb_d$ belongs to $\bb_d^{\ell}$ if and only if the birational transformation $\psi_f$ admits an expression of the form
\[ \psi_f = a_1 \circ \varphi_1  \circ \cdots \circ a_\ell \circ \varphi_\ell  \circ a_{\ell+1},\]
where $a_1,\dots,a_{\ell+1}\in \Aut (\p^2)$, $\varphi_1,\dots,\varphi_\ell \in \Jonq_p$,
and $ \deg \varphi_i \leq d$ for each $i$.

We now use the closed subvariety $\JJonq_{p,d} \subseteq \bb_d$ given in Lemma~\ref{Lemm:Jonqpclosed}.
Define the product $\PP := ( \PGL_3)^{\ell+1} \times (\JJonq_{p,d})^\ell $ and let $\Comp \colon \PP  \to \bb_{d^\ell}$ be the formal composition morphism defined by
\[ (a_i)_{1 \, \leq \, i \,  \leq  \, \ell+1} \times (\varphi_i)_{1 \, \leq \, i \,  \leq  \, \ell}   \quad  \mapsto  \quad a_1 \circ \varphi_1 \circ \cdots \circ a_\ell \circ \varphi_\ell  \circ a_{\ell+1} .\]

Let $\Delta \subseteq \bb_d \times \bb_{d^\ell}$ be the pseudo-diagonal, i.e.~the set of pairs $(f,g)$ such that $\psi_f = \psi_g$. Being given by the equations $f_i g_j = f_j g_i$, for all $i,j$, where $ f = [ f_0 : f_1 : f_2 ]  \in \bb_d$, $ g = [ g_0 : g_1 : g_2 ]  \in \bb_{d^\ell}$, the set $\Delta$ is closed into $\bb_d \times \bb_{d^\ell}$.

Denote by $ \id   \times \Comp \colon \bb_d \times \PP \to \bb_d \times \bb_{d^\ell}$, the morphism sending $(f,p)$ to $(f, \Comp (p) )$ and by $\Delta'$ the closed variety defined by
\[ \Delta ' := (\id   \times \Comp)^{-1} (\Delta) \subseteq \bb_d \times \PP. \]
By what has been said above, we have $\bb_d^{\ell}= \pr_1 (\Delta') $ where $\pr_1 \colon \bb_d \times \PP \to \bb_d$ is the first projection. Setting $\varphi = \iota \circ \pr_1$ where $\iota \colon \bb_d \into \rr_d$ is the natural injection, we also have $\bb_d^{\ell}= \varphi (\Delta') $.

Since $h \in \overline{ \varphi (\Delta') } $,  Lemma~\ref{lemma: valuative criterion} yields us the existence of an element $ (f,p) = (f(t), p(t) ) \in  \Delta' ( K ) \subseteq \bb_d(K)  \times \PP(K)$ such that $h= f(0)$. We observe that the birational map $\psi_f \in \Bir_{\overline{K}} (\p^2)$ associated to  $f \in \bb_d( \overline{K} ) $ has length at most $\ell$.
\end{proof}

\subsection{The end of the proof of Theorem \ref{Continuitylength}}

The main result of the previous section (Proposition~\ref{proposition:application-of-the-valuative-criterion-to-our-case}) asserts that any element $h \in \overline{\bb_d^{\ell}}$ is equal to $f(0)$ for a certain element $f \in \bb_d (K)$ such that the length of $\psi_f \in \Bir_{\overline{K}} (\p^2)$ is at most $\ell$. The main technical result of the present section is Proposition~\ref{proposition: the heart of the proof}
which establishes that the limit $\psi_{f(0)} \colon \p^2 \dasharrow \p^2$ is either birational of length $\leq \ell$ or non-dominant (however, in the non-dominant case, we need to prove a stronger statement in order to make an induction).
 This information is sufficient for showing that $\bb_d^{\ell}$ is closed in $\bb_d$ thus proving Theorem~\ref{Continuitylength}. We begin with the following simple lemma to be used in the proof of Proposition~\ref{proposition: the heart of the proof}.

\begin{lemma} \label{lemma: linear algebra over the field k((t))}
Let $V$ be a finite dimensional vector space over $k$ and let $u,v \in K \otimes_k V  $ be two vectors such that 
\begin{enumerate}
\item The vectors $u,v$ are linearly independent over $K$;
\item The vector $v$ belongs to  $R \otimes_k V$ and its evaluation $v(0)$ at $t = 0$ is nonzero.
\end{enumerate}
Then, there exist $\alpha, \beta \in K$ such that:
\begin{enumerate}
\item The vector $\tilde{u} := \alpha u + \beta v$ belongs to $R \otimes_k V$;
\item The vectors $\tilde{u}(0)$ and $v(0)$ are linearly independent over $k$.
\end{enumerate}
\end{lemma}

\begin{proof}
Let us complete the vector $e_1:= v(0)$ in a basis $e_1, \ldots, e_n$ of $V$. Decomposing the vectors $u,v$ in this basis, we obtain expressions
\[ u = \sum_i u_i e_i, \quad v = \sum_i v_i e_i, \]
where $u_1,\dots,u_n \in K$, $v_1,\dots,v_n\in R$ and $v_1(0)=1, v_i(0)=0$ for
$i=2,\dots,n$. The vector $w:= u - \frac{u_1}{v_1} v$ is nonzero and admits an expression
\[ w = \sum_i w_i e_i, \]
where $w_1,\dots,w_n\in K$ and $w_1=0$. Let $j$ be the unique integer such that the vector $\tilde{w} := t^jw$ belongs to $R \otimes_k V$ and its evaluation  $\tilde{w}(0)$ at $t=0$ is nonzero. The vectors $\tilde{w}(0)$ and $v(0) =e_1$ are linearly independent over $\k$ because $\tilde{w}(0)\in (\k e_2\oplus\cdots \oplus\k e_n)\setminus \{0\}$. Since $\tilde{w} = t^j (u - \frac{u_1}{v_1} v)$, it is enough to set $\alpha = t^j$ and $\beta = - t^j \frac{u_1}{v_1}$.
\end{proof}

\begin{proposition} \label{proposition: the heart of the proof}
Let $f = f(t) \in \bb_d(K) \subseteq \bb_d ( \overline{K} )$ be an element such that the associated birational map $\psi_f \in \Bir_{\overline{K}} (\p^2)$ has length $\ell \ge 0$
and denote by $f(0) \in \overline{\bb_d} \subseteq  \rr_d$ the evaluation of $f=f(t)$ at $t=0$ 
$($see Definition~$\ref{definition:definition-of-a(0)-when-a(t)-belongs-to-a-projective-space}$ and Remark~$\ref{remark:f(0)-for-an-element-f-of-Bir(P2)_d(K)})$. Then, the following implications hold:
\begin{enumerate}
\item
If $f(0)\in \bb_d$, then the birational map $\psi_{f(0)}$ is of length $\le \ell $.
\item
If $f(0)\in \overline{\bb_d} \setminus \bb_d$, then the rational map $\psi_{f(0)}$ is equal to $\kappa \circ v$ for some variable $v \colon \p^2\dasharrow \p^1$ of length $\le \ell$ and some morphism $\kappa\colon \p^1\to \p^2$.
\end{enumerate}
In particular, for each linear projection $\pi\colon \p^2\dasharrow \p^1$, the composition $\pi\circ \psi_{f(0)}\colon \p^2\dasharrow \p^1$ is either not defined or equal to $\rho\circ v$ for some variable $v \colon \p^2\dasharrow \p^1$ of length $\le \ell$ and some endomorphism $\rho \colon \p^1\to \p^1$.\end{proposition}

\begin{proof}
We prove the result by induction on $\ell$. 

\smallskip

{\bf Case of length $\ell=0$.} The equality $ \ell =0$ corresponds exactly to
asking that $\psi_f\in \Aut_K(\p^2)$. We write $f=[h a_0:h a_1 :h a_2]$, where $h\in K[x,y,z]$ is homogeneous of degree $d-1$ and $[a_0:a_1:a_2] \in \bb_1(K) \simeq \PGL_3(K)$. We can moreover assume that the coefficients of $h$ belong to $R \subseteq K$ and that the evaluation $h(0)$ of $h$ at $t=0$ is non-zero. Similarly, we can assume that $a_0,a_1,a_2$ have coefficients in $R$ and that at least one of these has a non-zero value at $t=0$. The element $[a_0(0):a_1(0):a_2(0)]\in \rr_1$ corresponds to a $3\times 3$-matrix. If the matrix is of rank $3$, the element $f(0) \in \bb_d$ corresponds to a linear automorphism $\psi_{f(0)} \in  \b$ of length $0$. If the matrix is of rank $2$, then $\psi_{f(0)}$
admits a decomposition
in the form $\kappa \circ v$ where $\kappa \colon \p^1 \to \p^2$ is a linear morphism and $v \colon \p^2 \dasharrow \p^1$ is a linear variable, i.e.~of degree
$1$. The last case is when the matrix has rank $1$, which corresponds to the case where
$\psi_{f(0)} \colon \p^2 \to \p^2$ is the constant map to some point $a \in \p^2$. Let $\kappa \colon \p^1 \to \p^2$ be the  constant map to $a$ and let $v$ be any variable of length $0$, then we have $\psi_{f(0)} = \kappa \circ v$.

\smallskip

{\bf Case of length $ \ell \geq 1$.}

This implies that the birational transformation $(\psi_f)^{-1} \in \Bir_{\overline{K} }(\p^2)$ admits at least one base-point. The proof is divided into the following steps:

{\bf Reduction to the case where all base-points are defined over $K$:}

By assumption, the birational map $\psi_{f}\in \Bir_{\overline{K} }(\p^2)$ has length $\ell $. Replacing $t$ with $t^{\frac{1}{a}}$ for some $a \ge 1$, we can thus assume that all base-points of $(\psi_f)^{-1}$ are defined over $K$.

{\bf Denote by $p\in \p^2(K)$ a base-point of maximal multiplicity of $(\psi_f)^{-1}$.}

{\bf Reduction to the case where $p= [0:0:1]$.} 

Write $p= [ p_0: p_1 : p_2]$ where each $p_i$ belongs to $K$. Up to multiplying $(p_0, p_1, p_2)$ by  $t^i$ for some well chosen (and unique) integer $i$, we may assume that $p_0,p_1,p_2\in R$ and that $p_i(0)\not=0$ for some $i$.
Let us choose coefficients $b_{ij}$ in the field $k$ such that the following matrix has nonzero determinant:
\[ M=\left(\begin{array}{lll} b_{00} & b_{01} &p_0(0)  \\ b_{10} & b_{11}&p_1(0)  \\ b_{20} & b_{21}& p_2(0)  \end{array}\right). \]
In other words, we have $M \in \GL_3 (\k)$. This implies that the matrix
\[ B(t)= \left(\begin{array}{lll}b_{00} & b_{01} &p_0  \\ b_{10} & b_{11}&p_1  \\ b_{20} & b_{21}& p_2 \end{array}\right)  \in \M_3 ( R) \]
is invertible in $\M_3( R )$  (because its determinant is invertible). The evaluation at $t=0$ of the corresponding automorphism  $\beta\in \Aut_K( \p^2) = \PGL_3(K)$ 
is the element of $\PGL_3(\k)=\GL_3(\k)/\k^*$ given by the class of the matrix $M\in \GL_3(\k)$.
We can replace $f$ with $\tilde{f}=\beta^{-1} \circ f  \in B'_d(K)$ (formal composition), because we have $\tilde{f}(0)=  \beta(0)^{-1} \circ f(0)$, where $\beta(0)$ belongs to $\PGL_3(\k)$. After this change, the point $p$ is equal to $[0:0:1] \in \p^2 \subseteq \p^2(K)$.

As in Definition~\ref{definition:definition-of-a(0)-when-a(t)-belongs-to-a-projective-space} (see also Remark~\ref{remark:f(0)-for-an-element-f-of-Bir(P2)_d(K)}), we write $f=[f_0:f_1:f_2]$ where the components  $f_i \in R [x,y,z]$ satisfy  $(f_0(0),f_1(0),f_2(0))\not=(0,0,0)$.

If $(f_0(0),f_1(0))=(0,0)$, then $\psi_{f(0)}$ is the constant map to $p$: Hence we have $f(0)\not\in \bb_d$ and the result is trivially true by taking $\kappa \colon \p^1\to \p^2$ the constant map to $p$ and $v \colon \p^2 \dasharrow \p^1$ any variable of length $\leq \ell $. We can thus assume that $(f_0(0),f_1(0))\not=(0,0)$, which means that $\psi_{f(0)}$ is not the constant map to $p$, and can consider the rational map $\Pi \circ \psi_{f(0)}\colon \p^2\dasharrow \p^1$, given by $[x:y:z]\dasharrow [f_0(0)(x,y,z):f_1(0)(x,y,z)]$. We achieve the proof by studying two cases, depending on whether this rational map is constant or not.

{\bf Case A: The rational map $\Pi \circ \psi_{f(0)}$ is not  constant --  construction of an element of length $\ell -1$.}

Since $p=[0:0:1]$ is a base-point  of maximal multiplicity of $(\psi_f)^{-1} \in \Bir_{\overline{K} } (\p^2)$, by Corollary~\ref{Cor:decrdegrlgth} there exists an element $\varphi \in \Jonq_p( \overline{K} )$ (i.e.~an element of $\Bir_{\overline{K} } (\p^2)$ which preserves the pencil of lines through $p$) such that $\varphi  \circ \psi_f  \in \Bir_{\overline{K} }(\p^2)$ is of length $\ell -1$ and of degree smaller than $\deg(\psi_f)\le d$.

We  may moreover assume that $\varphi$ satisfies the two following assertions:
\begin{enumerate}[(i)]
\item \label{phi-defined-over-K}
$\varphi$ is defined over $K$, i.e.~$\varphi \in \Jonq_p(K)$. 
\item  \label{Pi-o-phi=Pi}
$\varphi$ preserves a general line through $p$, i.e.~$\Pi \circ \varphi = \Pi$.
\end{enumerate}

To obtain \ref{phi-defined-over-K}, we could
use the fact that all base-points of $\varphi$ are defined over $K$ (since they are
contained in the base locus of $(\psi_f)^{-1}$). Alternatively, we can use the same trick as above: Since $\varphi$ is defined over $\k ((t^{1/a}))$ for some integer $a \geq 1$, it is enough to replace $t$ with $t^{1/a}$.

To obtain \ref{Pi-o-phi=Pi}, it is enough to note that any element $ \varphi \in \Jonq_p$ may be written as a composition $\alpha \circ \tilde\varphi$ where $\alpha \in \Aut (\p^2) \cap \Jonq_p$ and $\tilde \varphi \in \Jonq_p$ preserves a general line through $p$.

Let $g\in \bb_d(K)$ be such that $\psi_g=  \varphi  \circ  \psi_f $. Note that the assumption \ref{Pi-o-phi=Pi} above shows us that $\Pi \circ \psi_f = \Pi \circ \psi_g$. As before, write $g=[g_0:g_1:g_2]$ where the components  $g_i \in R [x,y,z]$ satisfy  $(g_0(0),g_1(0),g_2(0))\not=(0,0,0)$.
The fact that $\psi_{f(0)}$ is not the constant map to $p$ corresponds exactly to saying that $(f_0(0),f_1(0))\not=(0,0)$. Replacing $\varphi$ with its composition with $[x:y:z]\mapsto [t^{-i}x:t^{-i}y:z]$ for some well chosen integer
$i\ge 0$ we may replace
$(g_0,g_1,g_2)$ with $(t^{-i}g_0,t^{-i}g_1,g_2)$ and then assume that $(g_0(0),g_1(0))\not=(0,0)$. We obtain then a rational map 
\[\begin{array}{rrcl}
\nu\colon &\p^2&\dasharrow &\p^1\\
& [x:y:z]& \mapsto &  [f_0(0)(x,y,z):f_1(0)(x,y,z)]=[g_0(0)(x,y,z):g_1(0)(x,y,z)]\end{array}\]
which satisfies $\nu=\Pi \circ \psi_{f(0)}=\Pi \circ \psi_{g(0)}$ and is thus non-constant by hypothesis.

Applying the induction hypothesis to $g$ (and $g(0)$), the map $\nu=\Pi \circ \psi_{ g(0) } = \Pi \circ \psi_{ f(0) }$ is equal to $\theta \circ v$, where $v\colon \p^2\dasharrow \p^1$ is a variable of length at most $ \ell -1$ and $\theta\colon \p^1\to \p^1$ is an endomorphism.
Moreover, $\theta$ is non-constant since $\nu$ is non-constant.
 
a) If $\psi_{f(0)}$ is a birational map, then $\nu = \Pi \circ \psi_{f(0)}$ is a variable. Since $\nu=\theta \circ  v$, this implies that $\theta \in \Aut(\p^1)$ (Lemma~\ref{lemma:CompositionVariableEndoP1}) and thus that $\lgth(\nu)=\lgth(v)\le \ell -1$. In particular, $\lgth(\psi_{f(0)})\le \ell $ since $\nu=\Pi \circ \psi_{f(0)}$ (Lemma~\ref{Lem:VariableMapLength}) as we wanted.
 
b) If $\psi_{f(0)}$ is not a birational map, then it is non-dominant (see Remark~\ref{remark:f(0)-for-an-element-f-of-Bir(P2)_d(K)}). The equality $\nu = \Pi \circ \psi_{f(0)}=\theta \circ v$ yields the existence of a morphism $\kappa \colon \p^1\to \p^2$ such that $\kappa \circ v=\psi_{f(0)}$ (Lemma~\ref{lemma:NotDominantVariable}). This achieves the proof in this case.\\

{\bf Case B: The rational map $\Pi \circ \psi_{f(0)}$ is constant.}

Let $[\lambda : \mu] \in \p^1$ be the constant value of the map $\Pi \circ \psi_{f(0)}$. There exists a homogeneous polynomial $h \in \k[x,y,z]$ such that $(f_1(0),f_2(0))=(\lambda h, \mu h)$. Up to replacing $f$ with $\alpha \circ f$, where $\alpha \in \Aut (\p^2)$ is of the form $[x:y:z]\mapsto [ax+by:cy+dy:z]$, we can assume that $(\lambda, \mu) =(0,1)$, which implies that $f_0(0)=0$ and $f_1(0) \neq 0$.

In this case, we have $f(0)\notin \bb_d$ and $\psi_{f(0)}$ is the rational map $[x:y:z] \dasharrow [0:f_1(0)(x,y,z):f_2(0)(x,y,z)]$. Writing $\pi \colon \p^2\dasharrow \p^1, [x:y:z] \mapsto [y:z]$ the projection away from $[1:0:0]$, it remains to see that the rational map $\pi\circ \psi_{f(0)}\colon \p^2\dasharrow \p^1$, $[x:y:z] \dasharrow [f_1(0)(x,y,z):f_2(0)(x,y,z)]$ is the composition of a variable $\p^2 \dasharrow \p^1$ of length $\le \ell$ and an endomorphism of $\p^1$.

To show this,  let us note that by Lemma~\ref{lemma: linear algebra over the field k((t))}, there exist $\alpha, \beta\in K$ such that $\tilde{f_0}=\alpha f_0+\beta f_1\in R [x,y,z]$ and $\tilde{f_0}(0)$ does not belong to $\k\cdot f_1(0)$.

We observe that the result holds for $\tilde{f} := [\tilde{f_0} : f_1 : f_2] \in \bb_d(K)$. Indeed, $\psi_{\tilde{f}}$ and $\psi_{f}$ only differ by an element of $\Aut(\p^2)(K)$ that fixes $p=[0:0:1]$, so $\psi_{\tilde{f}}$ and $\psi_{f}$ have the same length, $p$ is a base-point of $\psi_{\tilde{f}}^{-1}$ of maximal multiplicity, and all base-points of $\psi_{\tilde{f}}^{-1}$ are defined over $K$. Moreover, $\tilde{f}$ satisfies Case~A. The result then holds for $f$, since $\pi\circ \psi_{\tilde{f}(0)}=\pi\circ \psi_{f(0)}$.
\end{proof}

\begin{proof}[Proof of Theorem~$\ref{Continuitylength}$]
We have already explained why Proposition~\ref{proposition: the heart of the proof} implies Theorem~\ref{Continuitylength}. Let us however summarise the proof. We want to show that $\b ^{\ell} =\{f\in \b, \; \lgth(f)\le \ell\}$ is closed in $\b$ for each integer $\ell \ge 0$. By Lemma~\ref{lemma:useful-characterisation-of-closed-subsets-of-Bir(P2)}, this is equivalent to saying that $\bb_d^{\ell} =\pi_d^{-1}( \b_d^{\ell}  )$ is closed in $\bb_d$ for each $d$. This latter point directly follows from Propositions  \ref{proposition:application-of-the-valuative-criterion-to-our-case} and \ref{proposition: the heart of the proof}.
\end{proof}

\end{document}